\documentclass[final]{siamltex}

\pdfoutput=1

\usepackage[latin1]{inputenc}
\usepackage{array}
\usepackage{rotating}
\usepackage{graphicx,cancel}
\usepackage{epsfig,hyperref}
\usepackage{epstopdf}
\usepackage{mystyle,ulem}
\usepackage{color}
\usepackage{vmargin}
\setpapersize{A4} 

\usepackage{color}

\author{Nicolas Papadakis\thanks{IMB, Universit\'e Bordeaux 1, 351, cours de la Lib\'eration, F-33405 TALENCE, FRANCE.}
\and Gabriel Peyr\'e\thanks{CEREMADE, Universit\'e Paris-Dauphine, Place du Marechal De Lattre De Tassigny, 75775 PARIS CEDEX 16, FRANCE.}
\and Edouard Oudet\thanks{LJK, Universit\'e de Grenoble, 51 rue des Math\'ematiques, Campus de Saint Martin d'H\`eres, BP 53, 38041 GRENOBLE CEDEX 09}
}

\graphicspath{{./images/}}

\title{Optimal Transport with  Proximal Splitting}

\begin{document}

\maketitle

\begin{abstract}
	This article reviews the use of first order convex optimization schemes to solve the discretized dynamic optimal transport problem, initially proposed by Benamou and Brenier. We develop a staggered grid discretization that is well adapted to the computation of the $L^2$ optimal transport geodesic between distributions defined on a uniform spatial grid. We show how proximal splitting schemes can be used to solve the resulting large scale convex optimization problem. A specific instantiation of this method on a centered grid corresponds to the initial algorithm developed by Benamou and Brenier. We also show how more general cost functions can be taken into account and how to extend the method to perform optimal transport on a Riemannian manifold. 
\end{abstract}

\section{Introduction}

Optimal transport is a well developed mathematical theory that defines a family of metrics between probability distributions~\cite{Villani03}. These metrics measure the amplitude of an optimal displacement according to a so-called ground cost defined on the space supporting the distributions. The resulting distance is sometimes referred to as the Wasserstein distance in the case of $L^p$ ground costs. The geometric nature of optimal transportation, as well as the ability to compute optimal displacements between densities, make this theory progressively mainstream in several applicative fields (see bellow). However, the numerical resolution of the optimal transportation problem raises several challenges. This article is focused on the  computation of geodesics for the optimal transport metric associated to the $L^2$ cost. It reviews and extends the approach pioneered by Benamou and Brenier~\cite{Benamou2000} from the perspective of proximal operator splitting in convex optimization. This shows the simplicity and efficiency of this method, which can easily be extended beyond the setting of optimal transport by considering various convex cost functions.

\subsection{Previous Works} 
\label{sec-previous}

\paragraph{Applications of optimal transport}

Early successes of applications of optimal transport were mostly theoretical, such as for instance the derivation of functional inequalities~\cite{CorderoErausquin-sobolev} or the construction of solutions of some non-linear partial differential equations~\cite{jko}. But recently, computational optimal transport as gained much interest and is progressively becoming mainstream in several applicative fields. In computer vision, the Wasserstein distance has been shown to outperform other metrics between distributions for  machine learning tasks~\cite{Rubner1998,Pele-ICCV}. In image processing, the warping provided by the optimal transport has been used for video restoration~\cite{Delon-midway}, color transfer~\cite{Papadakis-colors}, texture synthesis~\cite{2013-ssvm-mixing} and medical imaging registration~\cite{haker-ijcv}. It has also been applied to interpolation in computer graphics~\cite{Bonneel-displacement} and surface reconstruction in computational geometry~\cite{digne-reconstruction}. Optimal transport is also used to model various physical phenomena, such as for instance in astrophysics \cite{Frisch-nature} and oceanography~\cite{Benamou-Semigeostrophic}.

\paragraph{Discrete optimal transport}

The easiest way to discretize and compute numerically optimal transports is to consider finite sums of weighted Diracs. In this specific case, the optimal transport is a multi-valued map between the Diracs locations.   Specific linear solvers can be used in this context and in particular network and transportation simplices~\cite{Dantzig-Book} can scale up to a few thousands of Dirac masses. Dedicated combinatorial optimization methods have been proposed, such as the auction algorithm~\cite{Bertsekas1988}, which can handle integer costs between the Diracs. In the even more restricted case of two distributions with the same number of Diracs with equal weights, the transportation is a bijection between the points and thus corresponds to the optimal assignment problem~\cite{Burkard09}. Combinatorial optimization methods such as the Hungarian algorithm~\cite{Kuhn-hungarian} have roughly cubic complexity in the number of Diracs. Faster schemes exist for specific cost functions, such as for instance convex cost of the distance on the line (where it boils down to a sorting of the positions) and the circle~\cite{delon-circle}, concave costs on the line~\cite{delon-concave}, the $\ell^1$ distance~\cite{Ling2007}. The computation can be accelerated using multi-scale clustering~\cite{Merigot2011}. Note also that various approximations of the transportation distance have been proposed, see for instance~\cite{Shirdhonkar2008}.

Despite being numerically intensive for finely discretized distributions, this discrete transport framework has found many applications, such as for instance color transfer between images~\cite{Rabin2011}, shape retrieval~\cite{Rubner1998}, surface reconstruction~\cite{deGoes2011} and interpolation for computer graphics~\cite{Bonneel-displacement}

\medskip
\paragraph{Optimal transport and PDE's}

The optimal transport for the $L^2$ ground cost has a special structure. It can be shown to be uniquely defined and to be the gradient of a convex function~\cite{Brenier1991}. This implies that it is also the solution of the fully non-linear Monge-Amp\`ere partial differential equation.
Several methods have been proposed to discretize and solve this PDE, such as for instance the method of~\cite{Oliker-Prussner-1988} which converges to the Aleksandrov solution and the one of~\cite{Oberman-2008} which converges to the viscosity solution of the equation. Alternative methods such as~\cite{Dean-Glowinski-2006} and~\cite{Feng-Neilan-2009} are efficient for regular densities.  A major difficulty in these approaches is to deal with compactly supported densities, which requires a careful handing of the boundary conditions. \cite{Froese2012} proposes to enforce these conditions by iteratively solving a Monge-Ampere equation with Neumann boundary conditions. \cite{Benamou2012} introduces a method requiring the solution of a well-posed Hamilton-Jacobi equation.
Another line of methods iteratively constructs mass preserving mappings converging to the optimal transport~\cite{Angenent2003}. This explicitly constructs the so-called polar factorization of the initial map, see also~\cite{Benamou1995} for a different approach. This method is enhanced in~\cite{Haber2010} to avoid drifting from the preservation constraint during the iterations.
These PDE's based approaches to the resolution of the optimal transport have found several applications,  such as image registration~\cite{haker-ijcv}, density regularization~\cite{Burger2011}, optical flow~\cite{Clarysse2010} and grid generation~\cite{Sulman2011b}.

Another line of research consists in using gradient flows where the gradient direction is computed according to the Wasserstein distance. This was initially proposed in~\cite{jko} to build solutions to certain non-linear PDE's. This technique is now being used to design numerical approximation schemes for the solution of these equations, see for instance~\cite{Carrillo-diffusive,During-gradient-flow,Ferragut-mixed-fem}.

\medskip
\paragraph{Dynamical optimal transport}

Instead of computing directly the transport, it is possible to consider the geodesic path between the two densities according to the Wasserstein metric (the so-called displacement interpolation~\cite{mccann1997convexity}). For the $L^2$ ground cost, this geodesic is obtained by linear interpolation between the identity and the transport. The geodesic can thus be computed by first obtaining the transport and then evolving the densities. If one considers discrete sums of Diracs, this corresponds to solving a convex linear program and can also be understood as a Lagrangian approximation of the transport between (possibly continuous) densities that have been discretized. This approach is refined in~\cite{Iollo2011}, which considers discretization with mixture of Gaussians.

It is also possible to consider an Eulerian formulation of the geodesic problem, for which densities along the path are discretized on a fixed spatial grid. Conservation of mass is achieved by introducing an incompressible velocity field transporting the densities. The breakthrough paper~\cite{Benamou2000} shows that it is possible to perform a change of variable to obtain a convex problem. They propose to solve numerically the discretized problem with a first order iterative method. We give further details bellow on this method in the paragraph on proximal methods.  

Geodesics between pairs of distributions can be extended to barycenters between an arbitrary finite collection of distributions. Existence and uniqueness of this barycenter is studied in~\cite{Carlier_wasserstein_barycenter}. Computing the barycenter between discrete distributions requires the resolution of a convex linear program that corresponds to a multi-marginal optimal transportation, as proved in~\cite{Carlier_wasserstein_barycenter}. However, in sharp contrast with the case of two distributions, the special case of un-weighted sums of Diracs is not anymore equivalent to an assignment problem, which is known to be NP-hard~\cite{Burkard09}. Computing numerically this barycenter for large scale problems can however be tackled using a non-convex formulation to solve for a Lagrangian discretization, which finds applications in image processing~\cite{Rabin_ssvm11}.

\medskip
\paragraph{Generalized transport problems}

The formulation of the geodesic computation as a convex optimization problem initiated by~\cite{Benamou2000} enables the definition of various metrics obtained by changing the objective function. A penalization of the matching constraint~\cite{Benamou2010} allows one to compute an unbalanced transport where densities are not normalized to have the same mass. An interpolation between the $L^2$-Wasserstein and $L^2$ distances is proposed in~\cite{Benamou2001}. Lastly, an interpolation between $L^2$-Wasserstein and $H^{-1}$ distances is described in~\cite{dolbeault2009}. This extension relies in a crucial manner on the convexity of the extended objective function, which enables a theoretical analysis to characterize minimizing geodesics~\cite{Cardaliaguet2012}. Convexity also allows one to use the numerical scheme we propose with only slight modifications with respect to the $L^2$-Wasserstein case, as we detail in Section~\ref{sec-generalized}. 

\medskip
\paragraph{Optimal transport on Riemannian manifolds}

Many properties of the $L^2$-Was\-ser\-stein distance extend to the setting where the ground cost is the square of the geodesic distance on a Riemannian manifold. This includes in particular the existence and uniqueness of the transport map, which is the manifold exponential of the gradient of a semi-convex map~\cite{McCann-PolarManifold}. Displacement interpolation for transport on manifolds has the same variational characterization as the one introduced in~\cite{Benamou2001} for Euclidean transport, see~\cite{Villani-OldNew} for a detailed review of optimal transport on manifolds. Interpolation between pairs of measures generalizes to barycenters of a family of measures, see~\cite{KimPass-MultiMarg-Manifold}.

Displacement interpolation between two measures composed, each one composed of a single Dirac, amounts to computing a single geodesic curve on the manifold. Discretization and numerical solutions to this problem are numerous. A popular method is the Fast Marching algorithm introduced jointly by~\cite{sethianFM1995,tsitsiklis-TAC} for isotropic Riemannian metrics (i.e. when the metric at each point is a scalar multiple of the identity) discretized on a rectangular grid. The complexity of the method is $O(N \log(N))$ operations, where $N$ is the number of grid points. This algorithm has been extended to compute geodesics on 2-D triangular meshes with only acute angles~\cite{sethian-geodesic-path}. More general discretizations and the extension to Finsler metrics require the use of slower iterative schemes, see for instance~\cite{bornemann-fm}. 

Computing numerically optimal transport on manifolds has been less studied. For weighted sums of Diracs, displacement interpolation is achieved by solving the linear program to compute the coupling between the Diracs and then advancing the Diracs with the corresponding weights and constant velocity along the geodesics. In this article, we propose to extend the Eulerian discretization method~\cite{Benamou2001} to solve for the displacement interpolation on a Riemannian manifold.

\medskip
\paragraph{First order and proximal methods}

The convex problem considered by Benamou and Brenier~\cite{Benamou2000} can be re-casted as the optimization of a linear  functional under second order conic constraints (see Section~\ref{sec-socp} for more details). This class of programs can be solved in time polynomial with the desired accuracy using interior points methods, see for instance~\cite{Nesterov-Nemirovsky-Book}. 

However, the special structure of the problem, especially when discretized on an uniform grid, makes its suitable for first order schemes and in particular proximal splitting methods. While they do not reach the same convergence speed for arbitrary conic programs, they work well in practice for large scale problems, in particular when high accuracy is not mandatory, which is a common setup for problems in image processing. Proximal splitting schemes are first order optimization methods that allows one to minimize a sum of so-called ``simple'' functionals, possibly (for some methods) pre-composed by linear operators. A functional is called ``simple'' when it is possible to compute its proximal operator (see expression~\ref{eq:def_prox} for its precise definition) either in closed form, or with high accuracy using a few iterations of some sub-routine. In this article, we focus our attention to the Douglas-Rachford algorithm, introduced by~\cite{Lions-Mercier-DR} and on primal-dual methods. We make use of the recently proposed method~\cite{Chambolle2011}, but other schemes could be used as well, see for instance~\cite{Briceno-Arias-PD}. We refer the reader to~\cite{combettes-pesquet-review} and the references therein for more information about the properties of proximal maps and the associated proximal splitting schemes. 

Note that the ALG2 algorithm proposed by~\cite{Benamou2000} corresponds to applying the  Alternating Direction Method of Multiplier (ADDM)~\cite{Fortin1983} to the Fenchel-Rockafeller dual of the (primal) dynamical transport problem. As shown by~\cite{Gabay83,Eckstein1992}, this corresponds exactly to applying directly (a specific instanciation of) the Douglas-Rachford method to the primal problem, see Section~\ref{subsec-addm-centered-grid} for more details.

\medskip
\paragraph{Fluid mechanics discretization}

While Lagrangian methods utilize a mesh-free discretization (see for instance~\cite{Iollo2011}), thats typically tracks the movement of centers of masses during the transportation, Eulerian methods require a fixed discretization of the spatial domain. The most straightforward strategy is to use an uniform centered discretization of an axis-aligned domain, which is used in most previously cited works, see for instance~\cite{Benamou2000,Angenent2003}. Because of the close connection between dynamical optimal transport and fluid dynamics, we advocate in this article the use of staggered grids~\cite{anderson-cfd}, which better cope with the incompressibility condition.

\subsection{Contributions}

Our first contribution is to show how the method initially proposed in~\cite{Benamou2000} is a specific instance of the Douglas-Rachford algorithm. This allows one to use several variations on the initial method, by changing the values of the two relaxation parameters and using different proximal splittings of the functional (possibly introducing auxiliary variables). Our second contribution is the introduction of a staggered grid discretization which is an efficient and convenient way to enforce incompressibility constraints. We show how this discretization  fits into our proximal splitting methodology by introducing an interpolation operator and either making use of auxiliary variables or primal-dual methods. Our last contribution includes an exploration of several variations on the original convex transportation objective, the one proposed in~\cite{dolbeault2009} and a spatially varying penalization which can be interpreted as replacing the $L^2$ ground cost by a geodesic distance on a Riemannian manifold.  Note that the Matlab source code to reproduce the figures of this article is available online\footnote{\url{https://github.com/gpeyre/2013-SIIMS-ot-splitting/}}.


\section{Dynamical Optimal Transport Formulation}

\subsection{Optimal Transport}

In the following, we restrict our exposition to smooth maps $T : [0,1]^d \mapsto [0,1]^d$ where $d>0$ is the dimension of the problem. A valid transport map $T$ is a map that  pushes forward the measure $f^0(x) \d x$ onto $f^1(x) \d x$. In term of densities, this corresponds to
the constraint
\eq{ 
	f^0(x) = f^1(T(x)) \abs{\det( \partial T(x) )} 
}
where $\partial T(x) \in \RR^{d \times d}$ is the differential of $T$
at $x$. This is known as the gradient equation. We call $\Tt(f^0,f^1)$ the set of transport that satisfies this constraint.
An optimal transport $T$ solves
\eql{\label{eq-ot} 
	\umin{T \in \Tt(f^0,f^1) } \int C(x,T(x)) \d x 
}
where $C(x,y) \geq 0$ is the cost of assigning $x \in [0,1]^d$ to $y \in [0,1]^d$. 
In the case $C(x,y)=\norm{x-y}^2$, the optimal value of~\eqref{eq-ot}, the so-called optimal transport distance, is often called the  $L^2$-Wasserstein distance between the densities $f^0$ and $f^1$.

\subsection{Fluid Mechanics Formulation}

The geodesic path between the measures with densities $f^0(x)$ and $f^1(x)$ can be shown to have density $t \mapsto f(x,t)$ where $t \in [0,1]$ parameterizes the path, where 
\eq{ 
	f(x,t) = f^0( T_t(x) ) 
   \abs{\det( \partial T_t(x) )}
   \qwhereq 
   T_t = (1-t) \text{Id}_d + t T. 
} 
Benamou and Brenier showed in~\cite{Benamou2000} that this geodesic solves the following 
non-convex problem over the densities $f(x,t) \in \RR$ and a velocity field $v(x,t) \in \RR^2$
\eql{\label{eq-bb-continuous_velocity} 
	\umin{ (v,f) \in \Cinc^0 } \frac12\int_{[0,1]^d} \int_0^1 f(x,t)\norm{v(x,t)}^2 \d t \d x, 
}
under the set of non-linear constraints
\eql{\label{eq-inc-constr} 
	\Cinc^0 = \enscond{(v,f)}{\partial_t f+ \diverg_x(fv) = 0, 
      \;  v(0,\cdot)=v(1,\cdot)=0, \;  f(\cdot,0)=f^0, \;  f(\cdot,1)=f^1  }. 
}
where the first relation in $\Cinc^0$ is the continuity equation. We impose homogeneous Neumann conditions on the velocity field $v$ which are the more natural boundary condition in the case of the square. Notice that both Neumann and Dirichlet boundary conditions can easily be implemented in our framework. The difference relies in the projection step on the divergence constraint. This step which is carried out using the Fast Fourier Transform algorithm, has to be adapted depending on the chosen boundary conditions.  We refer to~\cite{Froese2012,Benamou2012} for relevant boundary conditions for other convex geometries. The temporal boundary constraints on $f$ impose a match with the input density data.

From a theoretical point of view, the natural setting to prove existence of minimizers of~\eqref{eq-bb-continuous_velocity} is to relax the variational problem and perform the optimization over the Banach space of Radon measures (i.e. finite Borel measures). It is a sub-space of the space of distributions and the incompressibility constraint~\eqref{eq-inc-constr} should be understood in the sense of distributions. We refer the interested reader to~\cite{Cardaliaguet2012} for more details regarding the theoretical analysis of a class of variational problems generalizing~\eqref{eq-bb-continuous_velocity}.  

Note that once an optimal vector field $v$ solving~\eqref{eq-bb-continuous_velocity} has been computed, it is possible to recover an optimal transport $T$ by integrating the flow in time. From a given $x \in [0,1]^d$, we define the solution $t \mapsto T_t(x)$ solving 
\eq{
	T_0(x) = x
	\qandq
	\foralls t > 0, \quad \pd{T_t(x)}{t} = v( T_t(x), t ).
}
The optimal transport is then obtained at $t=1$, i.e. $T = T_1$, see~\cite{Benamou2000} for more details. 

Following~\cite{Benamou2000}, introducing the change of variable $(v,f) \mapsto (m,f)$, where $m$ is the momentum $m = f v$, one obtains a convex optimization problem over the couple $(f,m)$
\eql{\label{eq-bb-continuous} 
	\umin{ (m,f) \in \Cinc } 
	\Jfunc(m,f) = \int_{[0,1]^d} \int_0^1 \jfunc( m(x,t), f(x,t) ) \d t \d x, 
}
\eql{\label{eq-j-func}
	\qwhereq \foralls ( m,  f) \in \RR^d \times \RR, \quad
	\jfunc( m, f) = 
	\left\{\begin{array}{cl}
		\frac{\norm{ m}^2}{2f} &\text{if } f>0, \\
		0 &\text{if } (m,  f) = (0,0), \\
		+\infty &\text{otherwise}.\end{array}\right.	
}
and the set of linear constraints reads
\eq{ 
	\Cinc = \enscond{(m,f)}{ \partial_t f +\diverg_x(m)= 0, 
      \; m(0,\cdot)=m(1,\cdot)=0, \;  f(\cdot,0)=f^0, \; f(\cdot,1)=f^1  }. 
}

\section{Discretized Dynamic Optimal Transport}

For simplicity of exposure, we describe the discretization for the 1-D case. It extends verbatim to higher dimensional discretization $d>1$.

\subsection{Centered Grid}

We denote $N+1$ the number of discretization points in space, and $P+1$ the number of discretization points in time. We introduce the centered grid discretizing the space-time square $[0,1]^2$ in $(N+1)\times(P+1)$ points as
\eq{
	\Gc = \enscond{ (x_i=i/N ,\, t_j=j/P) \in [0,1]^2 }{ 0 \leq i \leq N, 0 \leq j \leq P }.
}
We denote 
\eq{
	V=(m,f) \in \Ec=(m_{i,j}, f_{i,j})_{ 0 \leq i \leq N }^{ 0 \leq j \leq P }
}
the variables discretized on the centered grid, where $\Ec = (\RR^{d+1})^{\Gc}= (\RR^{2})^{\Gc}$ 
is the finite dimensional space of centered variables. 

\subsection{Staggered Grid}

The use of a staggered grid is very natural in the context of the discretization of a divergence operator associated to a vector field $u=(u_i)_{i=1}^d$ on $\RR^d$ (we focus on the case $d=2$ here). The basic idea is to allow an accurate evaluation of every partial derivative $\partial_{x_i} u_i(P)$ at prescribed nodes  $P$ of a cartesian grid  using standard centered finite differences. One way to perform this computation is to impose to the grid on which the $u_i$ scalar field to be centered on $P$ points along the $x_i$ direction. This simple requirement forces the $u_i$ scalar field to be defined on different grids. The resulting  discrete vector field gives us the possibility to evaluate the divergence operator by a uniform standard centered scheme which is not possible using a single grid of discretization for every component of $(u_i)_i$. As a consequence, similarly to the discretization of PDE's in incompressible fluid dynamics (see for instance~\cite{Harlow1965}), we consider a staggered grid discretization which is more relevant to deal with the continuity equation, and is defined as
\begin{align*}
	\Gs^x \hspace{-0.05cm}&=\hspace{-0.1cm} \enscond{ \left(x_{i}=(i+1/2)/N ,\, t_j=j/P\right) \in \frac{[-1,2N+1]}{2N}\times[0,1]} {\hspace{-0.15cm}-1 \leq i \leq  N, 0 \leq j \leq  P },\\
	\Gs^t \hspace{-0.05cm}&=\hspace{-0.1cm} \enscond{ \left(x_i=i/N,\, t_{j}=(j+1/2)/P\right) \in [0,1]\times\frac{[-1,2P+1]}{2P}}{ 0 \leq i \leq  N, -1 \leq j \leq  P }.
\end{align*}

From these definitions, we see that $\Gs^x$ contains $(N+2)\times(P+1)$ points and $\Gs^t$ corresponds to a $(N+1)\times(P+2)$ discretization. We finally denote 
\eq{
	U=(\bar m,\bar f) \in \Es = ( (\bar m_{i,j})_{ -1 \leq i \leq  N }^{ 0 \leq j \leq  P },
	 (\bar f_{i,j})_{ 0 \leq i \leq  N }^{ -1 \leq j \leq  P } )
}
the variables discretized on the staggered grid, where $\Es = \RR^{\Gs^x} \times \RR^{\Gs^t}$ is the finite dimensional space of staggered variables.

\subsection{Interpolation  and Divergence Operators}

We introduce a midpoint interpolation operator $\interp : \Es \rightarrow \Ec$,  where, for $U = (\bar m,\bar f) \in \Es$, we define  $\interp(U) = (m,f) \in \Ec$  as
\eql{\label{eq-defn-interp}    
	\foralls 0 \leq i\leq  N, \quad \foralls 0 \leq j \leq  P, \quad
	\left\{\begin{array}{cl}
		m_{i,j} &= (\bar m_{i-1,j} + \bar m_{i,j})/2,  \\
  		f_{i,j} &= (\bar f_{i,j-1} + \bar f_{i,j})/2.
\end{array}\right.
}

The space-time divergence operator $\diverg : \Es \rightarrow \RR^{\Gc}$ is defined, for $U = (\bar m,\bar f) \in \Es$ as
\eq{
	\foralls 0 \leq i \leq  N, \: \foralls 0 \leq j \leq  P, \quad
	\diverg(U)_{i,j} = 
	N(\bar m_{i,j} - \bar m_{i-1,j}) + 
	P(\bar f_{i,j} - \bar f_{i,j-1}).
}

\subsection{Boundary Constraints}

We extract the boundary values on the staggered grid using the linear operator $b$, defined, for $U = (\bar m,\bar f) \in \Es$ as
\eq{
	b(U) = \left( (\bar m_{-1,j}, \bar m_{N,j})_{j=0}^{P}, (\bar f_{i,-1}, \bar f_{i,P} )_{i=0}^{N}\right)
	\in \RR^{P+1}\times\RR^{P+1} \times\RR^{N+1}\times\RR^{N+1}.
}
We impose the following boundary values 
\eq{
	b(U)=b_0
	\qwhereq
	b_0 = (0,0,f^0,f^1) \in \RR^{P+1}\times\RR^{P+1} \times \RR^{N+1} \times \RR^{N+1},
}
where $f^0,f^1 \in \RR^{N+1}$ are the discretized initial (time $t=0$) and final (time $t=1$) densities and the spatial boundary constraint $0\in  \RR^{P}$ on the momentum $m=fv$ comes from the discretized  Neumann boundary conditions on the velocity field $v$.  Notice that in the 2-D or 3-D  cases, Neumann and Dirichlet Boundary conditions do not coincide. For instance in 2-D, Neumann conditions are equivalent to force the first component of $\bar m$ to be zero only on the vertical segments of the boundary whereas the second component vanishes only on the horizontal ones.

\subsection{Discrete Convex Problem}

The initial continuous problem~\eqref{eq-bb-continuous} is approximated on the discretization grid by solving the finite dimensional convex problem
\eql{ \label{eq-optim-bb-disc}
	\umin{U \in \Es} \Jfunc(\interp(U)) + \iota_\Cinc(U). 
}
Here, for a closed convex set $\Cc$, we have denoted its associated indicator function
\eq{
	\iota_\Cc(U) = 
	\left\{\begin{array}{cl}
		0 &\text{if } U \in \Cc,\\
		+\infty &\text{otherwise.}
	\end{array}\right.
}

The discrete objective functional $\Jfunc$ reads for $V=(m,f) \in \Ec$:
\eql{\label{eq-jfunc-totale}
	\Jfunc(V) = \sum_{k \in \Gc} \jfunc(m_k,f_k),
}	
where we denote $k=(i,j) \in \Gc$ the indexes on the centered grid, and the functional $\jfunc$ is defined in~\eqref{eq-j-func}.

The constraint set $\Cinc$ corresponds to the divergence-free constraint together with the boundary constraints
\eq{
	\Cinc = \enscond{U \in \Gs}{ \diverg(U) = 0 \qandq b(U)=b_0}.
}

\subsection{Second Order Cone Programming Formulation}
\label{sec-socp}

The discretized problem~\eqref{eq-optim-bb-disc} can be equivalently solved as the following minimization over variables $(U,V,W) \in \Es \times \Ec \times \RR^{\Gc}$
\eql{\label{eq-formulation-socp}
	\umin{ (V,W) \in \Kk,  V=\Ii(U) } \sum_{k \in \Gc} W_k
}
where $\Kk$ is the product of (rotated) Lorentz cones 
\eq{
	\Kk = \enscond{ (V = (\bar m,\bar f)),W) \in \Ec \times \RR^{\Gc} }{ \foralls k \in \Gc, \; \norm{\bar m_k}^2 - W_k \bar f_k \leq 0 }.
}
Problem~\eqref{eq-formulation-socp} is a specific instance of so-called second-order cone program (SOCP), that can be solved in time polynomial with the accuracy using interior point methods (see Section~\ref{sec-previous} for more details). As already mentioned in Section~\ref{sec-previous}, we focus in this article on proximal splitting methods, that are more adapted to large scale imaging problems.

\section{Proximal Splitting Algorithms}
\label{sec-proximal}

In this section, we review some splitting schemes and detail how they can be applied to solve~\eqref{eq-optim-bb-disc}. This requires to compute the proximal operators of the cost function $J$ and of indicator of the constraint set $\Cinc$. 

\subsection{Proximal Operators and Splitting Schemes}

The minimization of a convex functional $F$ over some Hilbert space $\Hh$ requires the use of algorithms that are tailored to the specific properties of the functional. Smooth functions can be minimized through the use of gradient descent steps, which amounts to apply repeatedly the mapping $\Id_{\Hh} - \ga \nabla F$ for a small enough step size $\ga>0$. Such schemes can not be used for non-smooth functions such as the one considered in~\eqref{eq-optim-bb-disc}. 

A popular class of first order methods, so-called proximal splitting schemes, replaces the gradient descent step by the application of the proximal operator. The proximal operator $\prox_{\ga F} : \Hh \rightarrow \Hh$ of a convex, lower semi-continuous and proper functional $F : \Hh \rightarrow \RR \cup \{+\infty\}$ is defined as
\eql{\label{eq:def_prox} 
	\prox_{\ga F}(z) = \uargmin{\tilde z \in \Hh} \frac{1}{2} \norm{z-\tilde z}^2 + \ga F(\tilde z). 
}
This proximal operator is a single-valued map, which is related to the set-valued map of the sub-differential $\partial F$ by the relationship $\prox_{\ga F} = ( \Id_{\Hh} + \ga \partial F )^{-1}$. This is why the application of $\prox_{\ga F}$ is often referred to as an implicit descent step, which should be compared with the explicit gradient descent step, $\Id_{\Hh} - \ga \nabla F$. 

Proximal operators enjoy several interesting algebraic properties which help the practitioner to compute it for complicated functionals. A typical example is the computation of the proximal operator of the Legendre-Fenchel transform of a function. The Legendre-Fenchel transform of $F$ is defined as
\eql{\label{eq-legendre-transform}
	F^*(w) = \umax{z \in \Hh} \dotp{z}{w}-F(z).
}
Note that thanks to Moreau's identity~\cite{Moreau1965}
\eql{\label{eq-moreau-identity}
	\foralls w \in \Hh, \quad \prox_{\ga F^\star}(w) = w - \ga \prox_{F/\ga}(w/\ga),
}
computing the proximal operator of $F^*$ has the same complexity as computing the proximal operator of $F$. 

To enable the use of these proximal operators within an optimization scheme, it is necessary to be able to compute them efficiently. We call a function $F$ such that $\prox_{\ga F}$ can be computed in closed form a \textit{simple function}. It is often the case that the function to be minimized is not simple. The main idea underlying proximal splitting methods is to decompose this function into a sum of simple functions (possibly composed by linear operators). We detail bellow three popular splitting schemes, the Douglas-Rachford (DR) algorithm, the Alternating Direction Method of Multipliers (ADMM) and a primal-dual algorithm.  We refer the reader to~\cite{combettes-pesquet-review} for a detailed review of proximal operators and proximal splitting schemes.

\subsection{Computing $\prox_{\ga \Jfunc}$}

The following proposition shows that the functional $\Jfunc$ defined in~\eqref{eq-jfunc-totale} is simple, in the sense that its proximal operator can be computed in closed form. 

\begin{prop}\label{proposition1}
	One has
	\eq{
		\foralls V \in \Ec, 
		\quad \prox_{\ga \Jfunc}(V) = ( \prox_{\ga \jfunc}(V_k) )_{k \in \Gc}
	}
	where, for all $(\tilde m,\tilde f) \in \RR^d \times \RR$, 
	\eq{
		\prox_{\ga \jfunc}(\tilde m,\tilde f) = 
		\left\{\begin{array}{cl}
	  		(\mu(f^{\star}),f^\star) &\text{if } f^\star > 0, \\
		  	(0,0) &\text{otherwise}.\end{array}\right.
	}
	\eql{\label{eq:mstar}
		\qwhereq
		\foralls f \geq 0, \quad \mu(f) = \frac{ f \tilde m}{ f + \ga }
	}
	and $f^\star$ is the largest real root of the third order polynomial equation in $X$
	\eql{\label{eq:polynome3}
		P(X)=(X-\tilde f)(X+\ga) ^2-\frac{\ga}2\norm{\tilde m}^2 = 0	.
	}
\end{prop}
\begin{proof}
	We denote $(m,f) = \prox_{\ga \jfunc}(\tilde m,\tilde f)$. 
	If $f > 0$, since $\jfunc$ is $C^1$ and is strongly convex on $\RR^d \times \RR^{+,*}$, necessarily $(m,f)$ is the unique solution of $\nabla \jfunc(m,f)=0$, which reads
	\eq{
		\choice{
			\ga \frac{m}{f} + m-\tilde m=0, \\
			 -\ga \frac{\norm{ m}^2}{f^{2}} + f-\tilde f = 0.
		}
	}
	Reformulating these equations leads to the following equivalent conditions
	\eq{
		P(f) = 0
		\qandq
		m = \mu(f).
	}
	This shows that if $P$ as at least a strictly positive real root $f^\star$, it is necessarily unique and that $(f=f^\star,m=\mu(f^\star))$.
	On the contrary, one necessarily has $f=0$ and, by definition of $J$, $m=0$ as well.
\end{proof}

\subsection{Computing $\proj_{\Cinc}$}

The proximal mapping of $\iota_\Cinc$ is $\proj_{\Cinc}$ the orthogonal projector on the convex set $\Cinc$. This is an affine set that can be written as 
\eq{
	\Cc = \enscond{ U =(m,f) \in \Es }{ A U = y }
	\qwhereq
	A U = (\diverg(U), b(U))
	\qandq
	y = (0,b_0).
}
This projection can be computed by solving a linear system as
\eq{
	\proj_{\Cinc} = (\Id - A^* \De^{-1} A) + A^* \De^{-1} y
}
where applying $\De^{-1} = (AA^*)^{-1}$ requires solving a Poisson equation on the centered grid $\Gc$ with the prescribed boundary conditions. This can be achieved with Fast Fourier Transform in $O(NP \log(NP))$ operations where $N$ and $P$ are number of spatial and temporal points, see~\cite{Strang1988}.

\subsection{Douglas-Rachford Solver}
\label{DR-algo}

\paragraph{DR algorithm}

The Douglas-Rachford (DR) algorithm~\cite{Lions-Mercier-DR} is a proximal splitting method that allows one to solve
\eql{\label{eq-dr-problem}
	\umin{ z \in \Hh } G_1(z) + G_2(z)
}
where $G_1$ and $G_2$ are two simple functions defined on some Hilbert space $\Hh$.

The iterations of the DR algorithm define a sequence $(\iter{z}, \iter{w}) \in \Hh^2$ using a initial $(z^{(0)}, w^{(0)})\in \Hh^2$ and 
\begin{equation}\label{eq-dr-iters}
\begin{aligned}
	\iiter{w} &=\iter{w} + \al (\prox_{\ga G_1} (2\iter{z}-\iter{w})-\iter{z}),\\
\iiter{z} &= \prox_{\ga G_2}(\iiter{w}).
\end{aligned}
\end{equation}
If $0 < \al < 2$ and $\ga > 0$, one can show that $\iter{z} \rightarrow z^\star$ a solution of~\eqref{eq-dr-problem}, see~\cite{Combettes2007} for more details. 

In the following, we describe several possible ways to map the optimal transport optimization~\eqref{eq-optim-bb-disc} into the splitting formulation~\eqref{eq-dr-problem}, which in turn gives rise to four different algorithms.

\paragraph{Asymmetric-DR (A-DR) splitting scheme}

We recast the initial optimal transport problem~\eqref{eq-optim-bb-disc} as an optimization of the form~\eqref{eq-dr-problem} by using the variables 
\eq{
	z=(U,V) \in \Hh = \Es \times \Ec
}
and setting the functionals minimized as
\eq{
	\foralls z=(U,V) \in \Es \times \Ec, \quad
	G_1(z) = \Jfunc(V) + \iota_{\Cinc}(U)
	\qandq
	G_2(z) = \iota_{\Cc_{s,c}}(z).
}
In this expression, $\Cc_{s,c}$ is the constraint set that couples staggered and centered variables
\eq{
	\Cc_{s,c} = \enscond{z=(U,V) \in \Es \times \Ec}{ V=\interp(U) }
}
and $\interp$ is the interpolation operator defined in~\eqref{eq-defn-interp}.

Both $G_1$ and $G_2$ are simple functions since
\begin{align}\label{eq-def-dr-functionals}
	\prox_{\ga G_1}(U,V) &= ( \proj_{\Cinc}(U), \prox_{\ga \Jfunc}(V) ), \\
	\prox_{\ga G_2}(U,V) &= \proj_{\Cc_{s,c}}(U,V) = ( \tilde U, \Ii(\tilde U) )
	\qwhereq
	\tilde U = (\Id + \Ii^* \Ii)^{-1}( U + \Ii^*(V) )
\end{align}
where $\interp^*$ is the adjoint of the linear interpolation operator. Note that computing $\proj_{\Cc_{s,c}}$ requires solving a linear system, but this system is separable along each dimension of the discretization grid, so it only requires solving a series of small linear systems. Furthermore, since the corresponding inverse matrix is the same along each dimension, we pre-compute explicitly the inverse of these $d+1$ matrices.

In our case, the iterates variables appearing in~\eqref{eq-dr-iters} read $\iter{z} = (\iter{U}, \iter{V})$, which allows one to retrieve an approximation $\iter{f}$ of the transport geodesic as $\iter{U} = (\iter{m},\iter{f})$.

A nice feature of this scheme A-DR is that the iterates always satisfy $\iter{U} \in \Cc$, but in general one does not have $\iter{V} = \interp(\iter{U})$.

\paragraph{Asymmetric-DR' (A-DR') splitting scheme}

In the DR algorithm~\eqref{eq-dr-iters}, the functionals $G_1$ and $G_2$ do not play a symmetric role. One thus obtains a different algorithm (that we denote as A-DR'), by simply exchanging the definitions of $G_1$ and $G_2$ in~\eqref{eq-def-dr-functionals}. Using this scheme A-DR', one has $\iter{V} = \interp{\iter{U}}$, but in general $\iter{U}$ is not in the constraint set $\Cc$.

\paragraph{Symmetric-DR (S-DR) splitting scheme}

In order to restore symmetry between the functionals $\Jfunc$ and $\iota_\Cc$ involved in the minimization algorithm, one can consider the following formulation
\eq{
	z = (U,V,\tilde U, \tilde V) \in \Hh = 	(\Es \times \Ec)^2
}
using the following functionals
\eq{
	G_1(z) = \Jfunc(V) + \iota_{\Cinc}(U) +	\iota_{\Cc_{s,c}}(\tilde U, \tilde V)
	\qandq
	G_2(z) = \iota_{\Dd}(z), 
}
where $\Dd$ is the diagonal constraint
\eql{\label{eq-split-spingarn}
	\Dd = \enscond{ z = (U,V,\tilde U, \tilde V) \in \Hh }{  
		U=\tilde U \qandq V=\tilde V
	}.	
} 
These functionals are simple since, for $z = (U,V,\tilde U, \tilde V) \in \Hh$, one has
\begin{align*}
	\prox_{\ga G_1}(z) &= 
	\pa{
		\proj_{\Cinc}(U), \prox_{\ga \Jfunc}(V),
		\proj_{\Cc_{s,c}}(\tilde U,\tilde V)
	}\\
	\qandq
	\prox_{\ga G_2}(z) &= \frac{1}{2} \pa{ 
		U+\tilde U, V+\tilde V, U+\tilde U, V+\tilde V
	}.
\end{align*}

The splitting reformulation of the form~\eqref{eq-split-spingarn} was introduced in~\cite{Spingarn85}, and the corresponding DR scheme, extended to a sum of an arbitrary number of functionals, is detailed in~\cite{Pustelnik-PPXA,CombettesPesquet-PPXA}.

\paragraph{Symmetric-DR' (S-DR') splitting scheme}

Similarely to the relationship between A-DR and A-DR' algorithm, it is possible to define a S-DR' algorithm by reversing the roles of $G_1$ and $G_2$ in the algorithm S-DR.

\subsection{Primal-Dual Solver}
\label{PD-algo}

Primal dual (PD) algorithms such as the relaxed Arrow-Hurwitz method introduced in~\cite{Chambolle2011} allows one to minimize functionals of the form $G_1+G_2 \circ A$ where $A$ is a linear operator and $G_1, G_2$ are simple functions. One can thus directly apply this method to problem~\eqref{eq-optim-bb-disc} with $G_2=\Jfunc$, $A=\interp$ 	and $G_1 = \iota_{\Cinc}$.

The iterations compute a sequence $(\iter{U},\iter{\Upsilon}, \iter{V}) \in \Es \times \Es \times \Ec$ of variables from an initial $(\Upsilon^{(0)}, V^{(0)})$ according to
\begin{equation}
\begin{aligned}
	\iiter{V} &= \prox_{\sigma G_2^*}( \iter{V} + \sigma \interp    \iter{\Upsilon} ), \\
	\iiter{U} &= \prox_{\tau G_1}(  \iter{U} - \tau \interp^*\iiter{V}   ), \\
	\iiter{\Upsilon} &= \iiter{U} + \theta (  \iiter{U} - \iter{U} ).
\end{aligned}
\end{equation}
Note that $\prox_{\sigma G_2^*}$ can be computed using $\prox_{G_2}$ following equation~\eqref{eq-moreau-identity}. If $0 \leq \theta \leq 1$ and  $\sigma \tau \norm{\interp}^2<1$ then one can prove that $\iter{U} \rightarrow U^\star$ which is a solution of~\eqref{eq-optim-bb-disc}, see~\cite{Chambolle2011}.

The case $\th=0$ corresponds to the Arrow-Hurwicz algorithm~\cite{Arrow1958}, and the general case can be interpreted as a preconditioned version of the ADDM algorithm, as detailed in~\cite{Chambolle2011}.

\subsection{ADMM Solver on Centered Grid}
\label{subsec-addm-centered-grid}

In this Section, we give some details about the relationship between the algorithm ALG2 developed by Benamou and Brenier in~\cite{Benamou2000} and our DR algorithms. The original paper~\cite{Benamou2000} considers a finite difference implementation on a centered grid, which leads to solve the following optimization problem
\eql{\label{eq-min-centered-grid}
	\umin{ V \in  \Ec } \Jfunc(V) + \iota_{\tilde\Cinc}(V)
	\qwhereq
	\tilde\Cc = \enscond{ V \in \Es }{ A V = y }
} 
where $A$ is this time defined on a centered grid. 

\paragraph{Primal problem}

The minimization~\eqref{eq-min-centered-grid} has the form
\eql{\label{eq-addm-dr-primal}
	\umin{z \in \Hh} F \circ A(z) + G(z), 
}
where $A : \Hh \rightarrow \tilde\Hh$ is a linear operator, and $F \circ A$ and $G$ are supposed to be simple functions. 

For the OT problem on a centered grid~\eqref{eq-min-centered-grid}, the identification is obtained by setting 
\eq{
	F = \iota_{\{y\}}
	\qandq 
	G = \Jfunc,
} 
which are indeed simple functions. One can use the DR algorithm~\eqref{eq-dr-iters} with 
\eql{\label{eq-bb-splitting-dr}
	G_1=F \circ A
	\qandq 
	G_2=G
} 
to solve problems of the form~\eqref{eq-addm-dr-primal}. Of course, the variants A-DR', S-DR and S-DR' considered in Section~\ref{DR-algo} could be used as well on a centered grid.  Since we focus on the relationship with the work~\cite{Benamou2000}, we only consider the splitting \eqref{eq-bb-splitting-dr}.

\paragraph{Dual problem}

Following~\cite{Benamou2000}, one can consider the Fenchel-Rockafeller dual to the primal program~\eqref{eq-addm-dr-primal}, which reads
\eql{\label{eq-fenchel-dual-pbm}
	\umax{s \in \tilde\Hh}  - \pa{ G^* \circ (-A^*)(s) + F^*(s) }.
}
In the specific setting of the OT problem~\eqref{eq-min-centered-grid}, $F^*(p) = \dotp{y}{p}$ and the following proposition, proved in~\cite{Benamou2000}, shows that $G^* = \Jfunc^*$ is a projector on a convex set (which is a consequence of the fact that $\Jfunc$ is a 1-homogenous functional). 

\begin{prop}
	One has
	\eq{
		\Jfunc^* = \iota_{ \Cc_{\Jfunc} }
		\qwhereq
		\choice{
			\Cc_{\Jfunc} = \enscond{ V \in \Ec }{ \foralls k \in \Gc, \; V_k \in \Cc_\jfunc }, \\	
			\Cc_\jfunc = \enscond{ (a,b) \in \RR^d \times \RR }{ \norm{a}^2 + 2b \leq 0 }.
		}
	}
\end{prop}
\begin{proof}
	The Lengendre-Fenchel transform~\eqref{eq-legendre-transform} of  the functional $\jfunc$ defined in~\eqref{eq-j-func} at point $(a,b)\in \RR^d \times \RR$ reads
\eq{
		\jfunc^*(a,b) =\umax{(m,f) \in \RR^d \times \RR} \dotp{a}{m}+\dotp{b}{f}-\frac{m^2}{2f},
	}
where we just  focus on the case $f>0$, since $f=0$ (resp. $f<0$) will always give $\jfunc^*=0$ (resp. $\jfunc^*=-\infty$).
The optimality conditions  are given by
	\eq{
		a =\frac{m}f \qandq b=-\frac{m^2}{2f^2}.
	}
Hence, we have that
\begin{align*}
	\jfunc^*(a,b) &=\umax{(m,f) \in \Ec} f(\dotp{a}{\frac{m}f}+b) -\frac{m^2}{2f^2}\\
 &= \umax{(m,f) \in \Ec} f \left(\dotp{a}{a}+2b\right),
\end{align*}
which is $0$ if  $\norm{a}^2 + 2b \leq 0$ and $+\infty$ otherwise.
\end{proof}

Note that one can use the Moreau's identity~\eqref{eq-moreau-identity} to compute the proximal operator of $\jfunc$ using the orthogonal projection on $\Cc_\jfunc$ and vice-versa
\eq{
	\prox_{\ga \jfunc}(v) = v - \ga \proj_{\Cc_{\jfunc}} (v/\ga).
}

\paragraph{ADMM method}

The Alternating Direction Method of Multipliers (ADMM) is an algorithm to solve a minimization of the form
\eql{\label{eq-addm-dr-dual}
	\umin{s \in \tilde\Hh}  H \circ B(s) + K(s) ,    
}
where we assume that the function $K^* \circ B^*$ is simple and that $B : \tilde\Hh \rightarrow \Hh$ is injective. It was initially proposed in~\cite{GabayMercier,GlowinskiMarroco}.


We introduce the Lagrangian associated to~\eqref{eq-addm-dr-dual} to account for an auxiliary variable $q \in \Hh$ satisfying $q=Bs$ with a multiplier variable $v \in \Hh$
\eq{
	\foralls (s,q,v) \in \tilde \Hh \times \Hh \times \Hh, \quad
	L(s,q,v) = K(s) + H(q) + \dotp{v}{ Bs-q}, 
}
and the augmented Lagrangian for $\ga > 0$
\eq{
	\foralls (s,q,v) \in \tilde \Hh \times \Hh \times \Hh, \quad
	L_\ga(s,q,v) = L(s,q,v)  + \frac{\ga}{2}\norm{ Bs-q}^2
}
The ADMM algorithm generates iterates $(\iter{s}, \iter{q}, \iter{v})  \in \tilde \Hh \times \Hh \times \Hh$ as follow 
\begin{equation}\label{eq-admm-iter}
\begin{aligned}
	\iiter{s} &= \uargmin{s} L_\ga(s,\iter{q},\iter{v}), \\
	\iiter{q} &= \uargmin{q} L_\ga(\iiter{s},q,\iter{v}), \\
	\iiter{v} &= \iter{v} + \ga (B \iiter{s}-\iiter{q}).
\end{aligned}
\end{equation}
This scheme can be shown to converge for any $\ga > 0$, see~\cite{GabayMercier,GlowinskiMarroco}. A recent review of the ADMM algorithm and its applications in machine learning can be found in~\cite{BoydADMM}.

We introduce the proximal operator $\prox_{\ga F}^B : \Hh \rightarrow \tilde \Hh$ with a metric induced by an injective linear map $B$ 
\eql{\label{eq-def-proxB}
	\foralls z \in \Hh, \quad
	\prox_{K/\ga}^B(z) = \uargmin{s \in \tilde\Hh} \frac{1}{2} \norm{B s - z}^2 + \frac{1}{\ga} K(s).
}
One can then re-write the ADMM iterations~\eqref{eq-admm-iter} using proximal maps
\begin{equation}\label{eq-admm-prox-iters}
\begin{aligned}
	\iiter{s} &= \prox_{K/\ga }^B ( \iter{q} - \iter{u}), \\
	\iiter{q} &= \prox_{H/\ga } ( B\iiter{s} + \iter{u} ), \\
	\iiter{u} &= \iter{u} + B \iiter{s}-\iiter{q} .
\end{aligned}
\end{equation}
where we have performed the change of variable $\iter{u} = \iter{v}/\ga$ to simplify the notations. 


The following proposition shows that if $K^* \circ B^*$ is simple, one can indeed perform the ADMM algorithm. Note that in the case $\Hh=\tilde\Hh$ and  $B=\Id_{\Hh}$, one recovers Moreau's identity~\eqref{eq-moreau-identity}.

\begin{prop}\label{eq-proxB-proxD}
One has
\eq{
	\foralls z \in \Hh, \quad
		\prox_{K/\ga}^B(z) = B^+ \pa{ z - \frac{1}{\ga} \prox_{\ga K^* \circ B^* }(\ga z) }.
}
\end{prop}
\begin{proof}
	We denote $\Uu = \partial K$ the set-valued maximal monotone operator. One has
	$\partial K^* = \Uu^{-1}$, where $\Uu^{-1}$ is the set-valued inverse operator, and $\partial (K^* \circ B^*) = B \circ\Vv \circ B^*$. 
	One has $\prox_{\ga K^* \circ B^* } = (\Id + \ga \Vv)^{-1}$, which a single-valued operator. 
	Denoting  $s = \prox_{K/\ga}^B(z)$, the optimality conditions of~\eqref{eq-def-proxB} lead to
	\begin{align}
		& 
		0 \in B^*( B s - z) + \frac{1}{\ga} \Uu(s) 
		\Longleftrightarrow 
		s \in \Uu^{-1} \pa{ \ga B^* B s - \ga B^* z   }
		\Longleftrightarrow 
		\ga B s \in \ga \Vv  \pa{ \ga z - \ga B s    } \\
		& 
		\Longleftrightarrow
		\ga B s \in ( \Id + \ga \Vv )\pa{ \ga z - \ga B s   } + \ga B s - \ga z 
		\Longleftrightarrow
		\ga z \in ( \Id + \ga \Vv )\pa{ \ga z - \ga B s   } \\
		&
		\Longleftrightarrow
		\ga z - \ga B s = ( \Id + \ga \Vv )^{-1}(\ga z) 
		\Longleftrightarrow s = B^{+}\pa{ z - \frac{1}{\ga} ( \Id + \ga \Vv )^{-1}(\ga z) }		
	\end{align}
	where we used the fact that $B$ is injective. 
\end{proof}

\paragraph{Equivalence between ADMM and DR}

The ALG2 algorithm of~\cite{Benamou2000} corresponds to applying the ADMM algorithm to the dual problem~\eqref{eq-fenchel-dual-pbm} that can be formulated as
\eql{\label{eq-fenchel-dual-pbm2}
	\umin{s \in \tilde\Hh}   \pa{ G^* \circ (-A^*)(s) + F^*(s) },
} 
while retrieving at each iteration the primal iterates in order to solve the primal problem~\eqref{eq-addm-dr-primal}. The following proposition, which was initially proved in~\cite{Gabay83}, shows that applying ADMM algorithm to the dual~\eqref{eq-fenchel-dual-pbm} is equivalent to solving the primal~\eqref{eq-addm-dr-primal} using DR. More precisely, the dual variable (the Lagrange multiplier) $\iter{v}$ of the ADMM iterations is equal to the primal variable $\iter{z}$ of the DR iterations. This result was further extended by~\cite{Eckstein1992} which shows the equivalence of ADMM with the proximal point algorithm and propose several generalizations. For the sake of completeness, we detail the proof of this result using our own notations. 

\begin{prop}
	Setting
	\eql{\label{eq-chg-var-dr-addmm}
		H = G^*, \quad
		K = F^*, \qandq
		B = -A^*, 
	}
	the DR iterations~\eqref{eq-dr-iters} with functionals~\eqref{eq-bb-splitting-dr} and $\al=1$  are recovered from the ADMM iterations~\eqref{eq-admm-iter} using
	\begin{align*}
		   \iter{v}  &= \iter{z},\\
		\ga  \iter{q} &= \iter{w}-\iter{z}, \\
		\ga  A^* \iiter{s} &= \iter{z}-\iiter{w}.
	\end{align*}
\end{prop}
\begin{proof}
	Denoting $\iter{\bar s} = A^* \iter{s} \in \Im(A^*)$ (recall that $B=-A^*$ is injective), using the change of notations~\eqref{eq-chg-var-dr-addmm} and using the result of Proposition~\ref{eq-proxB-proxD}, the iterates~\eqref{eq-admm-prox-iters} read
	\begin{equation}\label{eq-equiv-iter1}
		\begin{aligned}
		-\iiter{\bar s} &=  \iter{q} - \iter{u}  + \frac{1}{\ga} \prox_{\ga F \circ A}\pa{ \ga(\iter{u}-\iter{q}) }, \\
		\iiter{q} &=  \iter{u} -\iiter{\bar s} - \frac{1}{\ga} \prox_{\ga G}\pa{ \ga( \iter{u}-\iiter{\bar s} )}, \\
		\iiter{u} &= \iter{u} - \iiter{\bar s}-  \iiter{q} 
		\end{aligned}
	\end{equation}
where we have used the fact that $\prox_{\ga F \circ (-A)}(z) =  -\prox_{\ga F \circ A}(-z)$.
Recall that the DR iterations~\eqref{eq-dr-iters} read in the setting $\al=1$
\begin{equation}\label{eq-equiv-iter2}
  		\begin{aligned}
		\iiter{w} &= \iter{w}+\prox_{\ga F \circ A}(2\iter{z}-\iter{w})-\iter{z}, \\
		\iiter{z} &= \prox_{\ga G} (\iiter{w}).
		\end{aligned}
	\end{equation}
	Iterations~\eqref{eq-equiv-iter1} and~\eqref{eq-equiv-iter2} are using the following identification between $(\iter{q}, \iiter{q},\iter{u}, \iiter{\bar s})$ and 
	$(\iter{w}, \iter{z}, \iiter{w}, \iiter{z})$

	\begin{align}
		\ga (\iter{u}-\iter{q})  &= 2\iter{z}-\iter{w},\label{eq-identif-1}\\
		\ga ( \iter{u}-\iiter{\bar s}) &= \iiter{w} , \label{eq-identif-2}\\
		\ga( \iter{u}- \iter{q}   -\iiter{\bar s}   )&= \iiter{w}+\iter{z}-\iter{w}, \label{eq-identif-3}\\
		\ga( \iter{u}-\iiter{q}- \iiter{\bar s}     )&=  \iiter{z}, \label{eq-identif-4}
	\end{align}	
Solving the linear system~\eqref{eq-identif-1}-\eqref{eq-identif-4}, one gets
	\begin{align}
		\ga   \iter{u}  &= \iter{z},\label{eq-identif2-1}\\
		\ga  \iter{q} &= \iter{w}-\iter{z}, \label{eq-identif2-2}\\
		\ga  \iiter{\bar s} &= \iter{z}-\iiter{w}, \label{eq-identif2-3}\\
		\ga\iiter{q}&=\iiter{w}-\iiter{z}. \label{eq-identif2-4}
	\end{align}
First notice that relations~\eqref{eq-identif2-2} and~\eqref{eq-identif2-4} are compatible at iterations $\ell$ and $\ell+1$. 
Then, identifying the update of the variable $u$ presented in~\eqref{eq-equiv-iter1} with  relation~\eqref{eq-identif-4}, we recover $\iiter{z}=\ga\iiter{u}=\iiter{v}$, which corresponds to the relation~\eqref{eq-identif2-1} at iteration $\ell+1$.
\end{proof}

This proposition thus shows that, on a centered grid, ALG2 of~\cite{Benamou2000} gives the same iterates as DR on the primal problem. Since we consider a staggered grid, the use of the interpolation operator makes our optimization problem~\eqref{eq-dr-problem} different from the original ALG2 and requires the introduction of an auxiliary variable $V$. Furthermore, the introduction of an extra relaxation parameter $\al$ is useful to speed-up the convergence of the method, as will be established in the experimentations. Lastly, let us recall that it is possible to use the variants A-DR', S-DR and S-DR' of the initial A-DR formulation, as detailed in Section~\ref{DR-algo}.

\section{Generalized Cost Functions}
\label{sec-generalized}

Following \cite{dolbeault2009,Cardaliaguet2012}, we propose to use a generalized cost function that allows one to compute geodesics that interpolate between the $L^2$-Wasserstein and the $H^{-1}$ geodesics. To introduce further flexibility, we introduce spatially varying weights, which corresponds to approximating a transportation problem on a Riemannian manifold.

\subsection{Spatially Varying Interpolation between $L^2$-Wasserstein and $H^{-1}$ }\label{sec:generalized_cost}

We define our discretized generalized transportation problem as
\eql{ \label{eq-optim-bb-gen}
	\umin{U \in \Es} \Jfunc_\bet^w(\interp(U)) + \iota_\Cinc(U),
}
where the vector of weights $w = (w_k)_{k \in \Gc}$ satisfies $c < w_k \leq +\infty$ and  $c>0$ is a small constant. Note that we allow for infinite weights $w_k=+\infty$, which corresponds to forbidding the transport to put mass in a given cell indexed by $k$. The generalized functional   reads, for $\be \geq 0 $
\eql{\label{eq:generelized_cost}
	\Jfunc_\bet^w(V) = \sum_{k \in \Gc} w_k \jfunc_\bet(m_k,f_k), 
}
\eql{
	\foralls (m,  f) \in \RR^d \times \RR, \quad
	\jfunc_\be(m,f) = 
	\left\{\begin{array}{cl}
		\frac{\norm{m}^2}{2f^\bet} &\text{if }  f>0, \\
		0 &\text{if } (m,  f) = (0,0), \\
		+\infty &\text{otherwise}.
	\end{array}\right.	
}
Note that $\Jfunc_1^1 = \Jfunc$ and that for $\be \in [0,1]$, $\Jfunc_\bet^w$ is convex. 
Furthermore, one has, for $f>0$  
\eql{\label{hessian}
	\det( \partial^2 \jfunc_\be(m,f) ) = \frac{\bet(1-\bet)\norm{m}^2}{f^{3\bet+2}}.
}
This shows that $\jfunc_\be$ is strictly convex on $\RR^{*,d} \times \RR^{+,*}$ for $\be \in ]0,1[$.

The case of constant weights $w_k=1$ is studied in~\cite{dolbeault2009,Cardaliaguet2012}. The case $\be=1$ corresponds to the Wasserstein $L^2$ distance. In a continuous (not discretized) domain, the value of the problem~\eqref{eq-optim-bb-gen} is equal to the $H^{-1}$ Sobolev norm over densities $\norm{f_0-f_1}_{H^{-1}}$, as detailed in~\cite{dolbeault2009}. Note that since in this case the induced distance is actually an Hilbertian norm, the corresponding geodesic is a linear interpolation of the input measures, and thus for measures having densities, one obtains $f_t = (1-t) f_0 + t f_1$.

When $\be=1$ and the weights $w_k$ are constant in time, the solution of~\eqref{eq:generelized_cost} discretizes the displacement interpolation between the densities $(f^0,f^1)$ for a ground cost being the squared geodesic distance on a Riemannian manifold (see Section~\ref{sec-previous} for more details). Note that we restrict our attention to isotropic Riemannian metrics (being proportional to the identity at each point), but this extends to arbitrary Riemannian or even Finsler metrics.  Studying the properties of this transportation distance is however not the purpose of this work, and we show in Section~\ref{subsec-riemanian-examples} numerical illustrations of the influence of $w$.

\subsection{Computing $\prox_{\ga J_\bet}$}

The following proposition, which generalizes Proposition~\ref{proposition1}, shows how to compute $\prox_{\ga \Jfunc_\bet^w}$.

\begin{prop}\label{prop-generalized}
	One has
	\eq{
		\foralls V \in \Ec, 
		\quad \prox_{\ga \Jfunc_\be^w}(V) = ( \prox_{\ga w_k \jfunc_\be}(V_k) )_{k \in \Gc}
	}
	where, for all $(\tilde m,\tilde f) \in \RR^d \times \RR$, 
	\eq{
		\prox_{\ga \jfunc_\be}(\tilde m,\tilde f) = 
		\left\{\begin{array}{cl}
	  		(\mu(f^{\star}),f^\star) &\text{if } f^\star > 0 \text{ and  }\ga< \infty, \\
		  	(0,0) &\text{otherwise}.\end{array}\right.
	}
	\eql{\label{eq:mstar-bis}
		\qwhereq
		\foralls f \geq 0, \quad \mu(f) = \frac{f^\bet \tilde m}{f^\bet+\ga} \in \RR^d
	}
	and $f^\star$ is the largest real root of the following equation in $X$
	\eql{\label{eq:poly_alpha}
     	P_\bet(X) = X^{1-\bet} (X-\tilde f)(X^\bet+\ga)^2  - \frac{\ga}2\bet \norm{\tilde m}^2 = 0
	}
\end{prop}
\begin{proof}
	We denote $(\bar m,\bar f) = \prox_{\ga \jfunc_\be}(\tilde m,\tilde f)$, and
	\eq{
		\foralls (m,f) \in \RR^d \times \RR, \quad
		\Phi_\be(m,f) = \frac{1}{2}\norm{ (m,f) - (\tilde m, \tilde f) }^2 + \ga \jfunc_\be( m,f ).
	}
	If $\bar f > 0$, since $\Phi_\be$ is $C^1$ and is strongly convex on $\RR^{d} \times \RR^{+,*}$, 
	necessarily $(\bar m,\bar f)$ is the unique solution of $\nabla \Phi_\be(\bar m,\bar f)=0$, which reads
	\eq{
		\choice{
    	  		 \ga \frac{\bar m}{\bar f^\bet} + \bar m-\tilde m=0 \\
       		-\ga \bet\frac{\norm{\bar m}^2}{2\bar f^{\bet+1}} + \bar f - \tilde f = 0.
     	}
	}
	Reformulating these equations leads to the following equivalent conditions
	\eq{
		P_\be(\bar f) = 0
		\qandq
		\bar m = \mu_\be(\bar f).
	}
	This shows that if $P_\be$ as at least a strictly positive real root $f^\star$, it is necessarily unique and that $(\bar f=f^\star,\bar m=\mu_\be(f^\star))$. Otherwise, necessarily $(\bar f = 0, \bar m = 0)$.
\end{proof}

Note that when $\be=p/q \in \QQ$ is a rational number, equation $P_\be(X)=0$ corresponds to finding the root of a polynomial. It can be solved efficiently using a few Newton descent iterations starting from a large enough value of $f$.



\section{Numerical Simulations}

\subsection{Comparison of Proximal Schemes}

This section compares the following algorithms introduced in Section~\ref{sec-proximal}:
\begin{itemize}
	\item[--] Douglas-Rachford (DR in the following) as exposed in Section~\ref{DR-algo}, parameterized with $\al$ and $\ga$ ;
	\item[--] ADDM on the dual (ADDM in the following, which is DR with $\al=1$) parameterized with~$\ga$ ; 
	\item[--] Primal-dual (PD in the following) as exposed in Section~\ref{PD-algo}, parameterized with~$\si$ and $\tau$.
\end{itemize}
Note that this ADMM formulation is related to the ALG2 method introduced in~\cite{Benamou2000}, but is computed over a staggered grid. For the DR algorithm, we first compared the $4$ possible implementations previously described. It appears in our experiments that A-DR and A-DR' (resp. S-DR and S-DR') have almost the same behavior. 

The first comparison is done on a simple example with two 2-D isotropic Gaussian distributions $(f^0,f^1)$ with the same variance. In the continuous case, the solution is known to be a translation between the mean of the Gaussians. The spatial domain is here of dimension $d=2$ and it is discretized on a grid with $N=M=32$ points for both each dimension. The temporal  discretization has also been fixed to $P=32$. We first compute an (almost) exact reference solution $(m^\star,f^\star)$ of the discrete problem with $10^6$ iterations of the DR. The obtained transported mass $f^\star(\cdot,t)$ is illustrated in Figure \ref{fig:data_bump}. Regarding the computation time, with a bi-processor system Intel Core i7 with 2.4 GHz, $1000$ iterations are done in $45$ seconds for a $32^3$ domain with our Matlab implementation.

\begin{figure}[!ht]
\begin{center}
\begin{tabular}{ccccc}
\includegraphics[width=2.5cm]{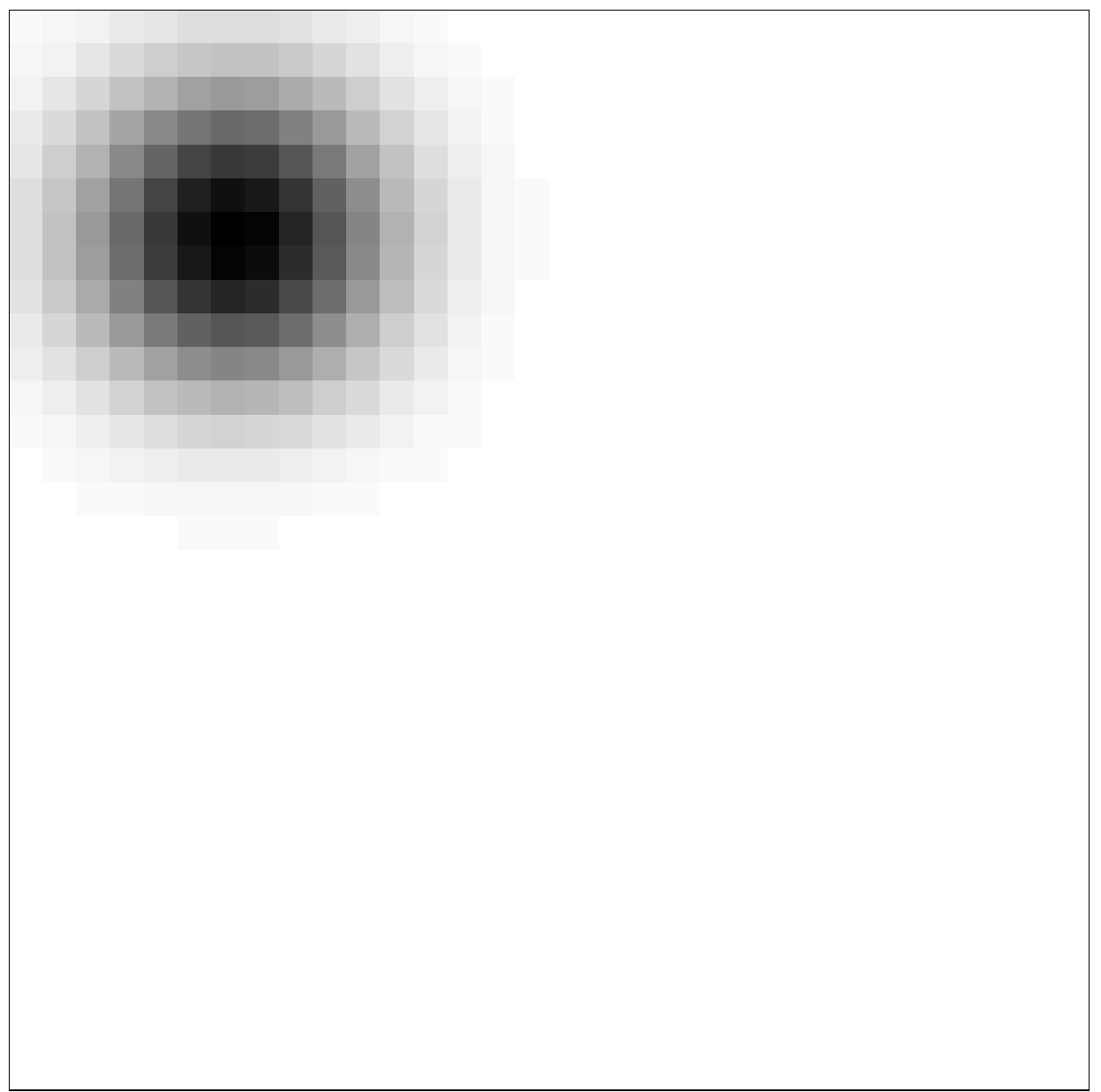}&
\includegraphics[width=2.5cm]{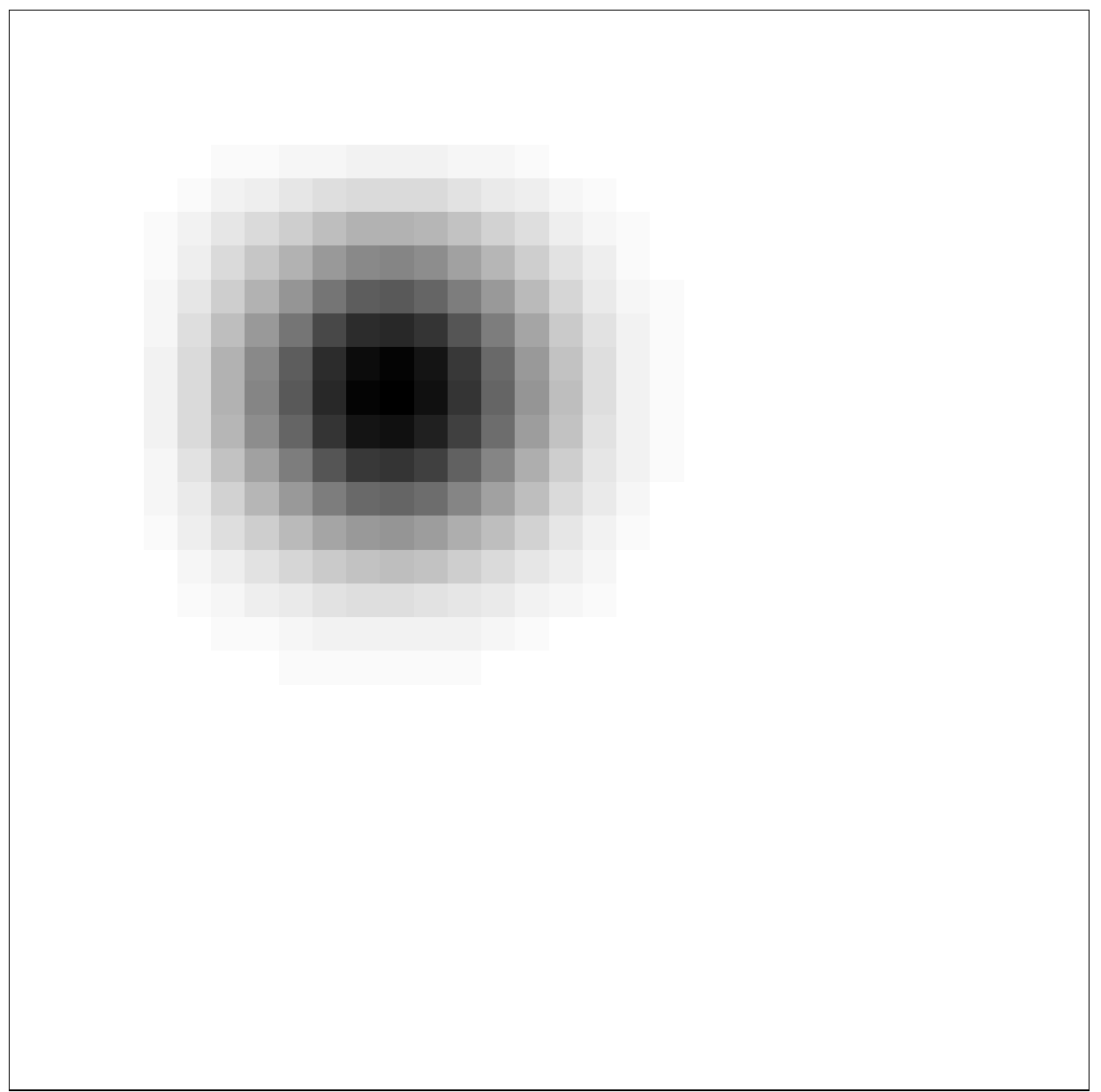}&
\includegraphics[width=2.5cm]{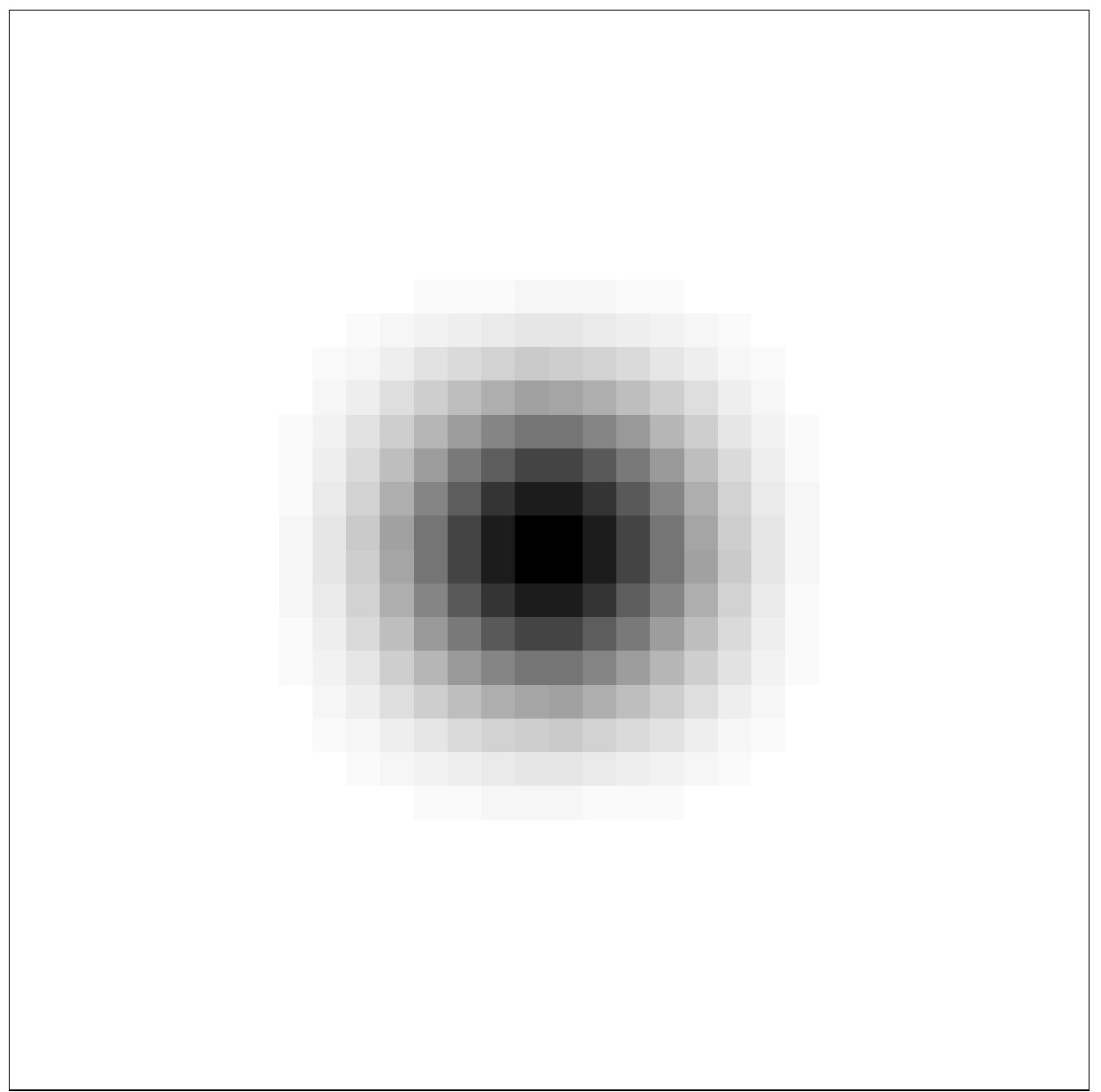}&
\includegraphics[width=2.5cm]{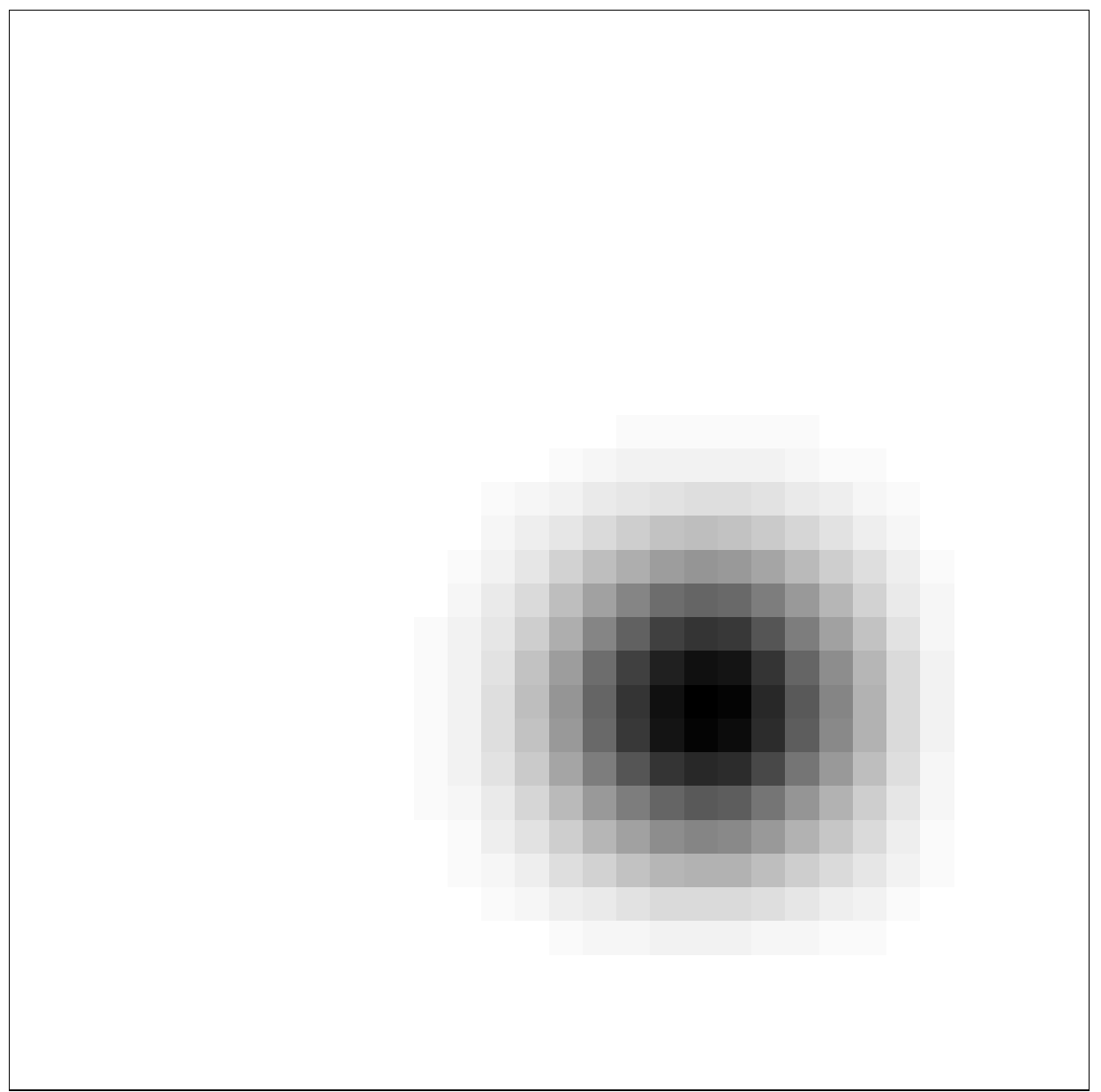}&
\includegraphics[width=2.5cm]{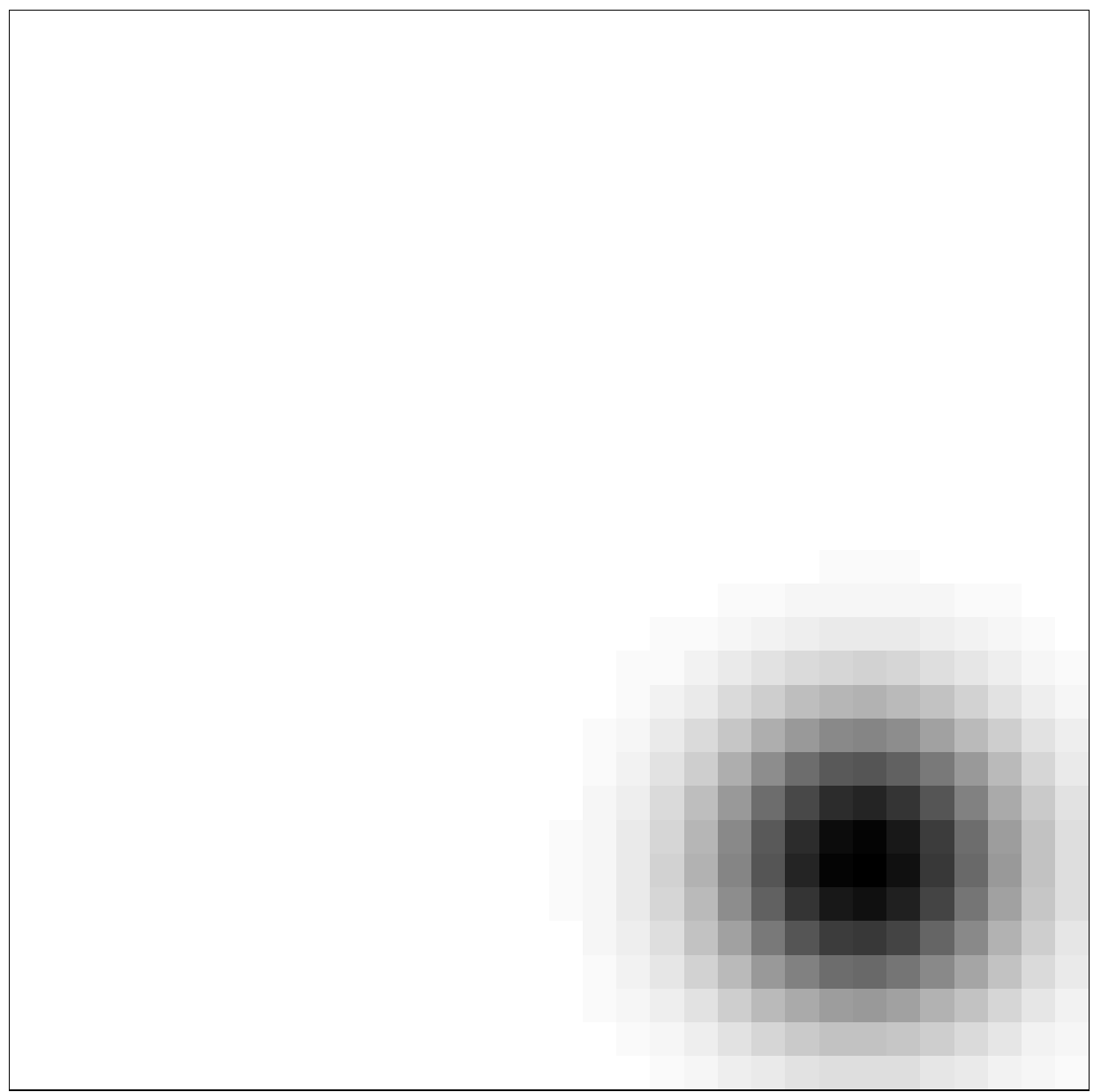}\\
$t=0$&$t=1/4$&$t=1/2$&$t=3/4$&$t=1$ 
\end{tabular}
\caption{\label{fig:data_bump} 
Display of $f^\star(\cdot,t)$ for several value of $t$ (note that for $t=0$ and $t=1$, this corresponds to $f^0$ and $f^1$). The grayscale values are linearly interpolating from black to white between 0 and the maximum value of $f^\star$. 
}
\end{center}
\vspace{3mm}
\end{figure}

For each algorithm, we perform an exhaustive search of the best possible set of parameters. These optimal parameters are those minimizing $\norm{(m^\star,f^\star) - (\iter{m},\iter{f})}$, the $\ell^2$ distance between $f^\star$ and the output of the algorithm after $\ell=500$ iterations. The optimal parameters for this data set are:  $\ga=1/80$ for ADMM on the dual, $(\ga=1/75,\al=1.998)$ for DR and $\si=85$ for PD. For PD, we found that simply setting $\tau=\frac{0.99}{\si \norm{\interp}^2}$ leads to almost optimal convergence rate in our tests, so we use this rule to only introduce a single parameter $\si$. Notice that this parameter choice is within the range of parameters $\sigma \tau \norm{\interp}^2<1$ that guaranties convergence of the PD method. Figure~\ref{fig:comp_bump} displays, for this optimal choice of parameters, the evolution of the cost function value as well as the  convergence on the staggered grid toward $(m^\star,f^\star)$ as a function of the iteration number $\ell$. 

One can observe that the quality of the approximation can not easily be deduced from the cost function  evolution alone since the functional is very flat. Indeed, an almost minimal value of the function  is reached by all the algorithms after roughly $10^3$ iterations, whereas the $\ell^2$ distance to the reference solution continues to decrease almost linearly in log-log scale. The last plot of the Figure \ref{fig:comp_bump} shows that asymptotically, all methods tend to satisfy the positivity constraint on the staggered grid at the same rate.

\begin{figure}[!ht]
\begin{center}
\begin{tabular}{cc}
\includegraphics[trim=25 5 45 20,clip,width=0.45\textwidth]{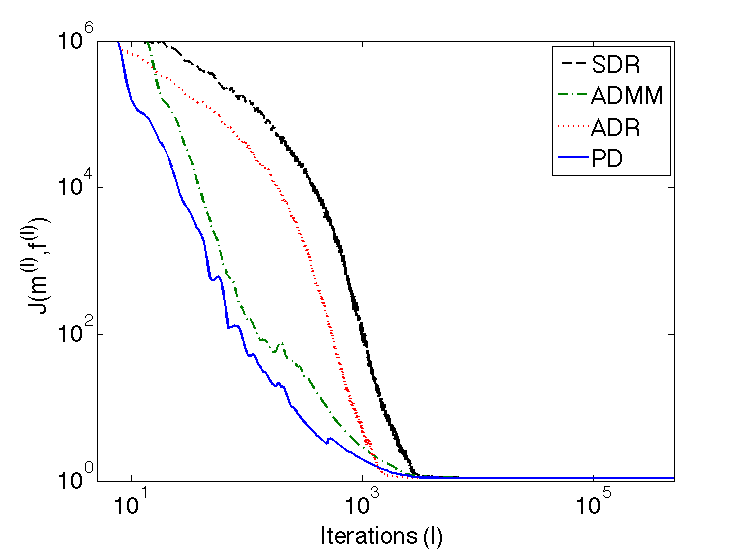}&\hspace{-0.15cm}
\includegraphics[trim=25 5 45 20,clip,width=0.45\textwidth]{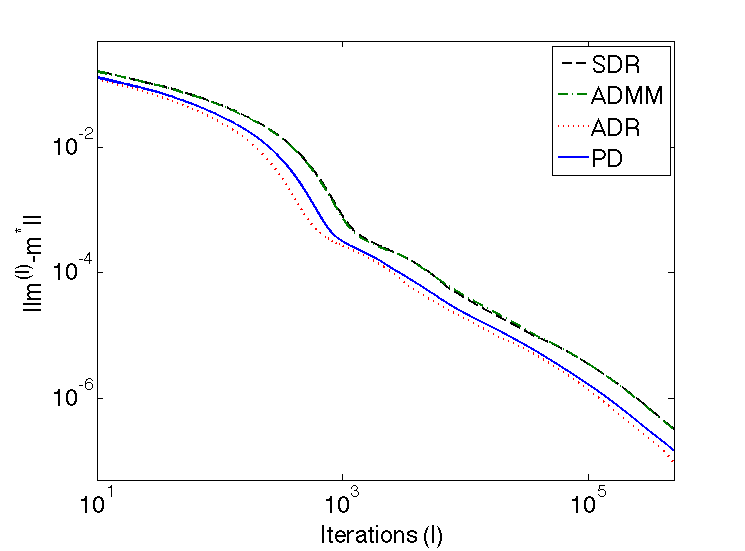}\vspace{0.1cm}\\
$\Jfunc( \iter{m}, \iter{f} )$&\hspace{-0.15cm}$\norm{m^\star-\iter{m}}$\\
\includegraphics[trim=25 5 45 20,clip,width=0.45\textwidth]{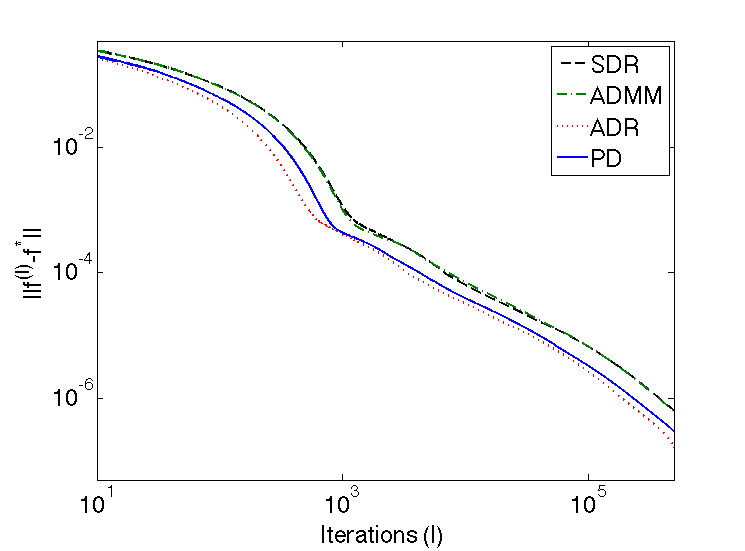}&\hspace{-0.15cm}
\includegraphics[trim=30 5 45 20,clip,width=0.45\textwidth]{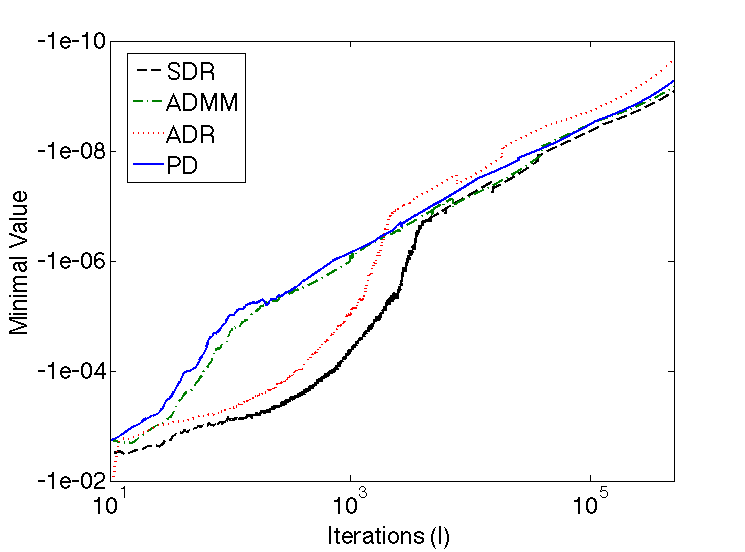}\vspace{0.1cm}\\
$ \norm{f^\star-\iter{f}}$ &\hspace{-0.15cm}Mininimum value of $\iter{f}$
\end{tabular}
\caption{\label{fig:comp_bump} 
	At each iteration $\ell$ on the staggered grid, we plot in the log-log scale the value of the cost function $\Jfunc$, the distance between the reference solution $(m^\star,f^\star)$ and the estimation $(\iter{m},\iter{f})$ and the current minimum value of $\iter{f}$ for the different proximal splitting algorithms with the best found parameters. }
\end{center}
\end{figure}

When comparing the three approaches, A-DR and A-PD shows the fastest convergence rate to the reference solution and then, the S-DR algorithm performs equally good as ADMM. This shows the advantage of introducing the $\al$ parameter, while symmetrizing the DR does not speed-up the convergence.  Notice also that the computational cost is smaller for PD, as it takes $0.13$s for one PD iteration and $0.2s$ for one DR or ADMM iteration for this example.  Note that these convergence results are  related to this specific example, but they illustrate the general behaviour of the different algorithms. 



Finally, Figure~\ref{fig:indic} shows an experiment in the context of vanishing and irregular densities.  This figure shows the geodesic, computed with the PD algorithm, between two characteristic functions of two connected sets, one being convex. Note that the geodesic is not composed of characteristic functions of sets, which is to be expected. This shows the ability of our methods to cope with vanishing densities.

\begin{figure}[!ht]
\begin{center}
\begin{tabular}{@{}c@{}c@{}c@{}c@{}c@{}}
\includegraphics[width=2.9cm]{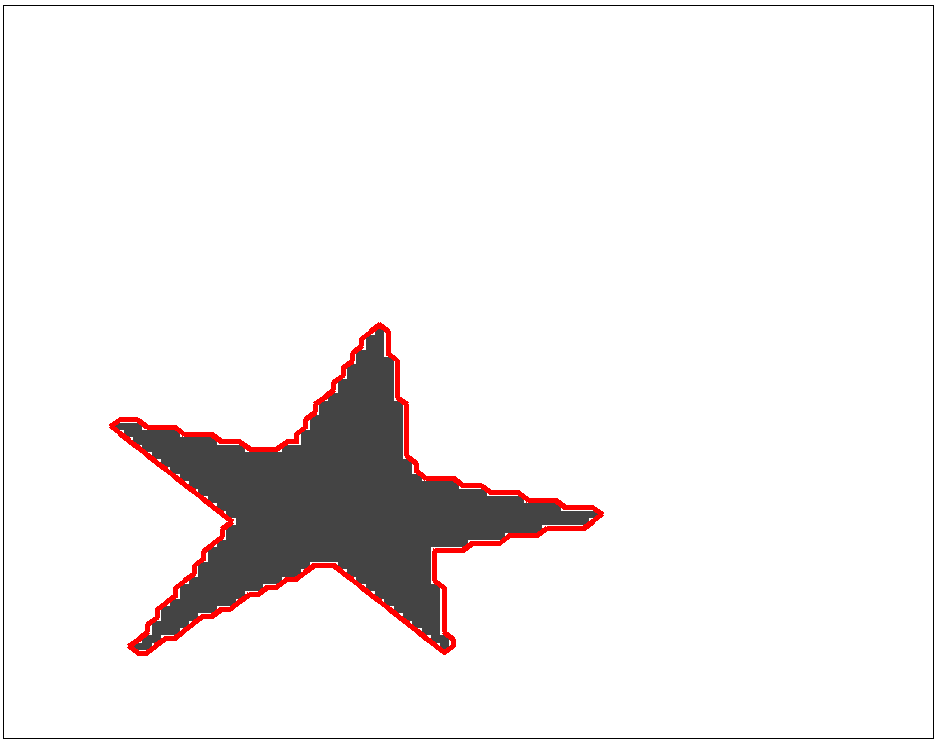}&
\includegraphics[width=2.9cm]{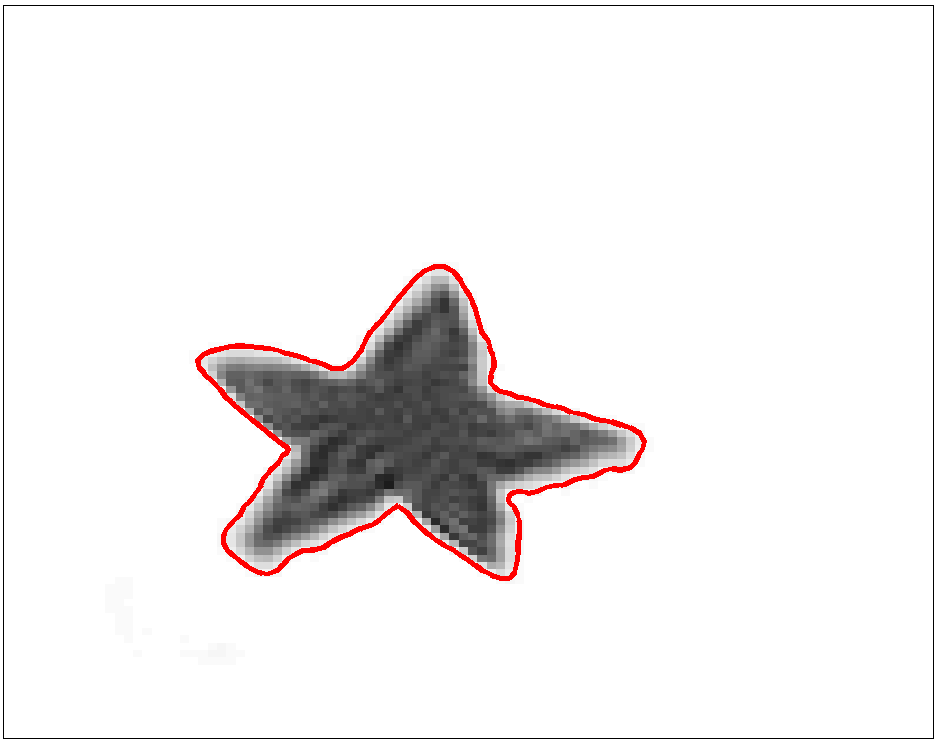}&
\includegraphics[width=2.9cm]{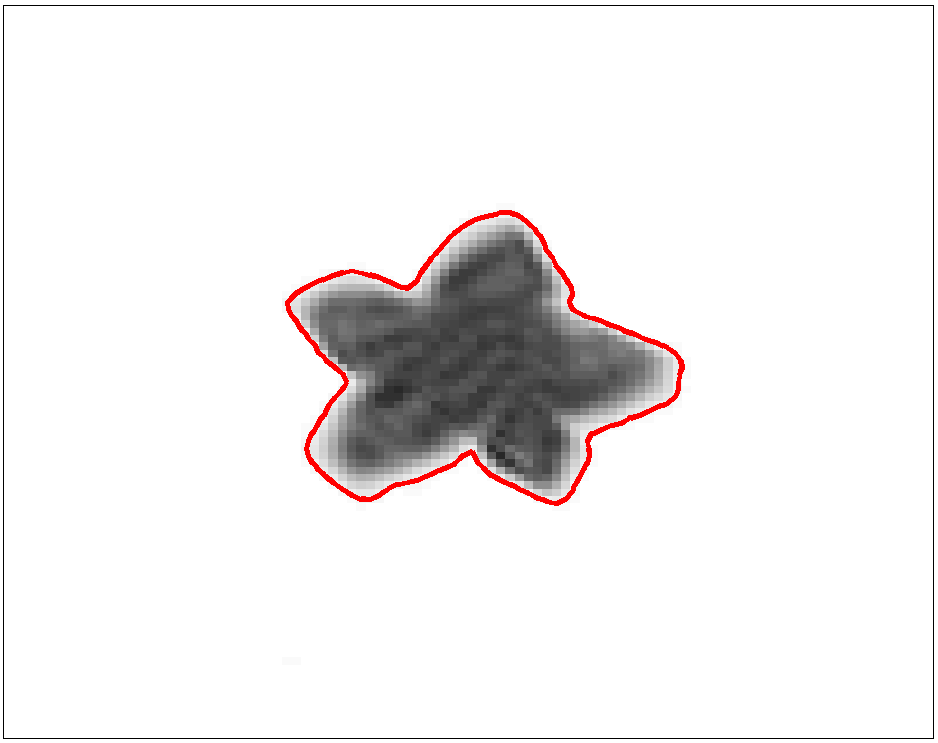}&
\includegraphics[width=2.9cm]{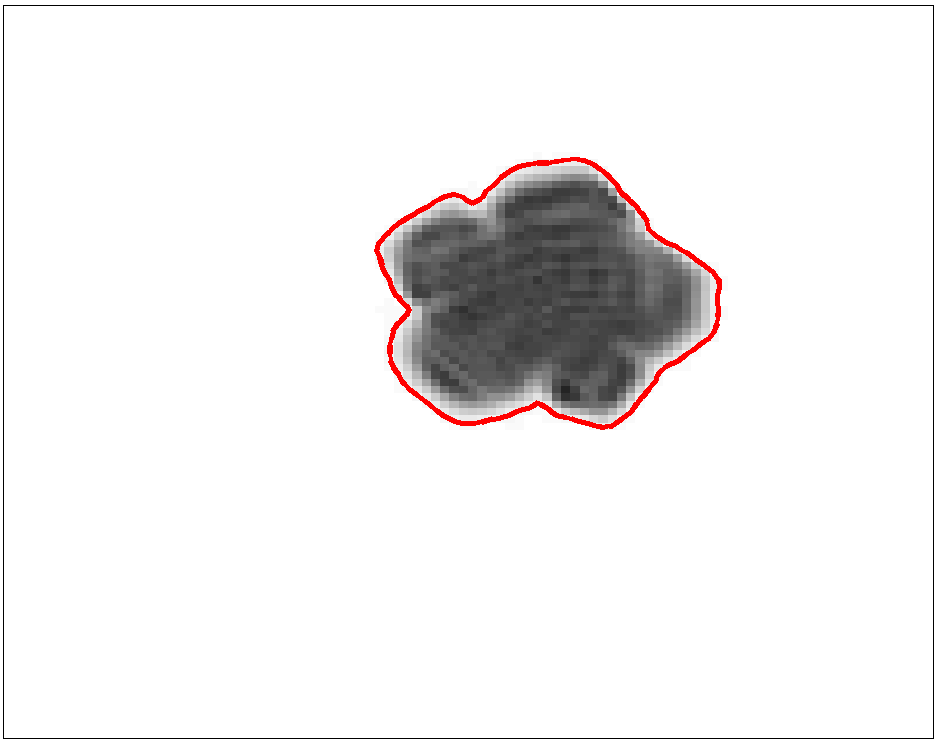}&
\includegraphics[width=2.9cm]{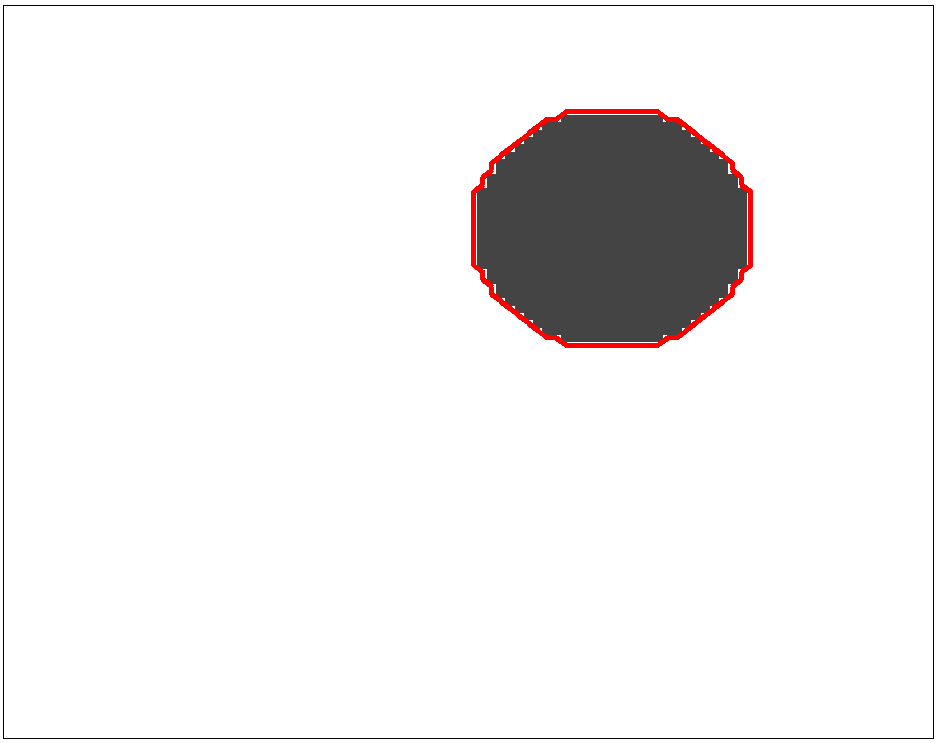}\\
$t=0$&
$t=1/4$&
$t=1/2$&
$t=3/4$&
$t=1$
\end{tabular}
\caption{\label{fig:indic} Transport between characteristic functions. Evolution of $f^\star(\cdot,t)$ for several values of $t$. The red curve denotes the boundary the area with positive density. }
\end{center}
\end{figure}

\subsection{Interpolation Between $L^2$-Wasserstein and $H^{-1}$}

We first apply the PD algorithm for different values of $\beta$ on the bump example introduced in the previous section. The results are presented in the Figure~\ref{fig:generalized_bump}, which shows the level-lines of the estimated densities $\iter{f}(\cdot,t)$ for $\ell=1000$ iterations. It shows the evolution of the solution between a linear interpolation of the densities ($\beta=0$) and a displacement interpolation with transport ($\beta=1$).

\newcommand{\sidecap}[1]{ {\begin{sideways}\parbox{1.4cm}{\centering #1}\end{sideways}} }
\newcommand{\myfigBeta}[1]{\includegraphics[width=.105\linewidth]{images/bump_beta/bump_beta_#1}}

\begin{figure}[!ht]
\begin{center}
\begin{tabular}{@{}c@{}c@{}c@{}c@{}c@{}c@{}c@{}c@{}c@{}c@{}}
\sidecap{$\beta=0$ } &
\myfigBeta{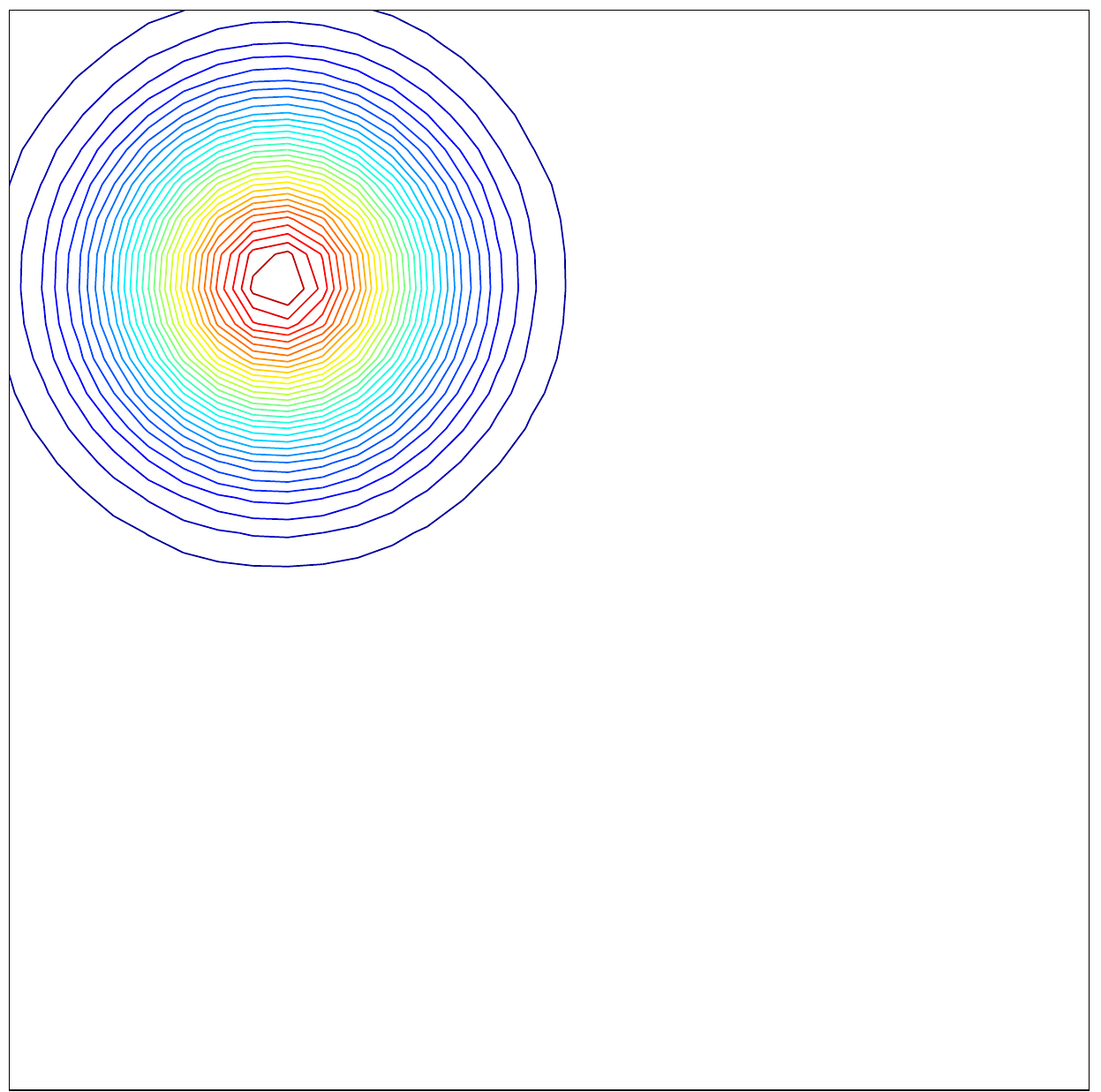}&
\myfigBeta{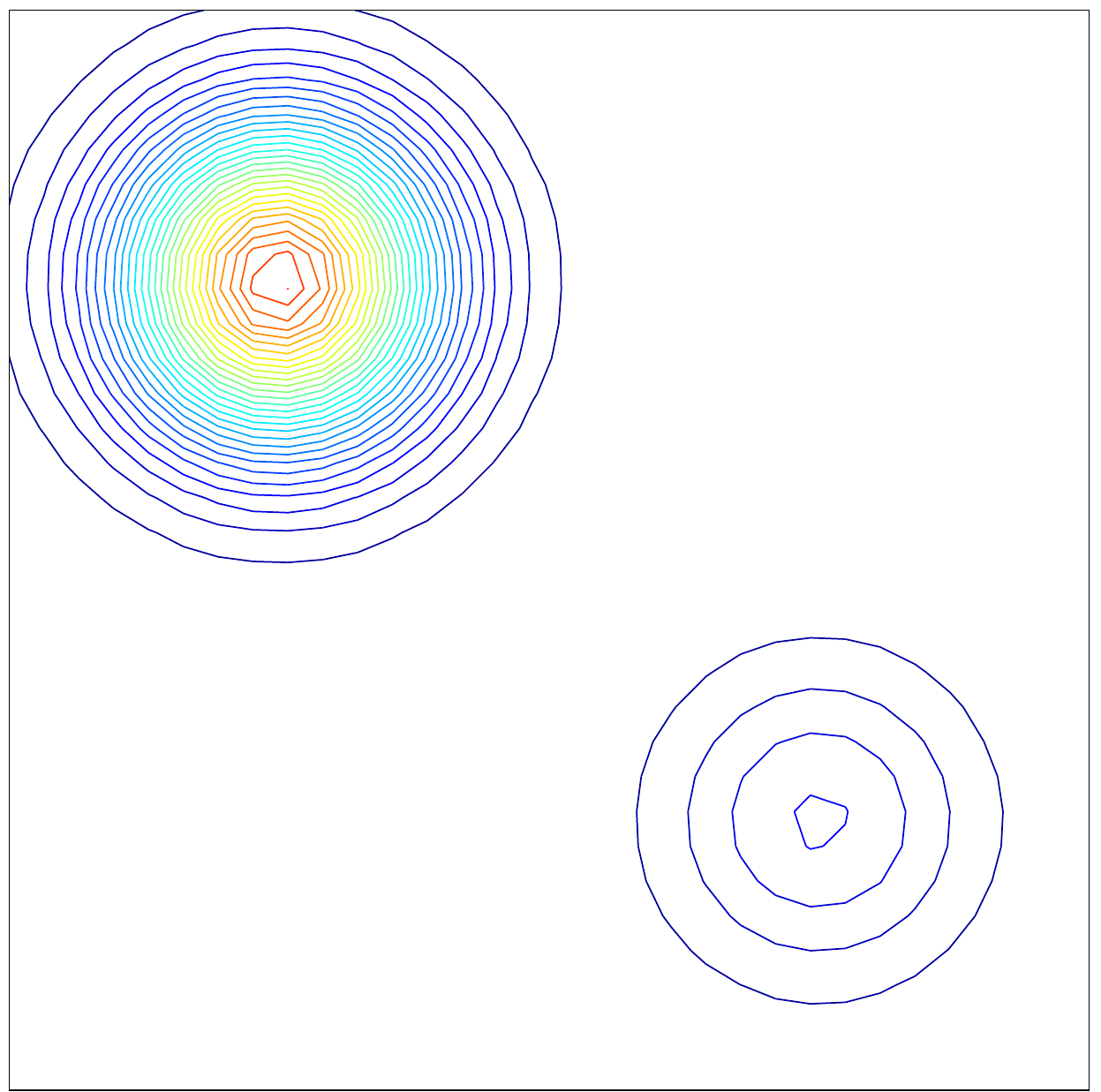}&
\myfigBeta{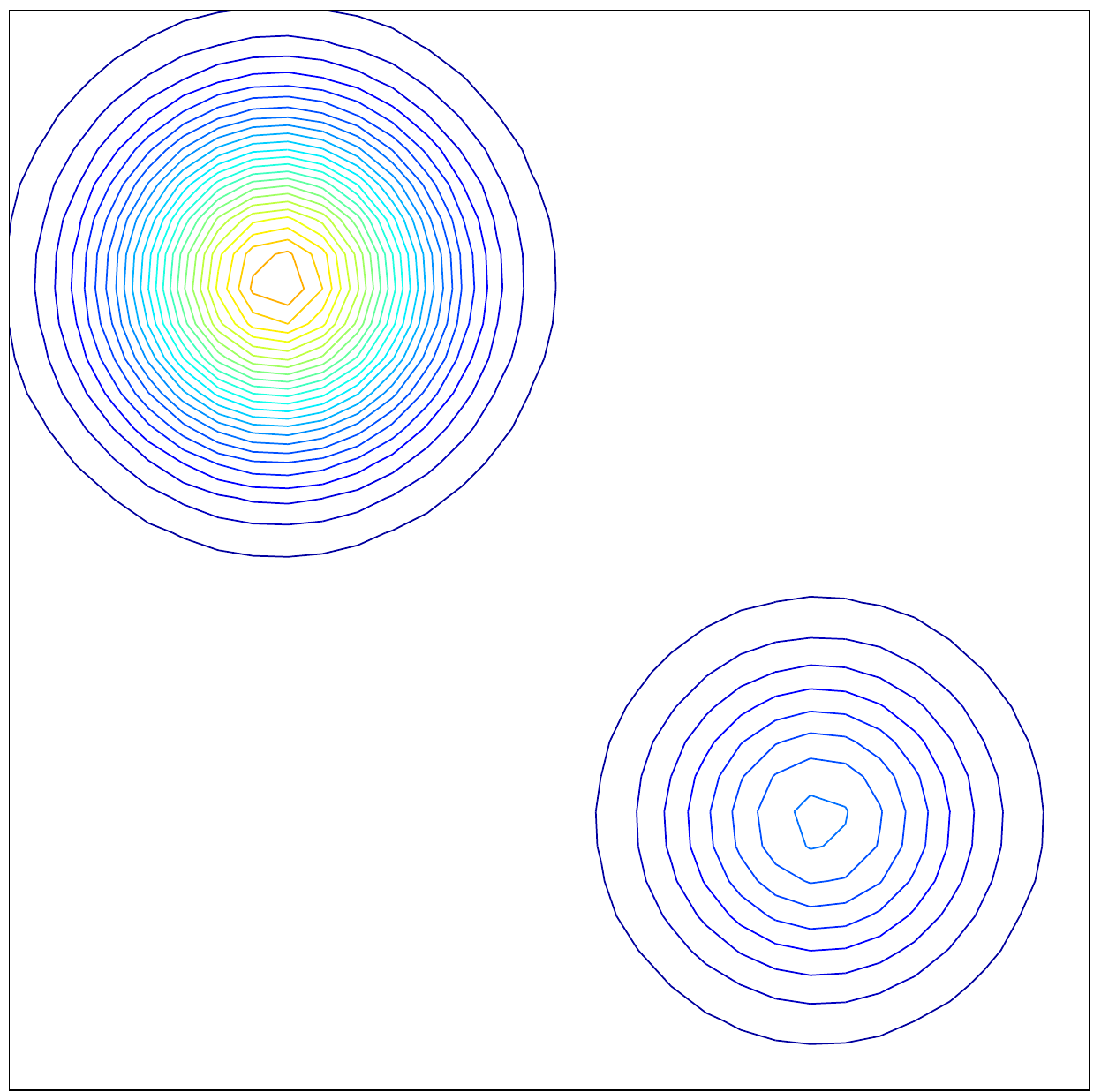}&
\myfigBeta{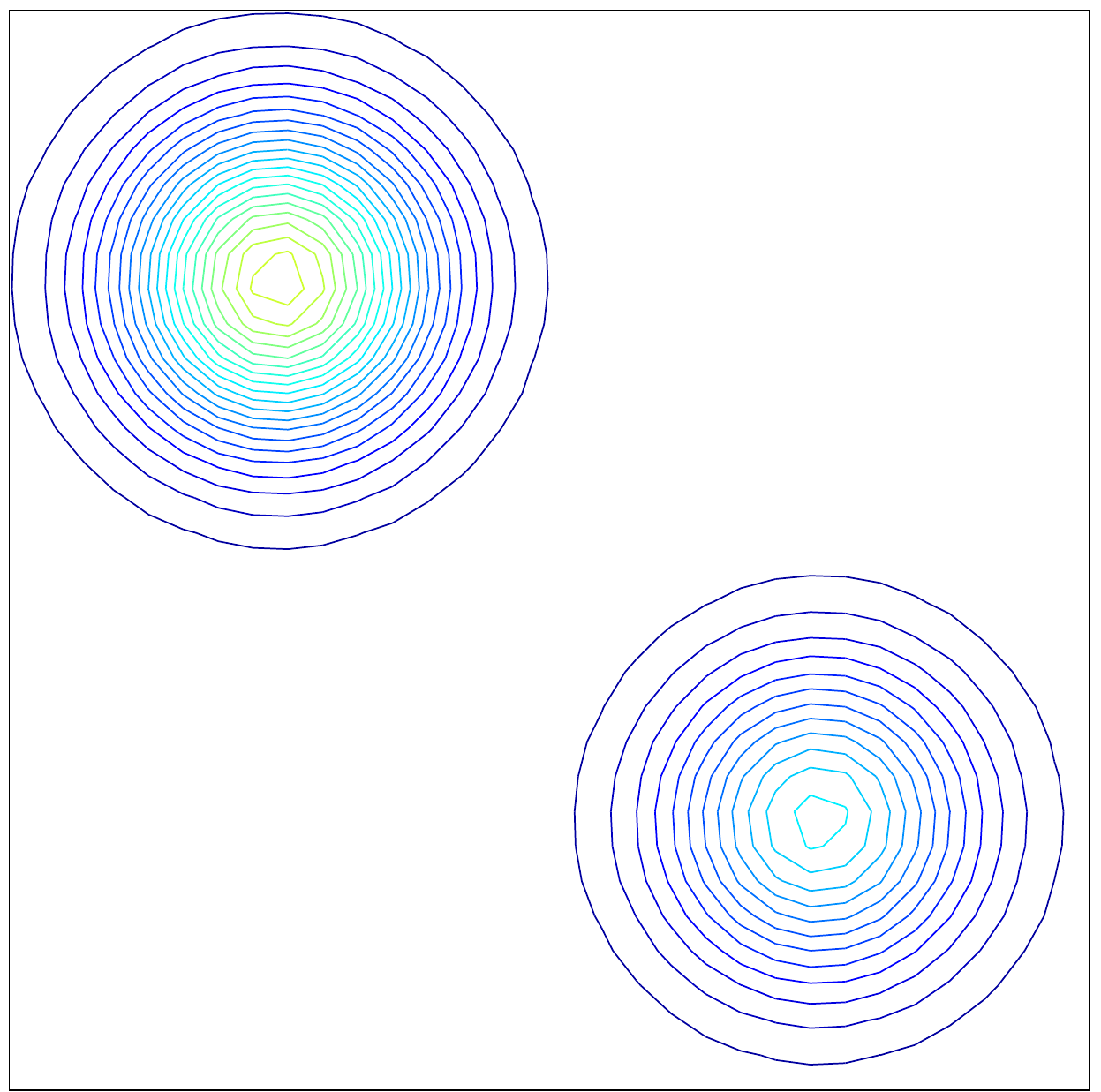}&
\myfigBeta{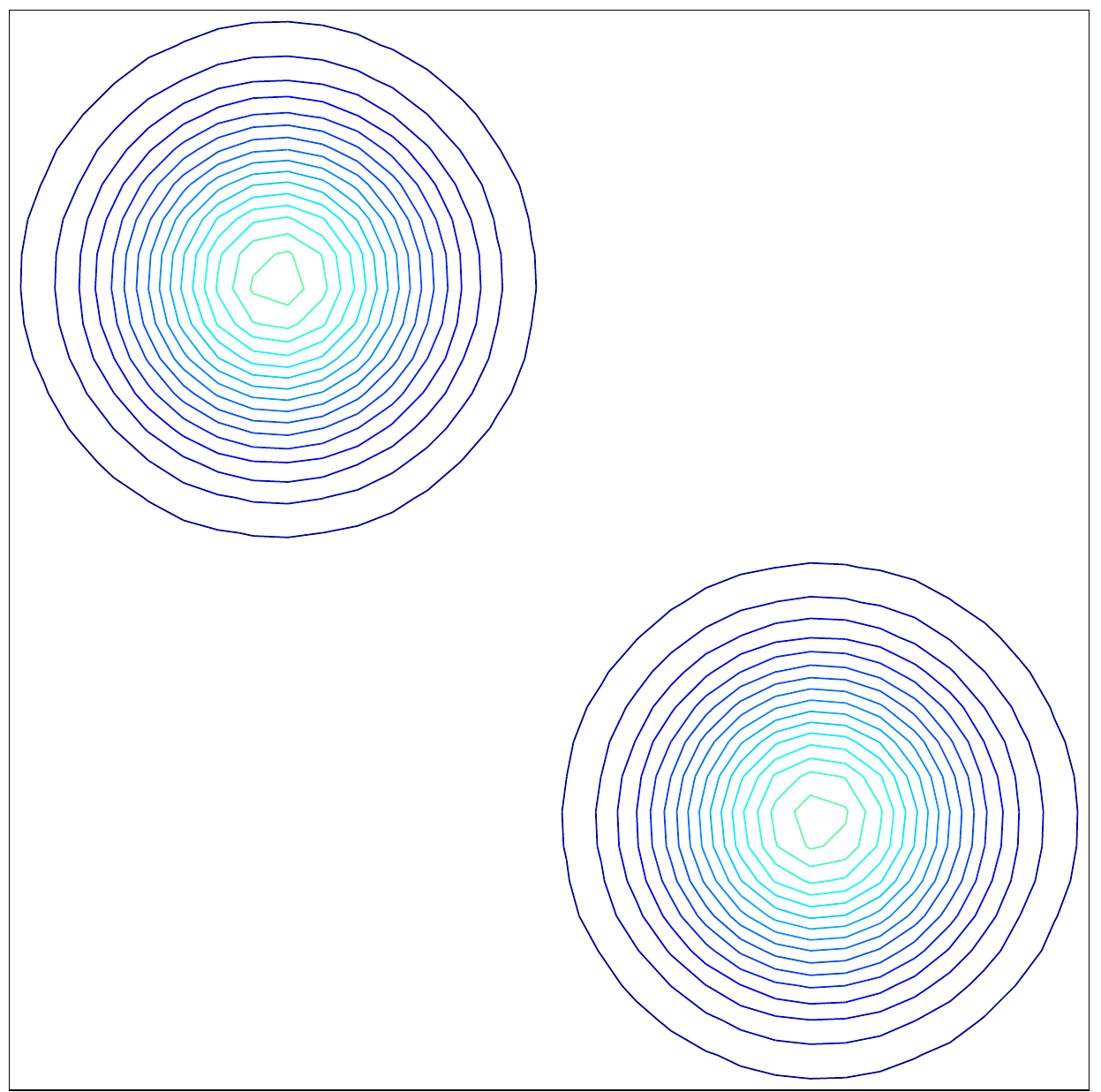}&
\myfigBeta{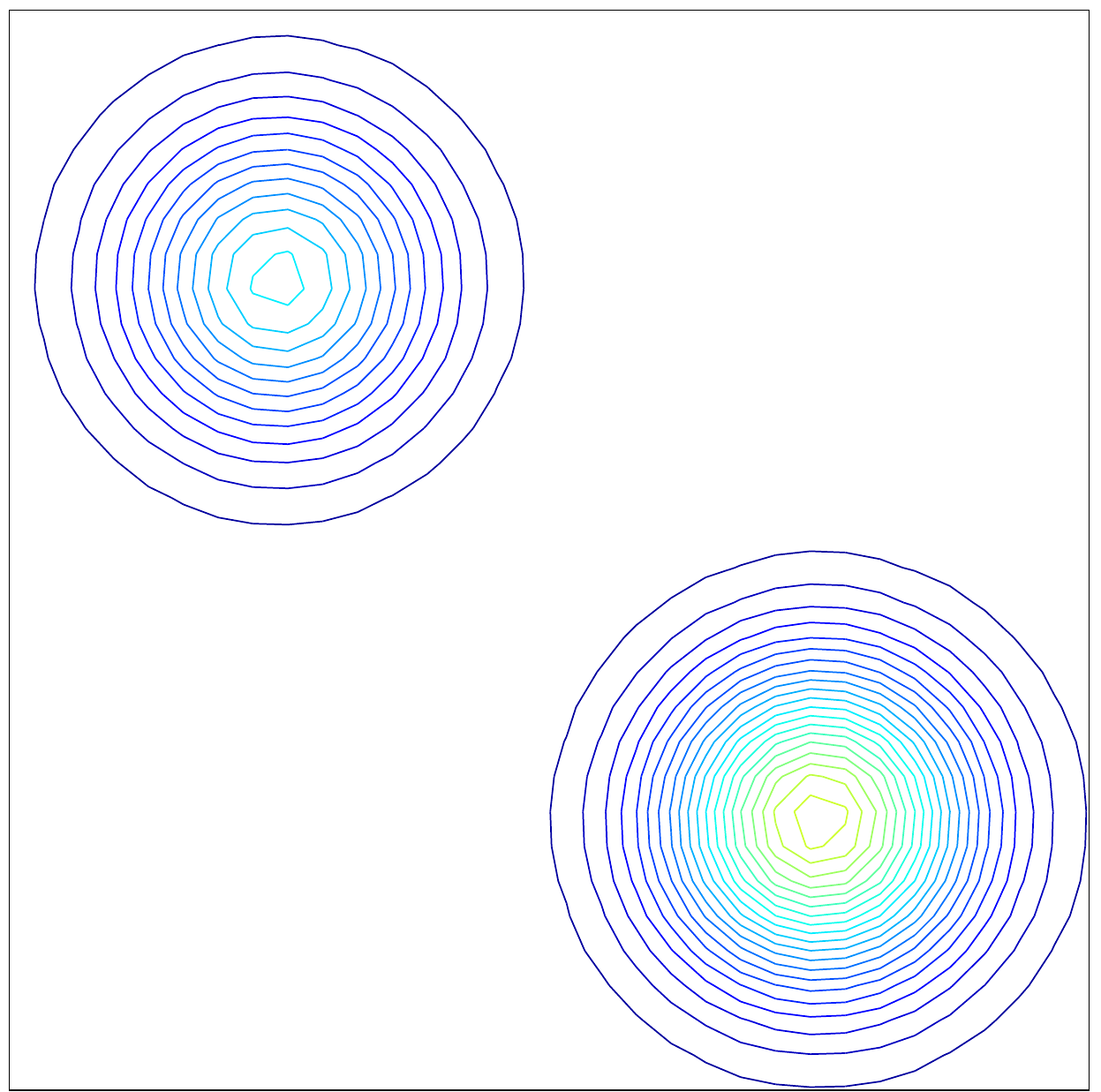}&
\myfigBeta{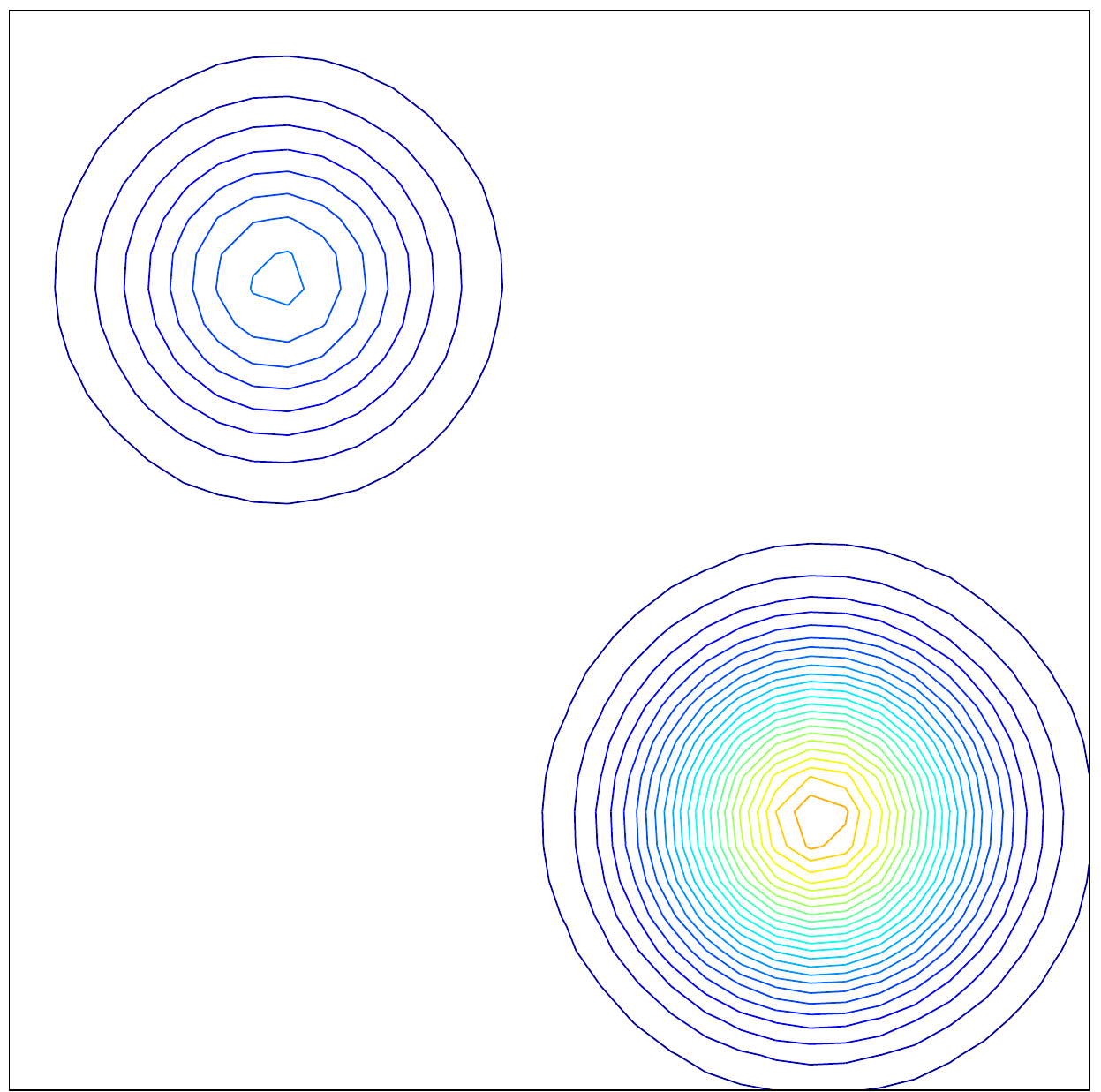}&
\myfigBeta{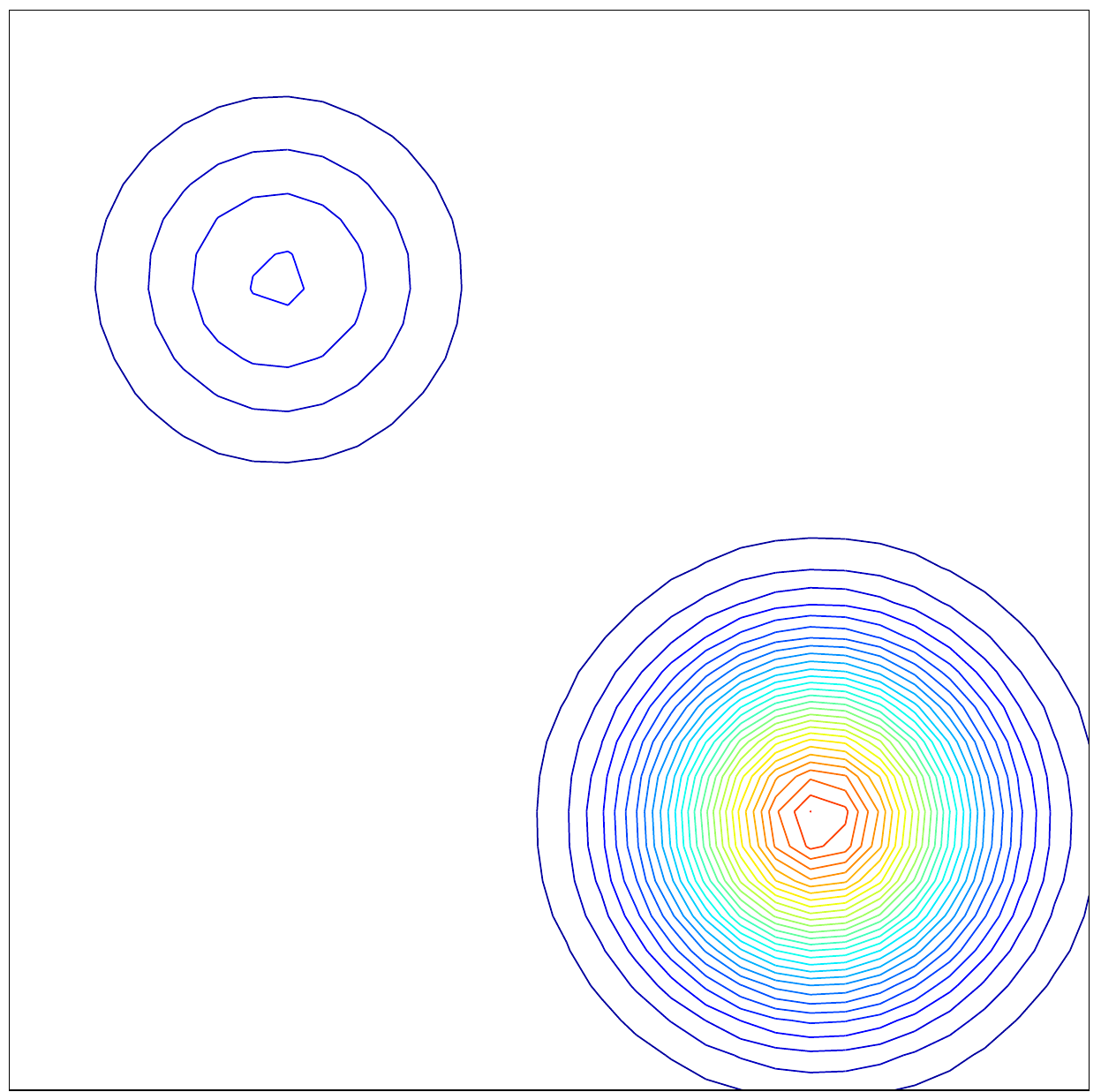}&
\myfigBeta{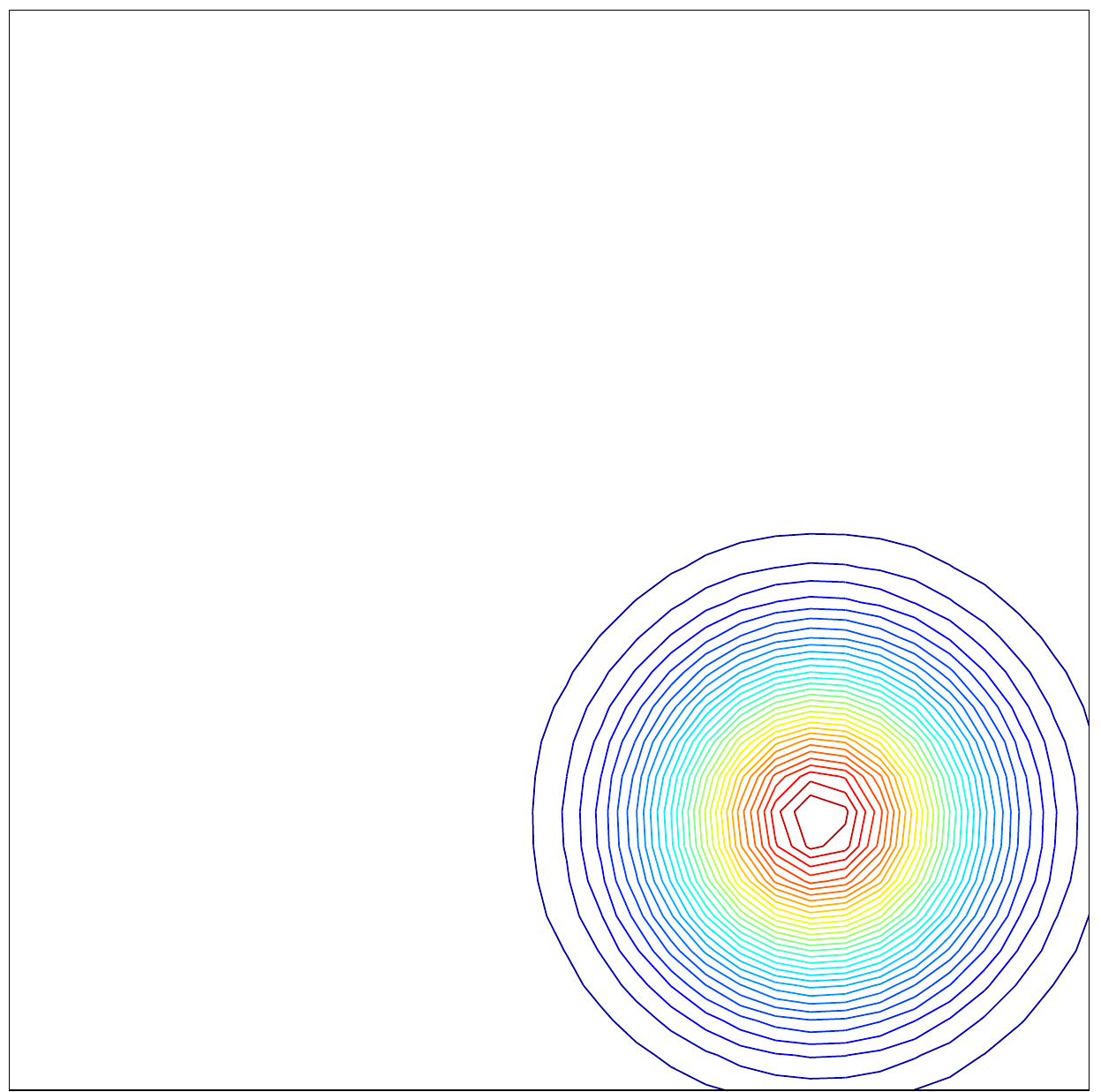}\\
\sidecap{$\beta=1/4$ } &
\myfigBeta{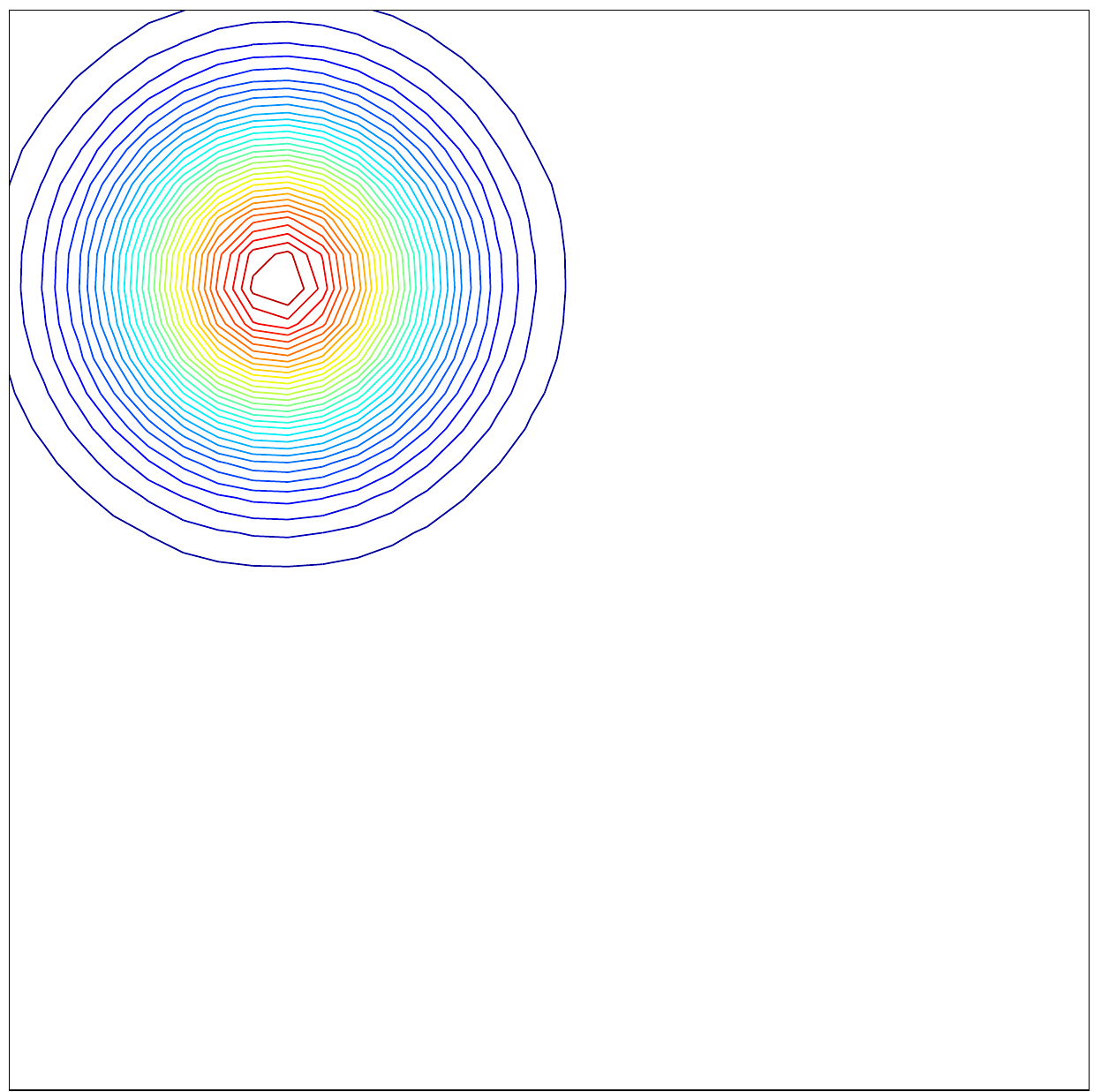}&
\myfigBeta{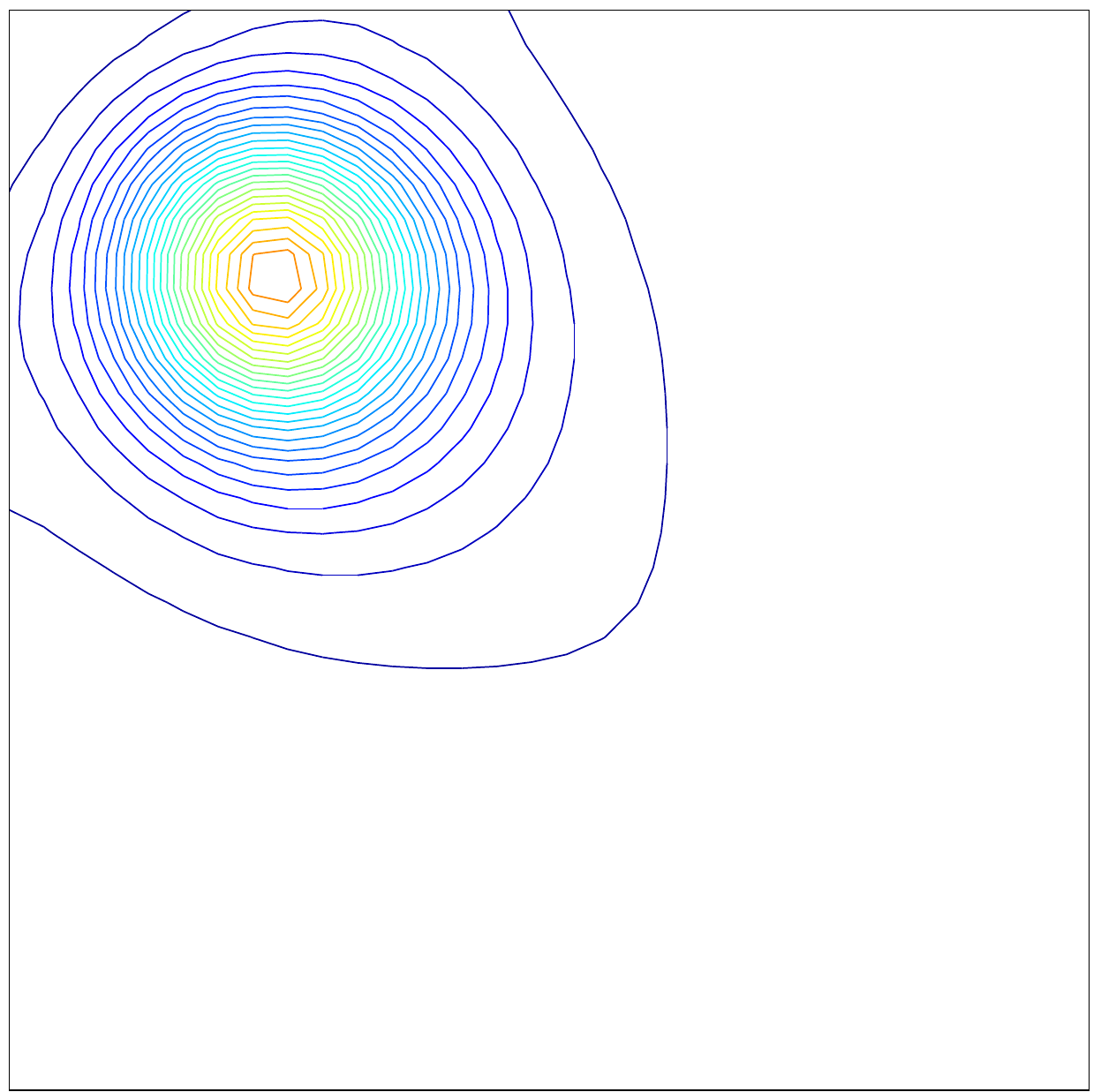}&
\myfigBeta{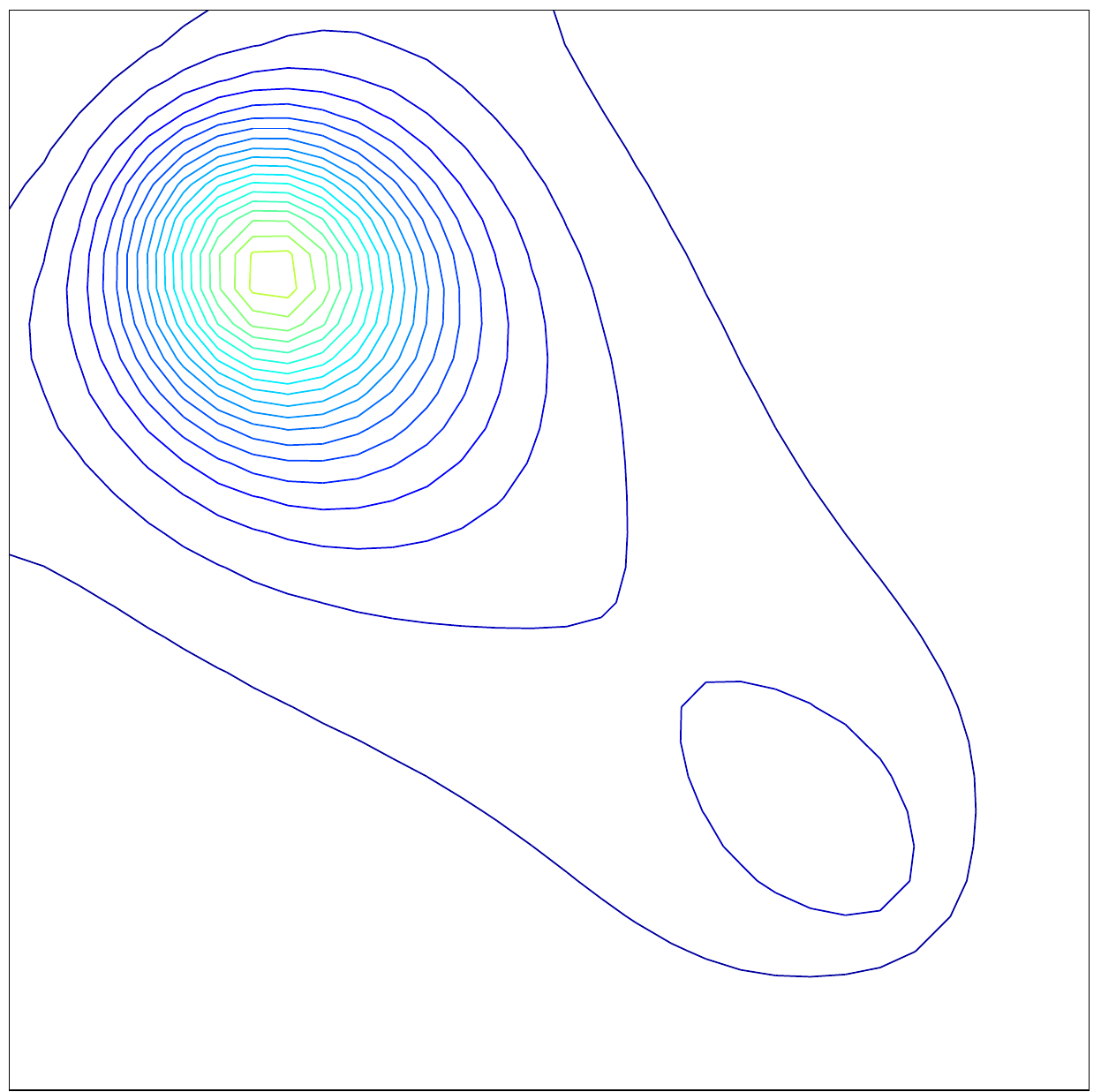}&
\myfigBeta{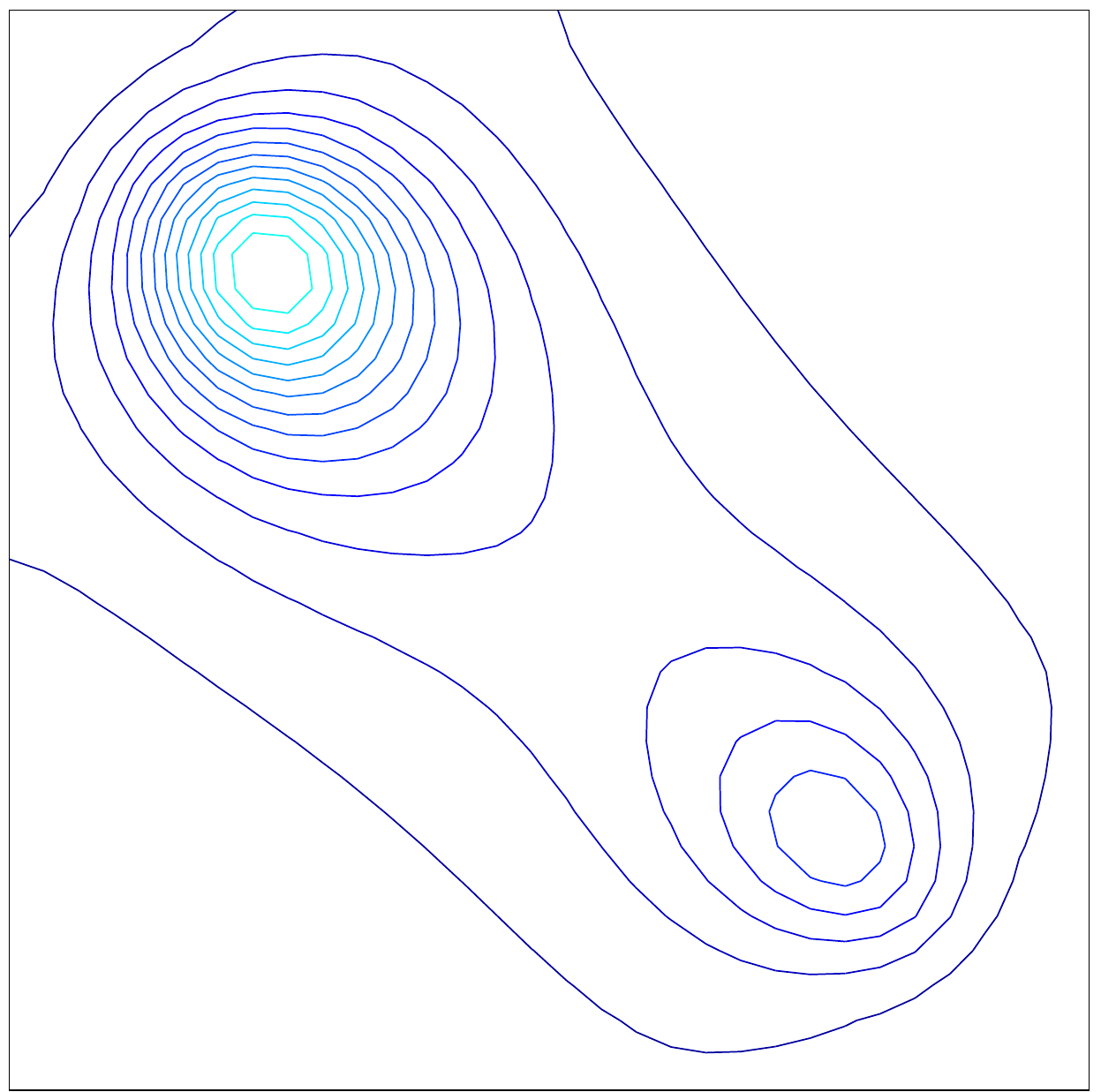}&
\myfigBeta{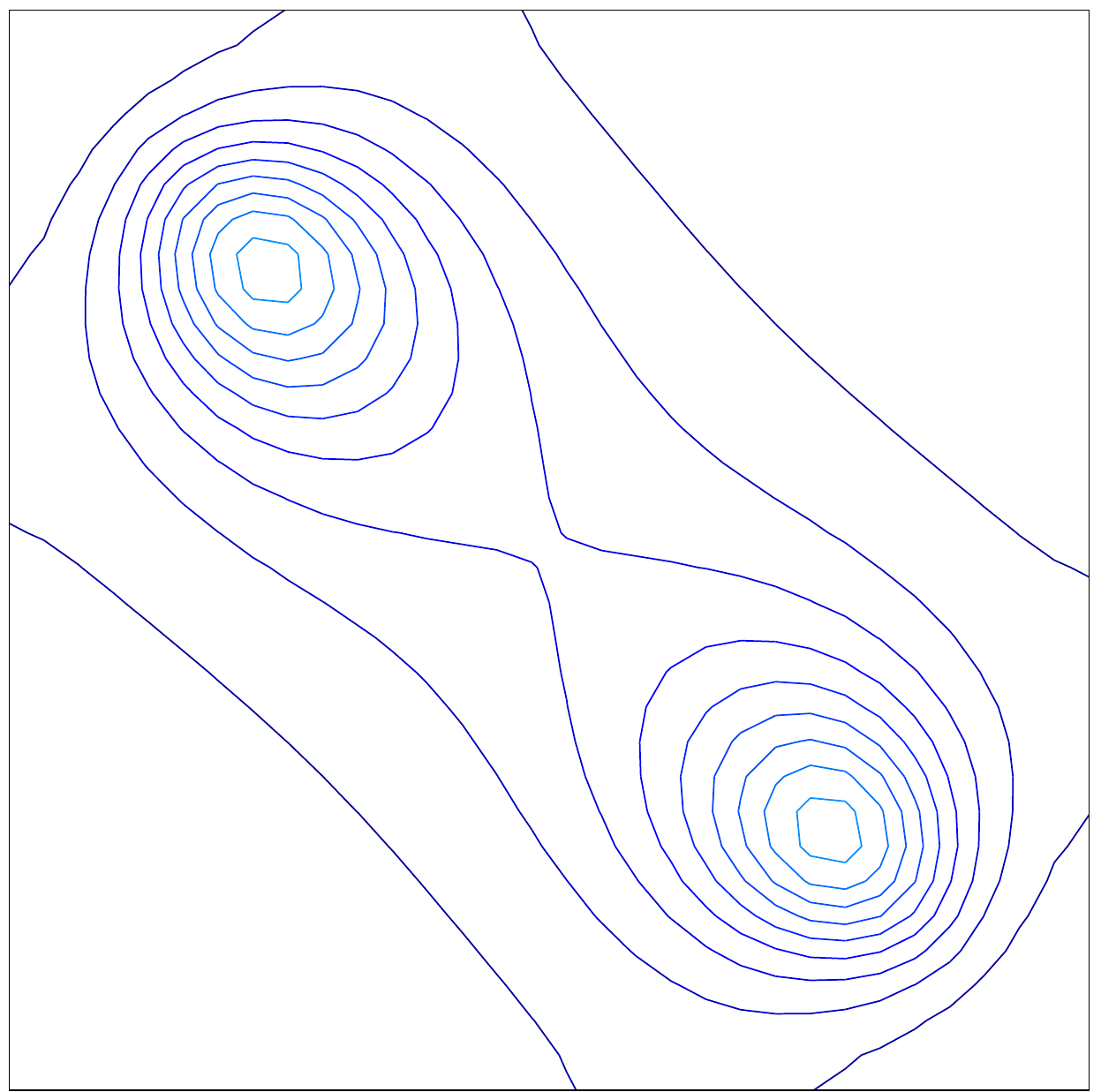}&
\myfigBeta{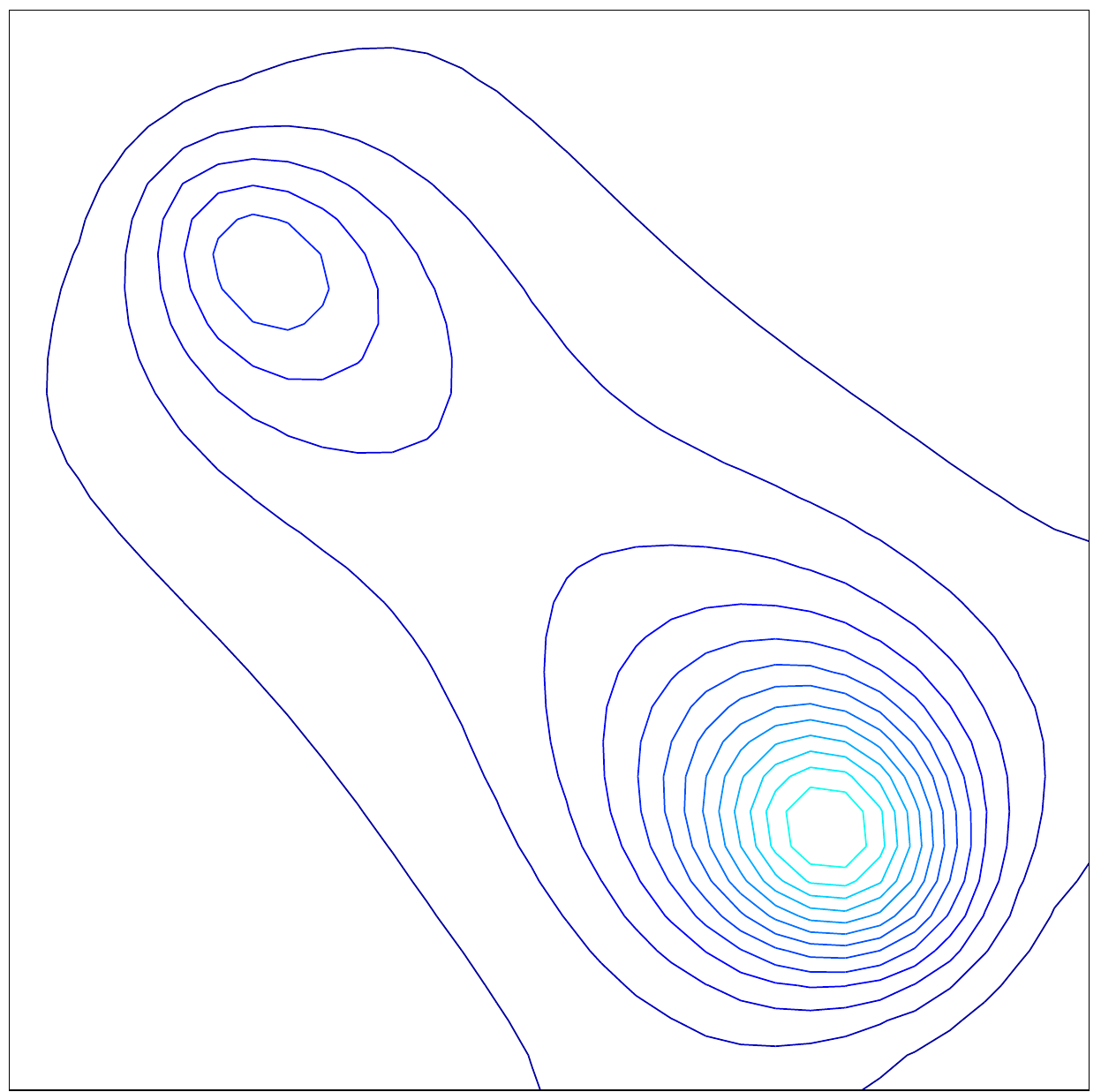}&
\myfigBeta{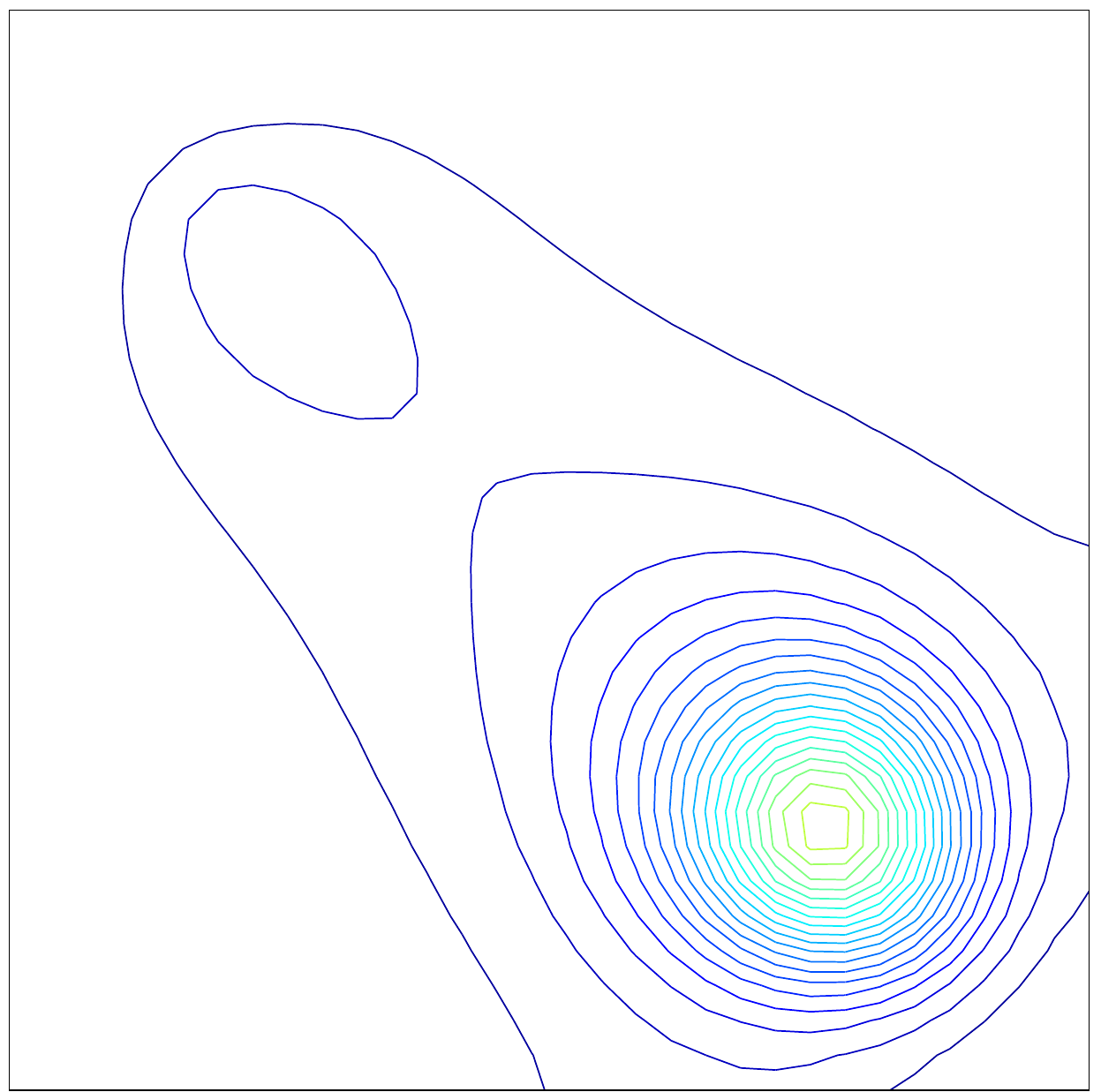}&
\myfigBeta{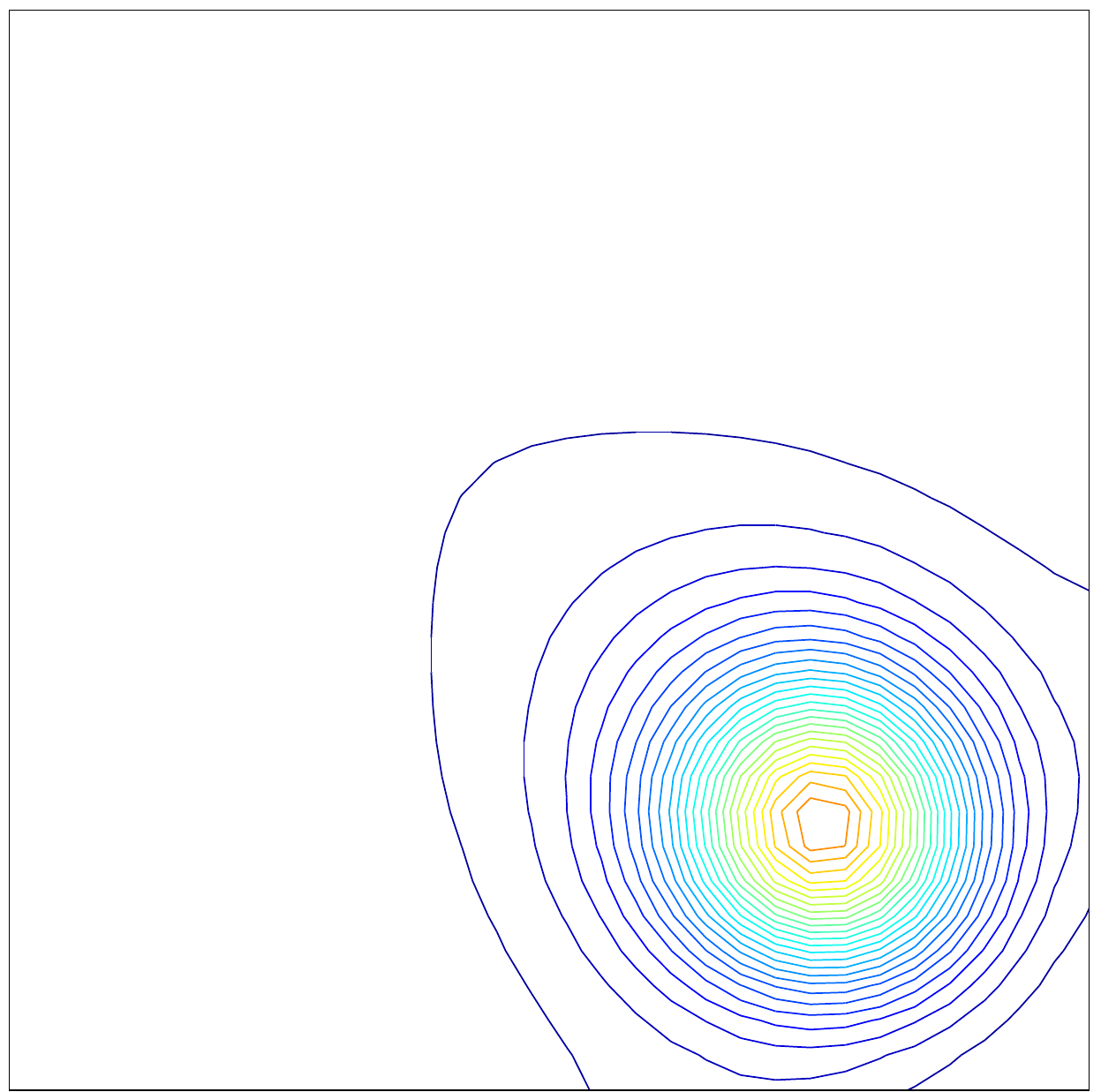}&
\myfigBeta{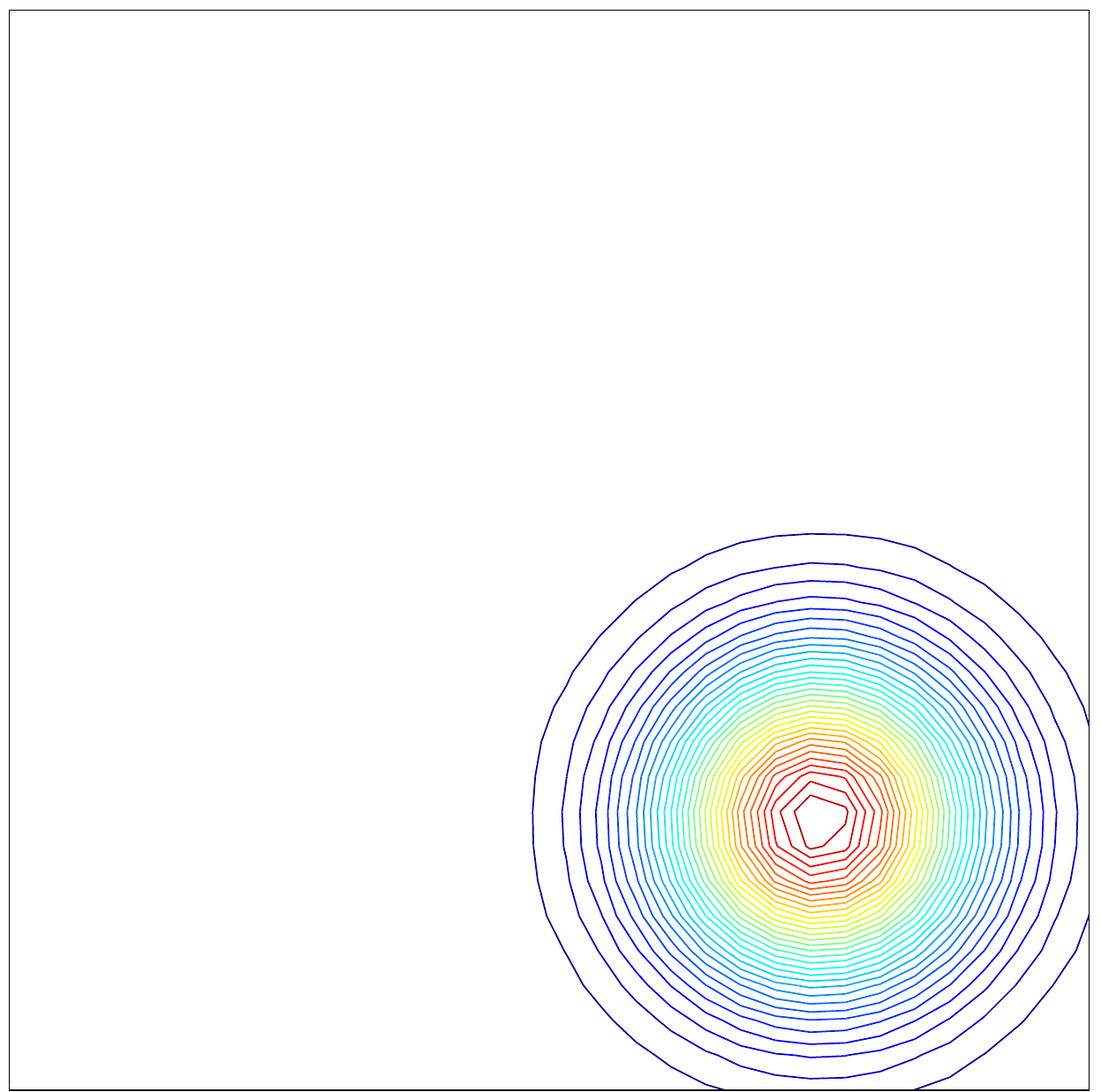}\\
\sidecap{$\beta=1/2$ } &
\myfigBeta{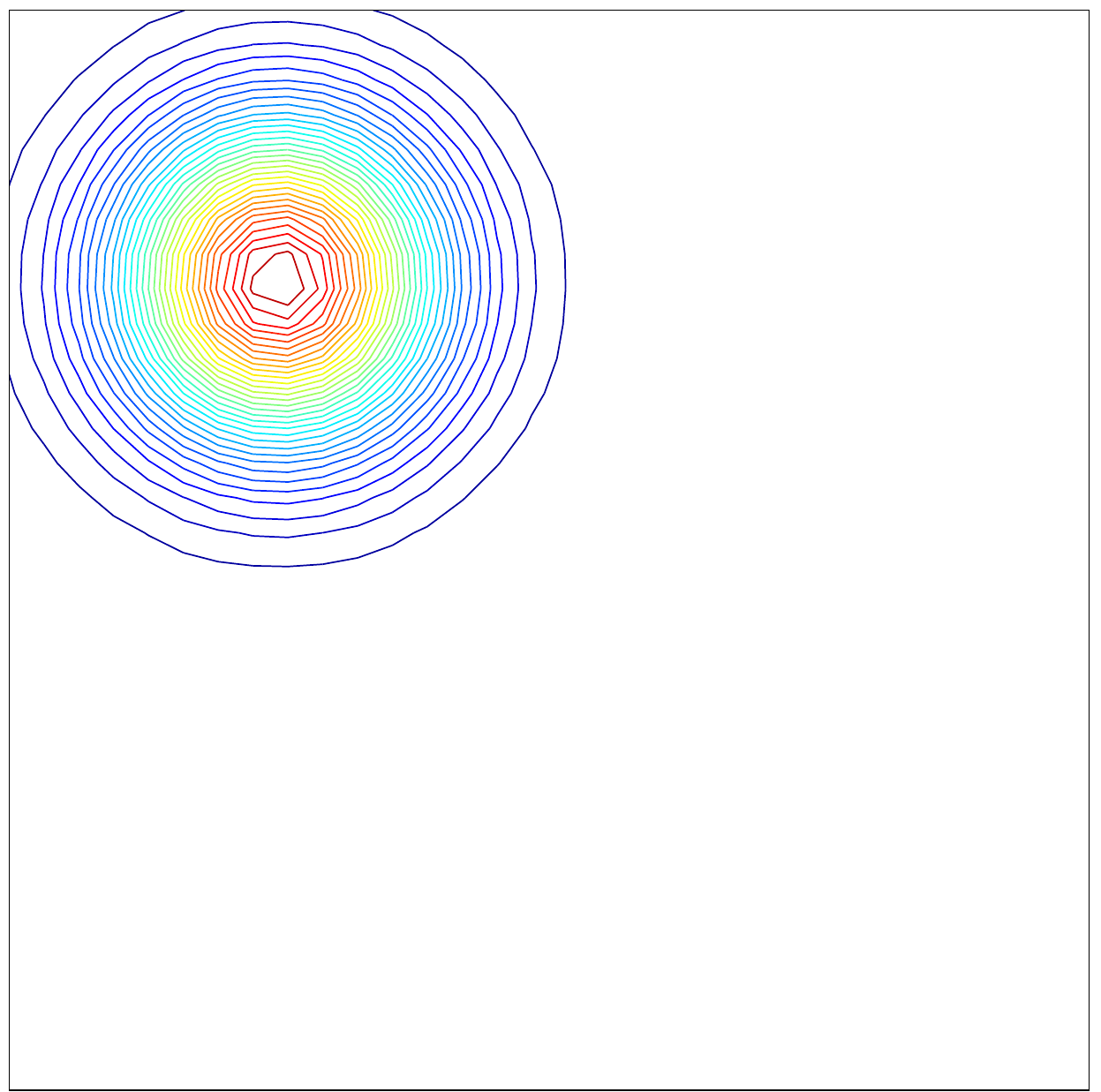}&
\myfigBeta{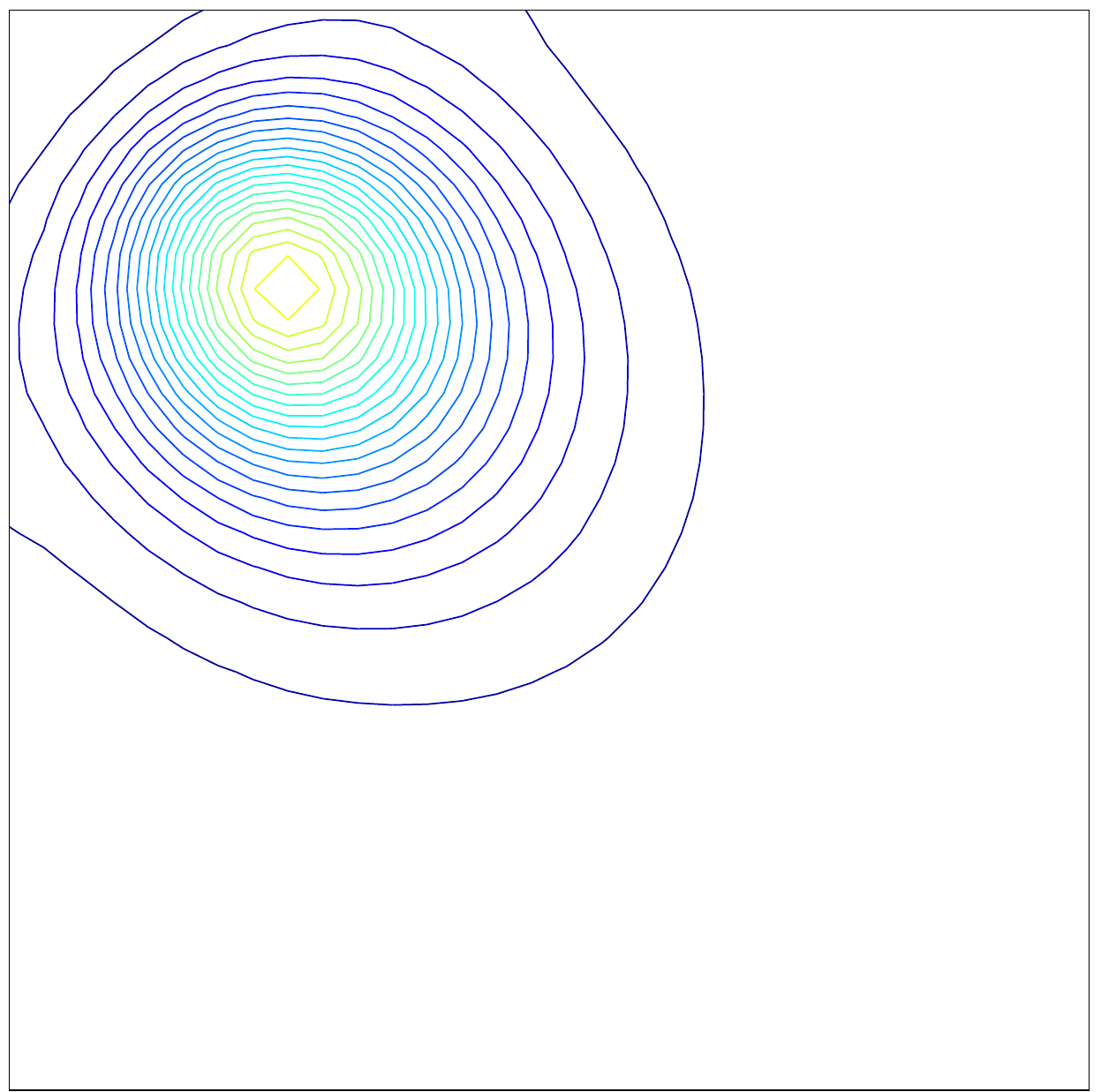}&
\myfigBeta{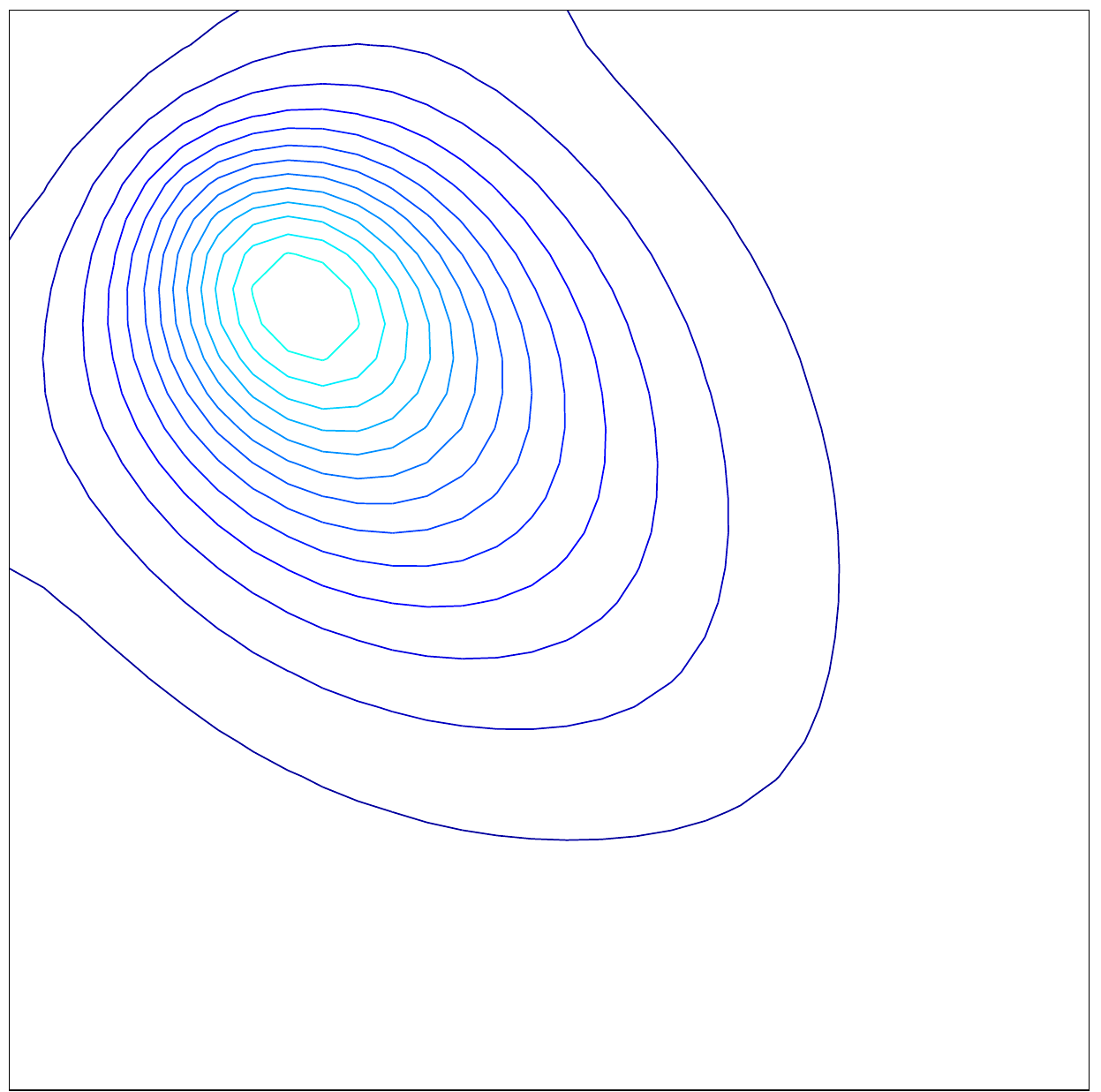}&
\myfigBeta{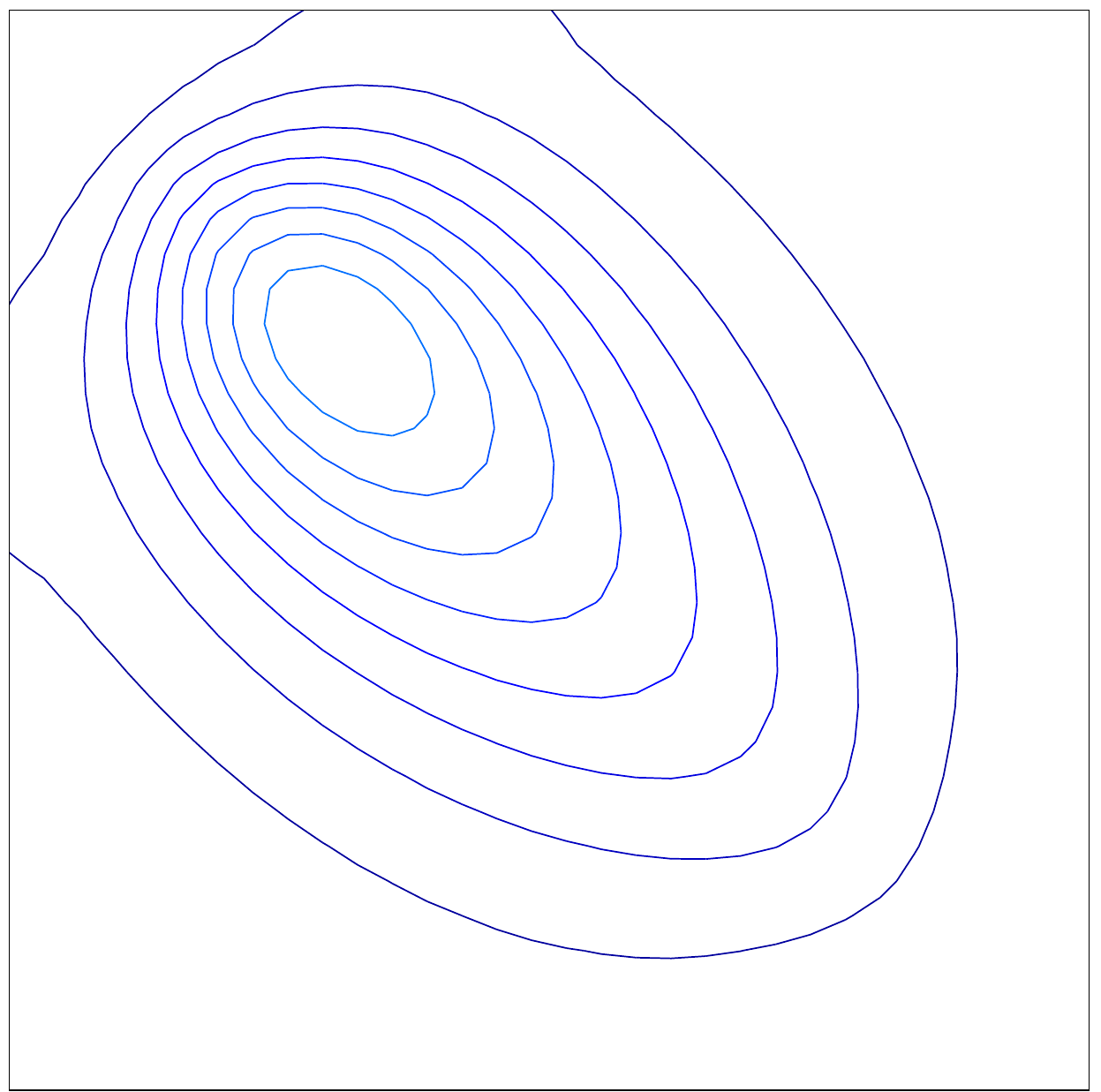}&
\myfigBeta{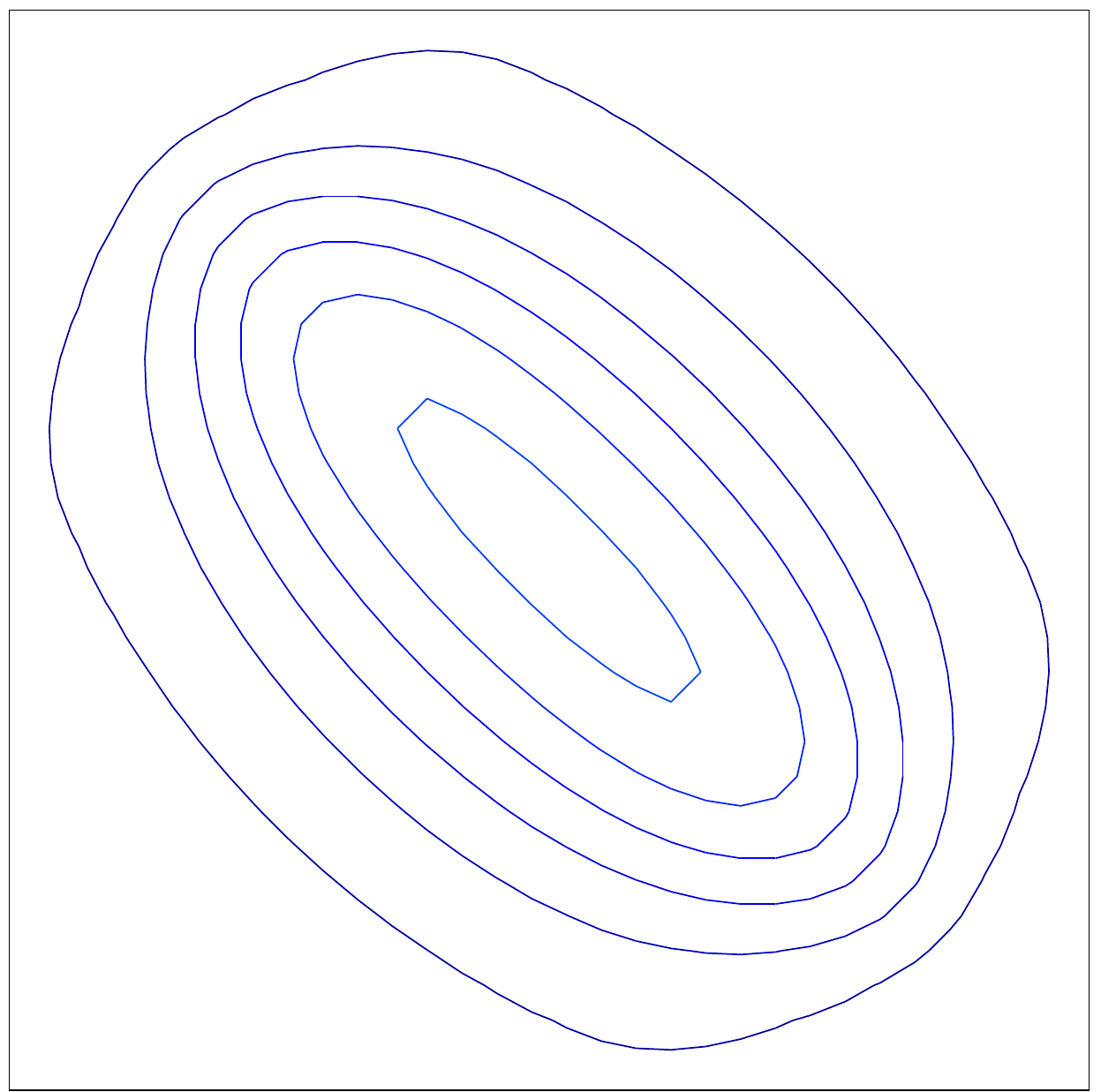}&
\myfigBeta{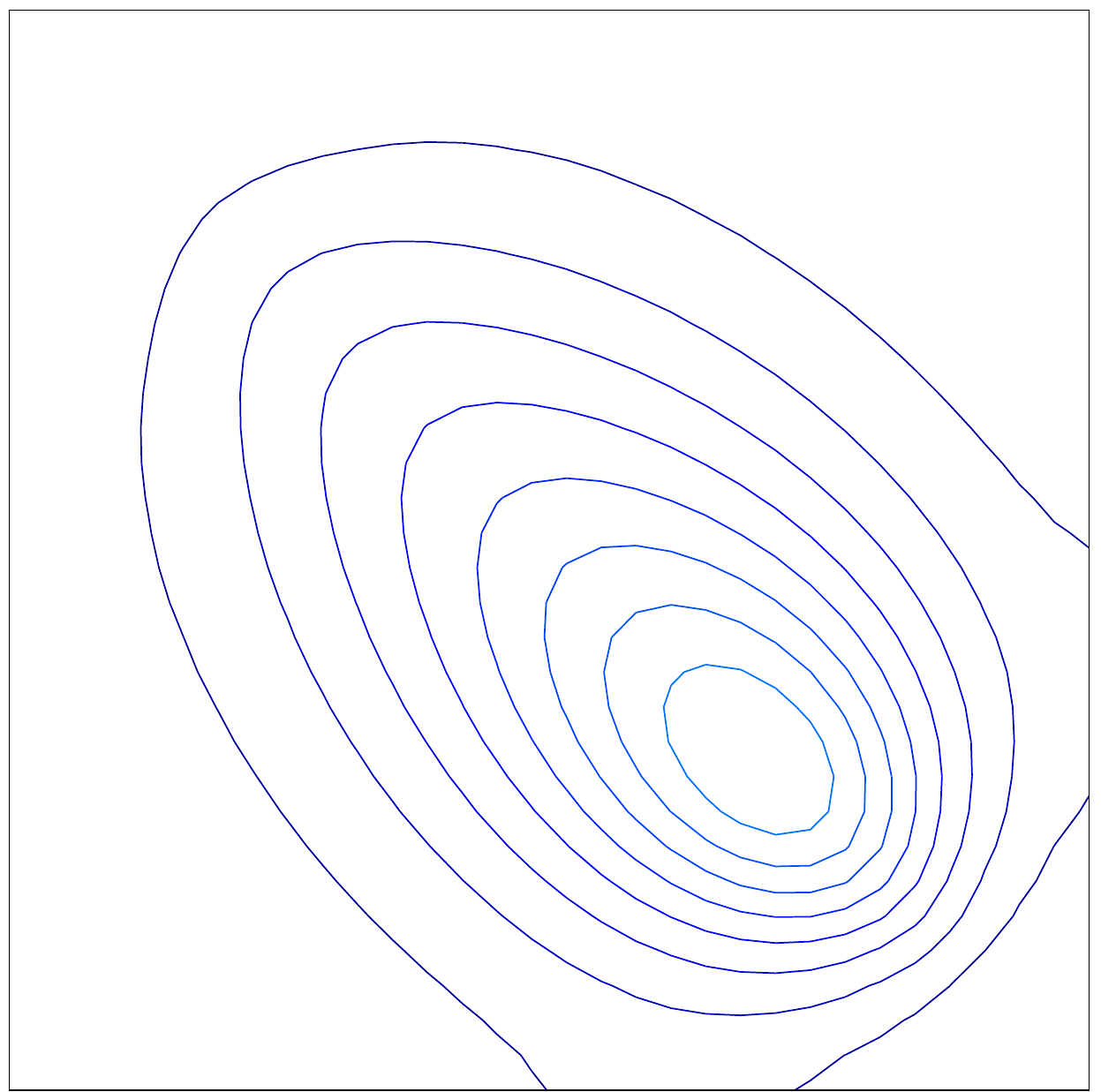}&
\myfigBeta{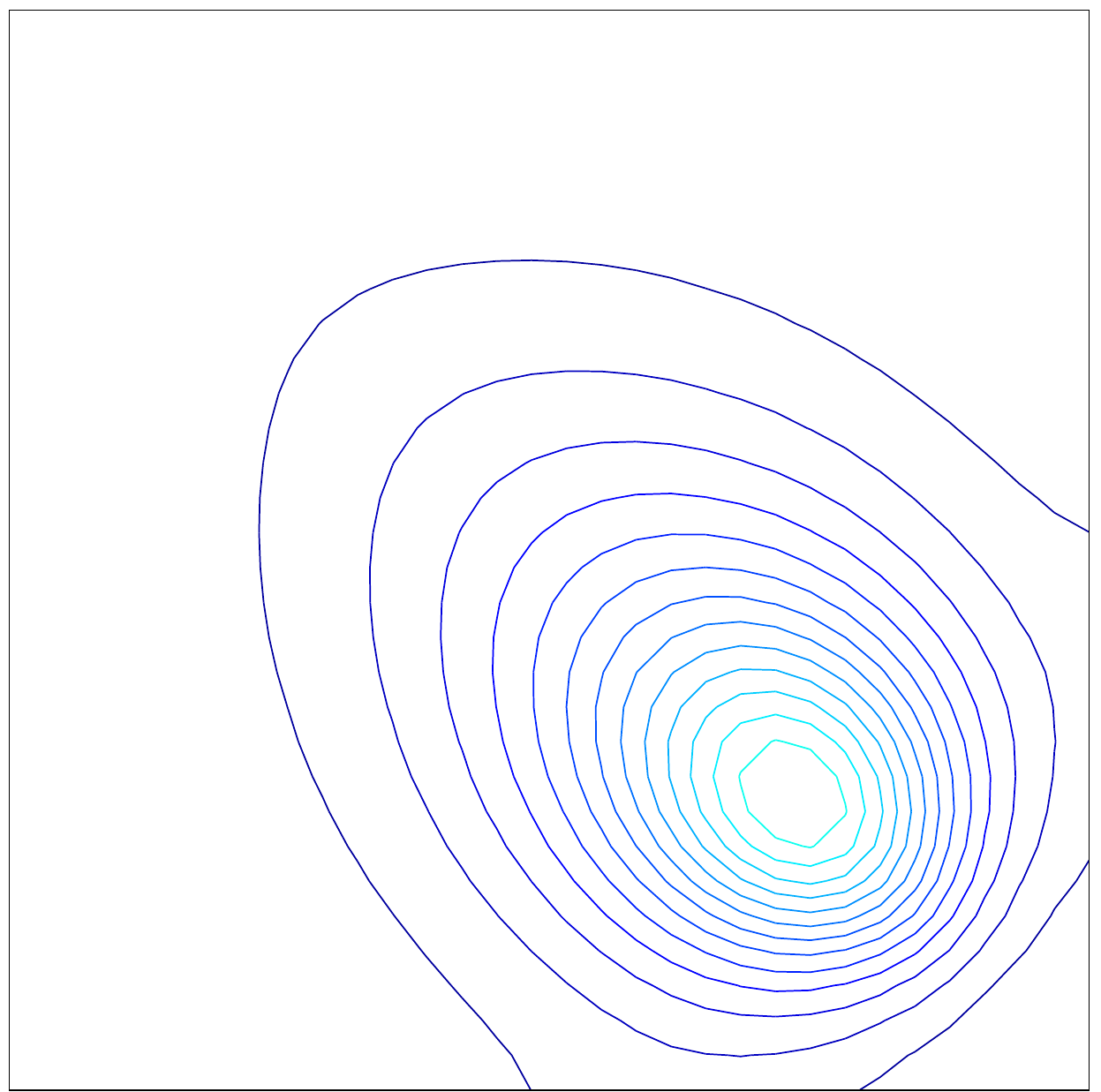}&
\myfigBeta{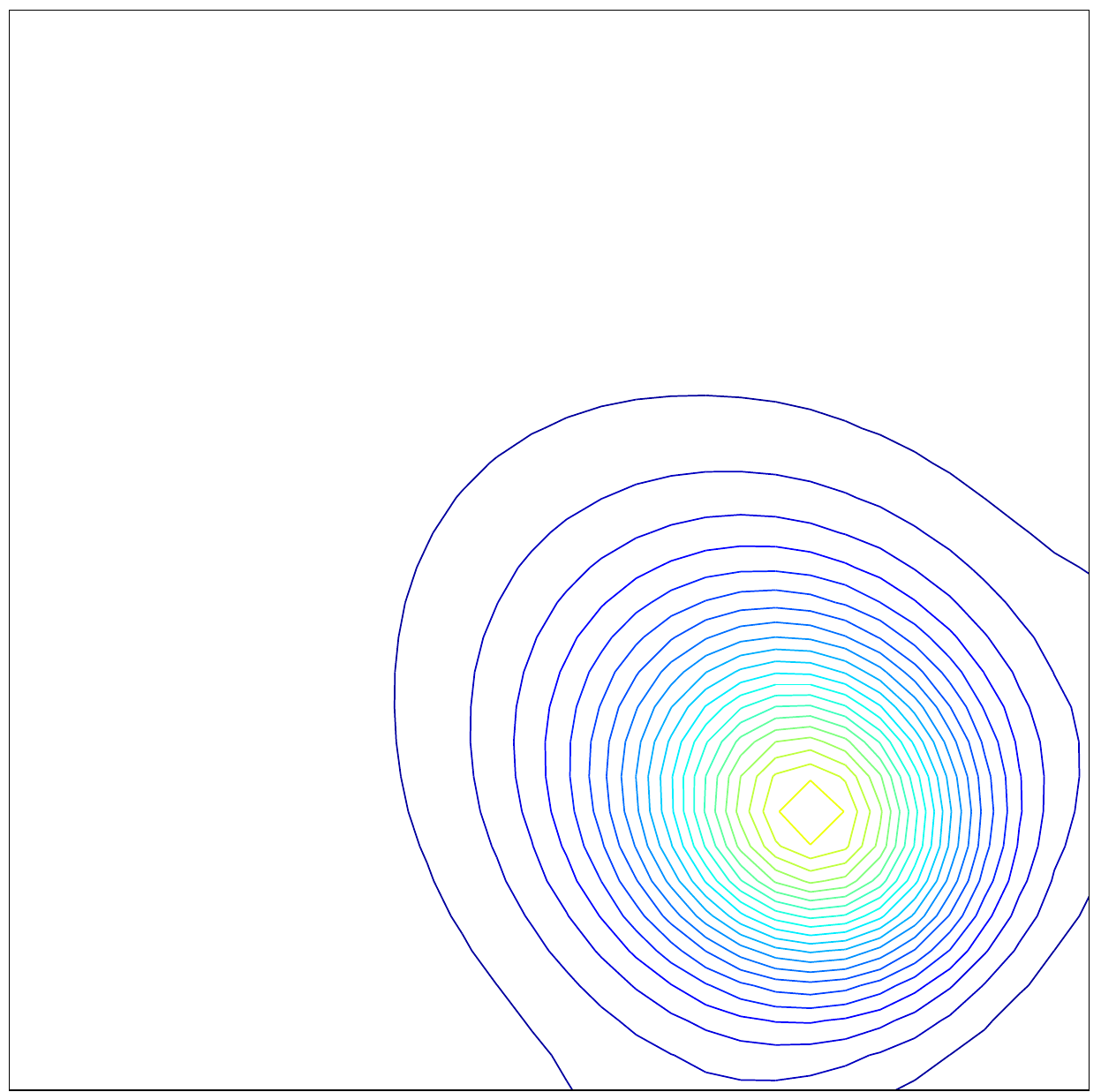}&
\myfigBeta{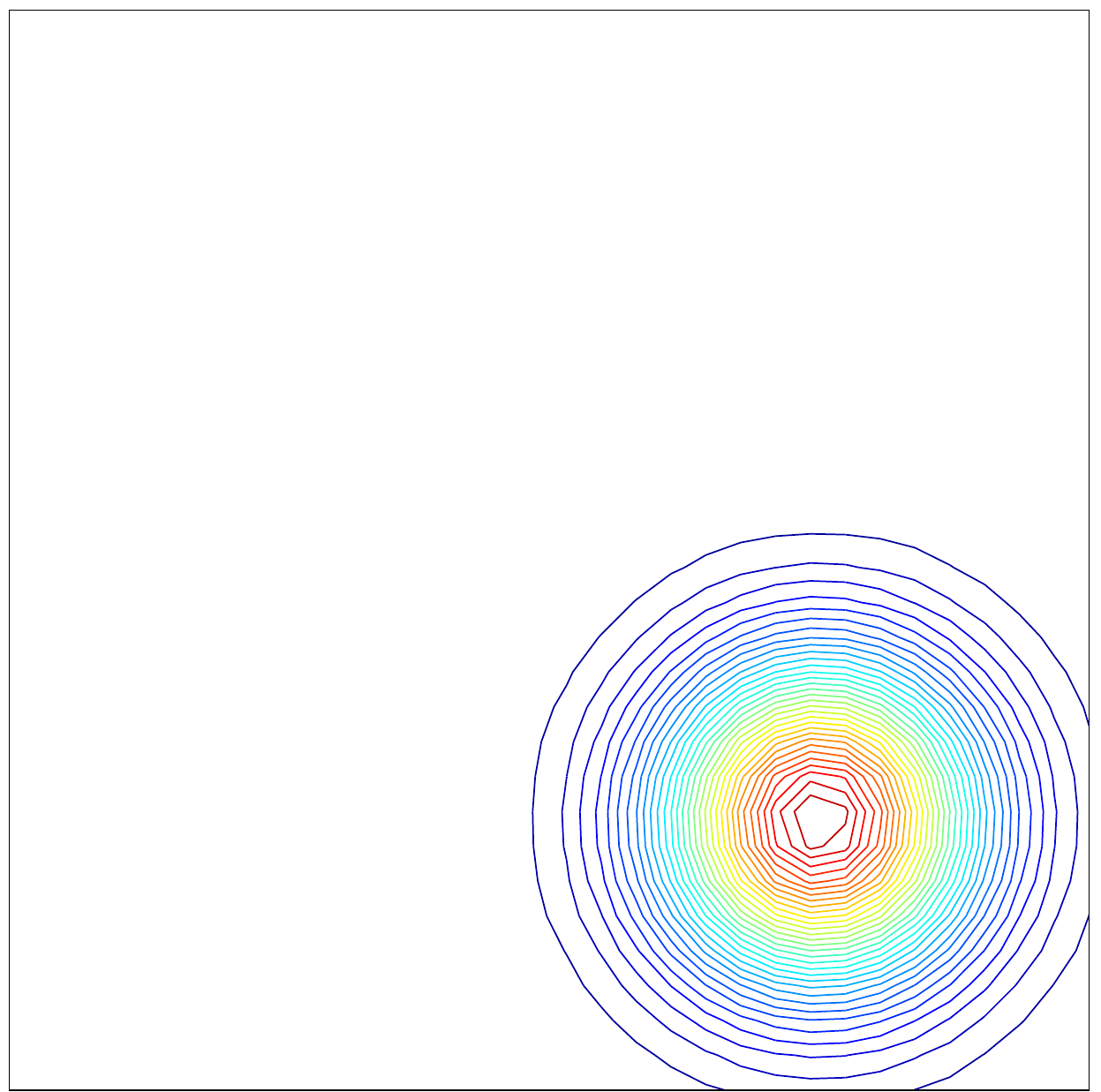}\\
\sidecap{$\beta=3/4$ } &
\myfigBeta{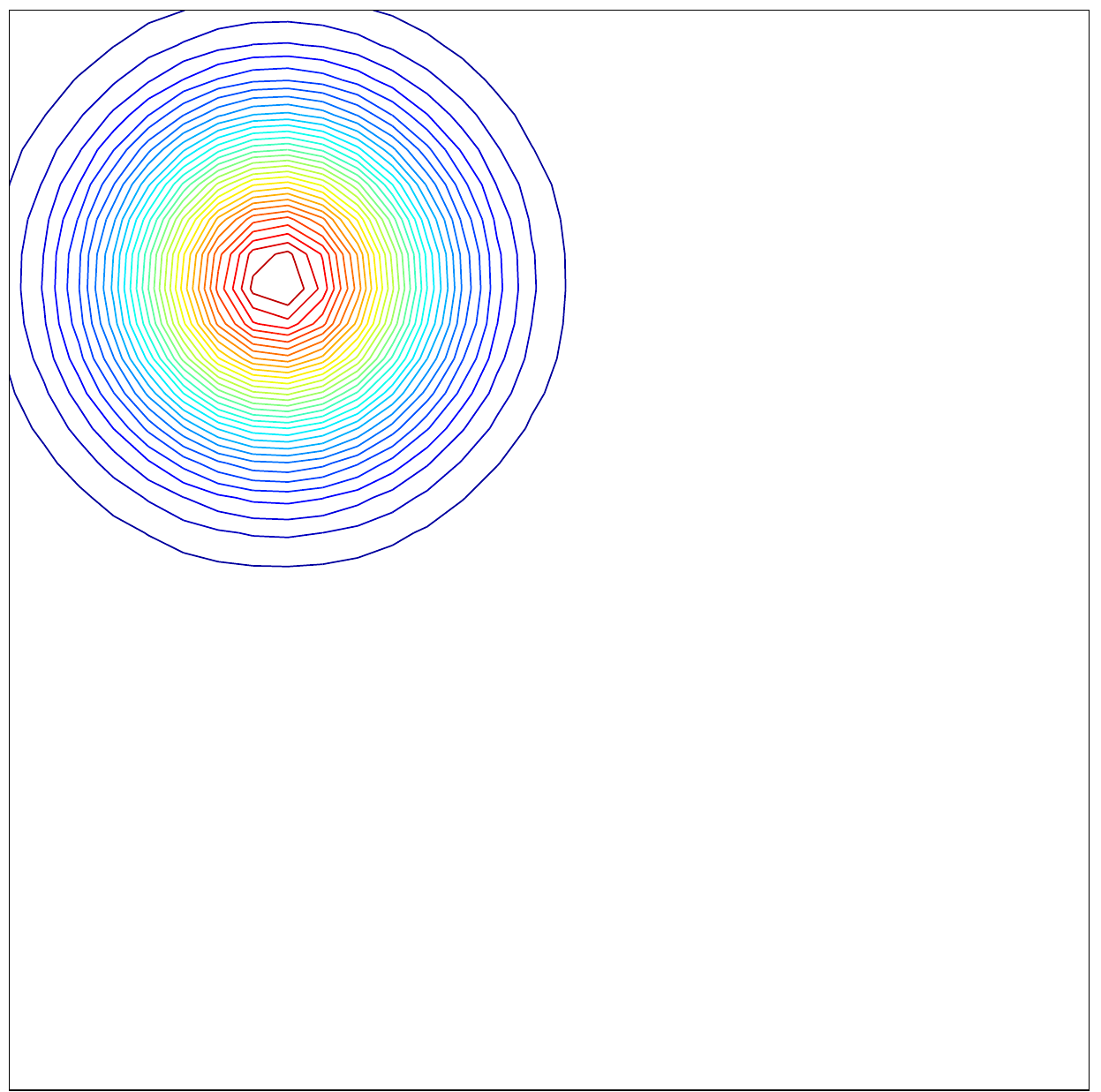}&
\myfigBeta{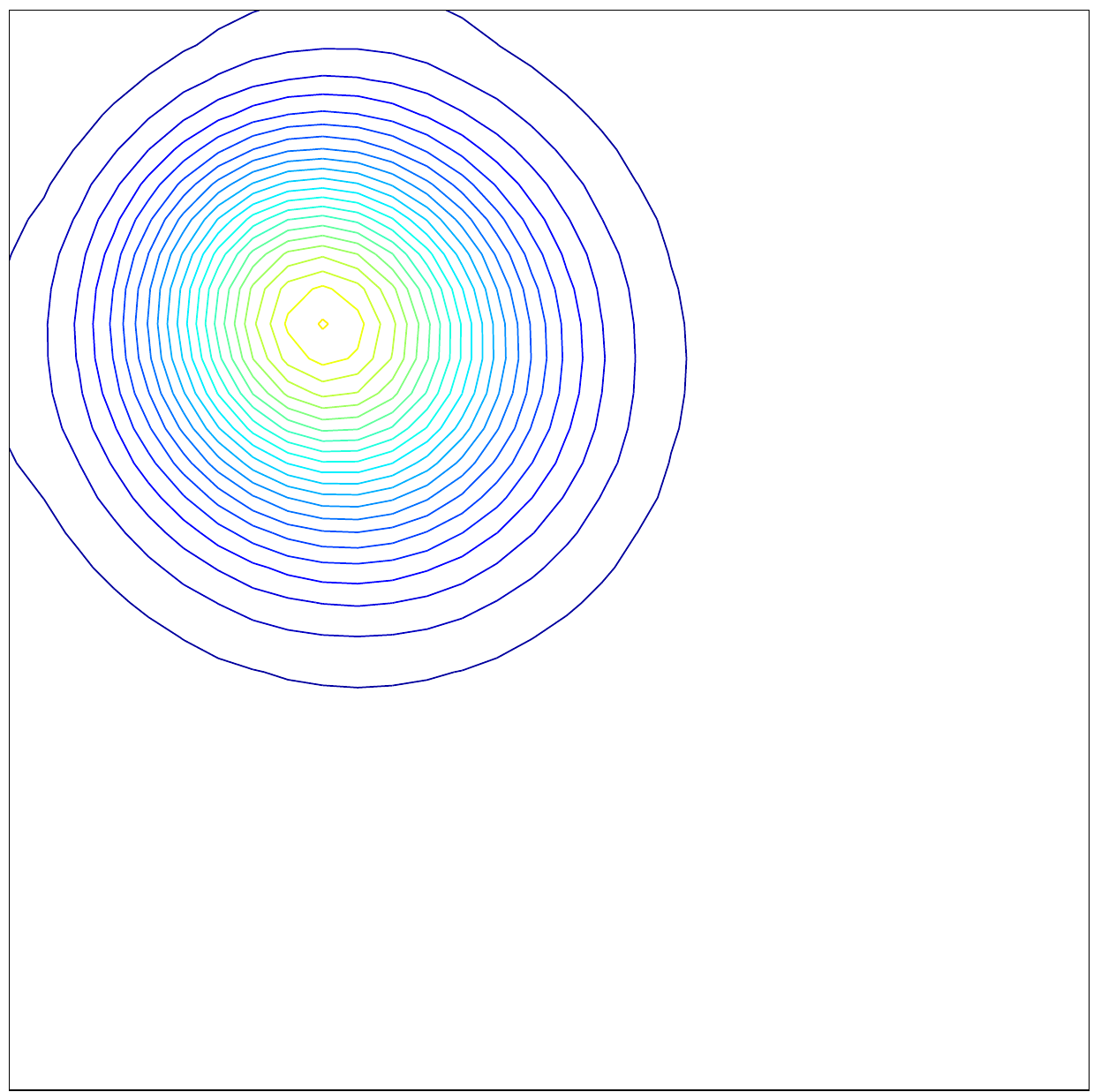}&
\myfigBeta{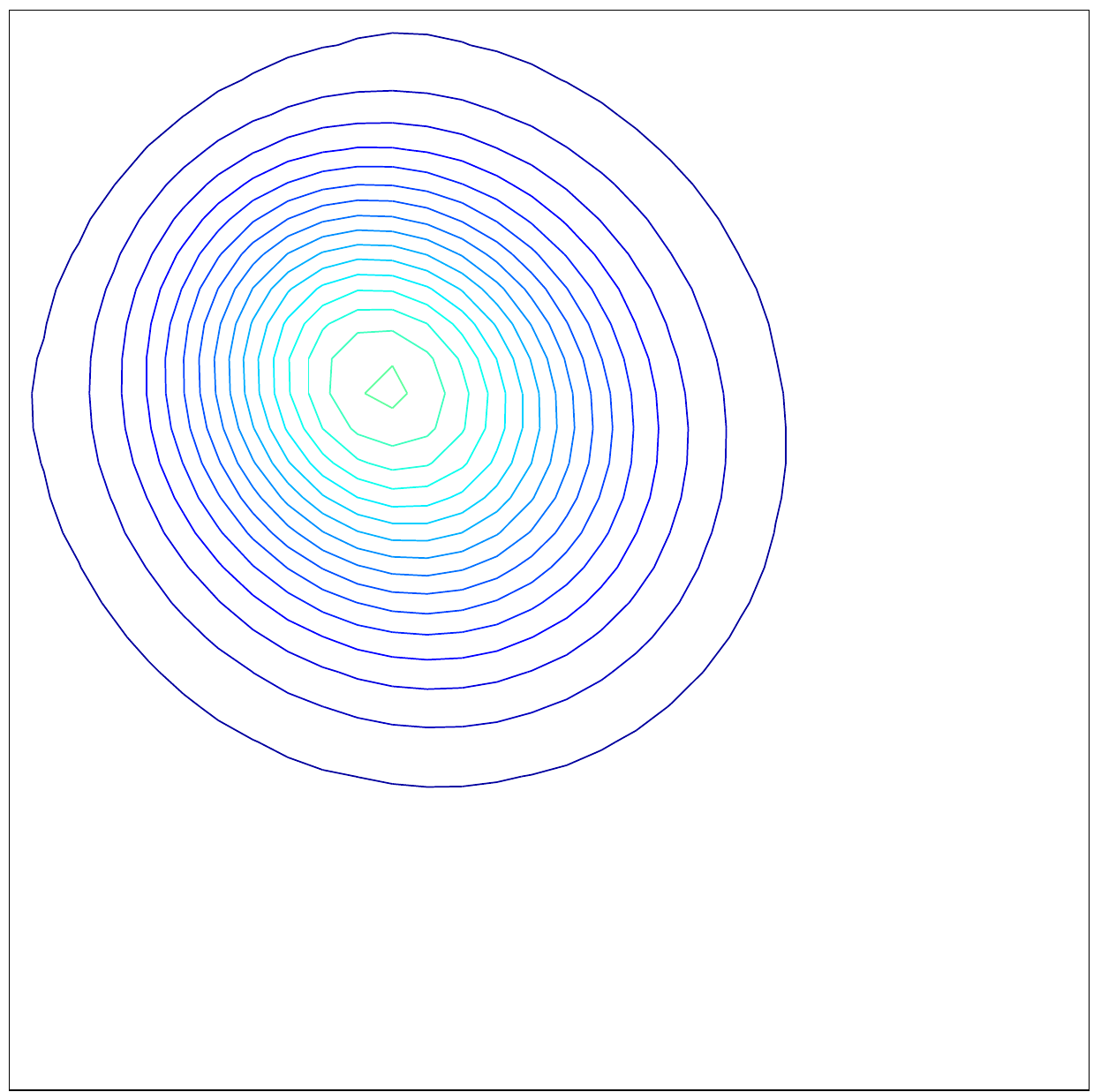}&
\myfigBeta{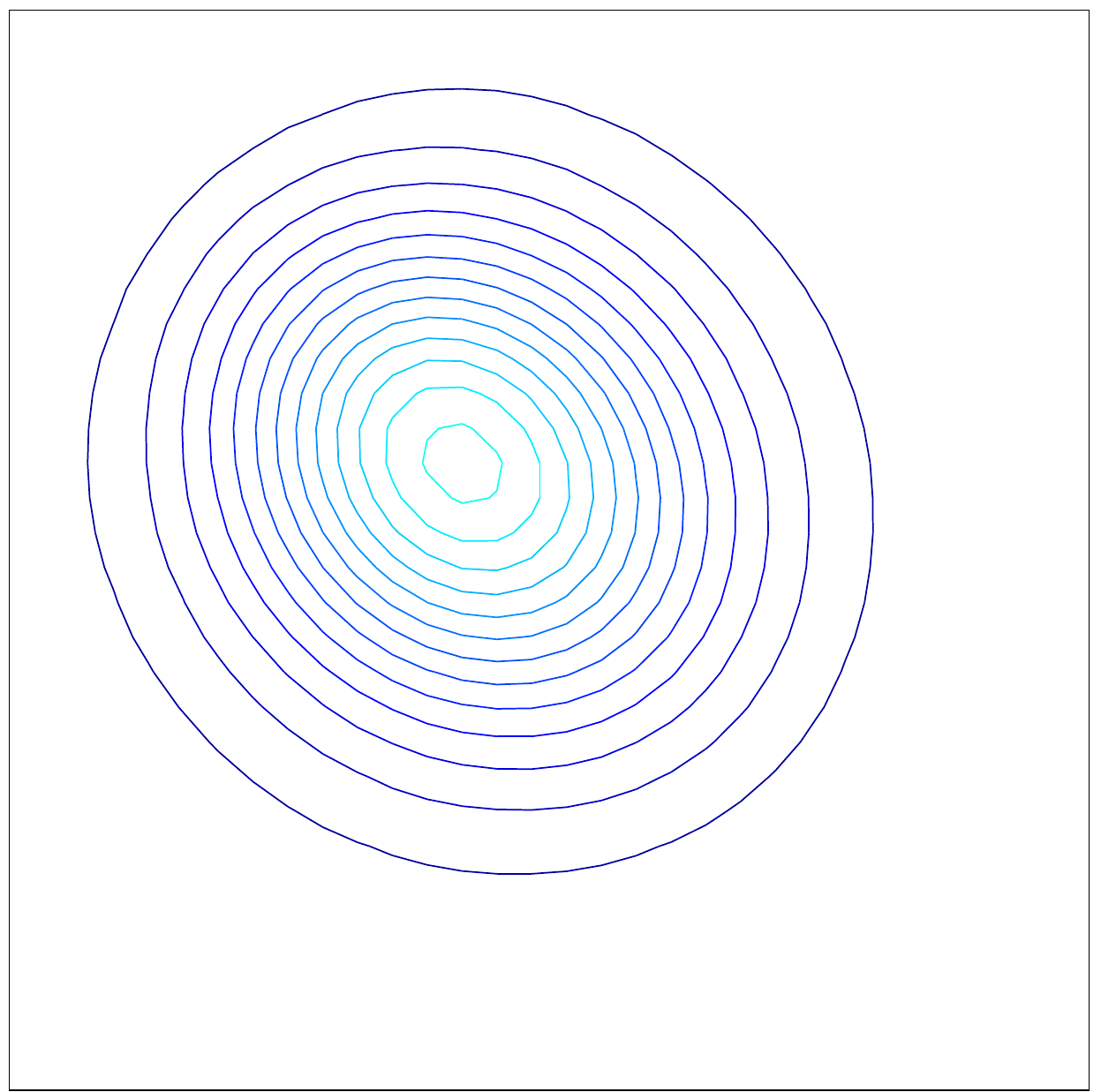}&
\myfigBeta{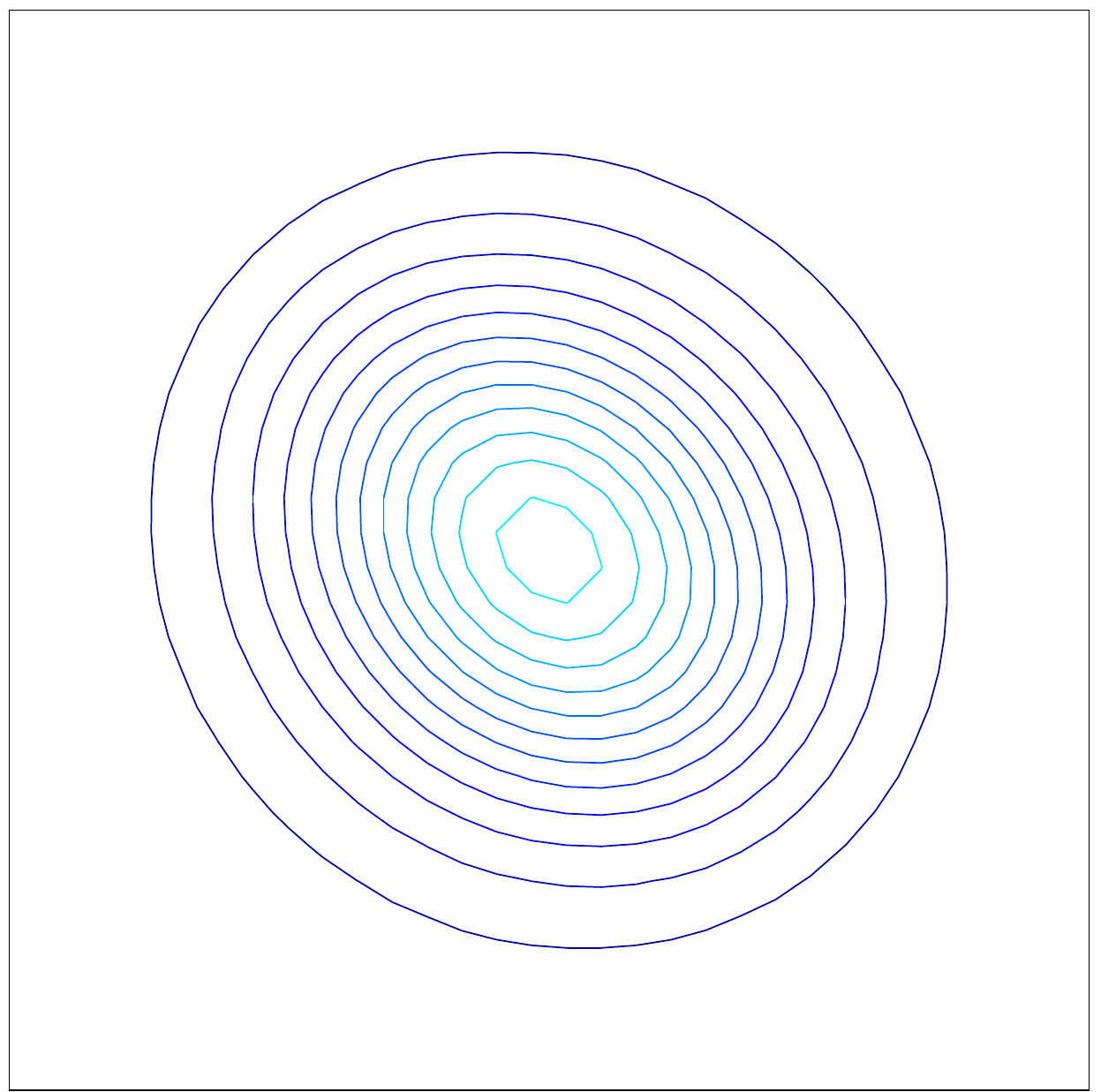}&
\myfigBeta{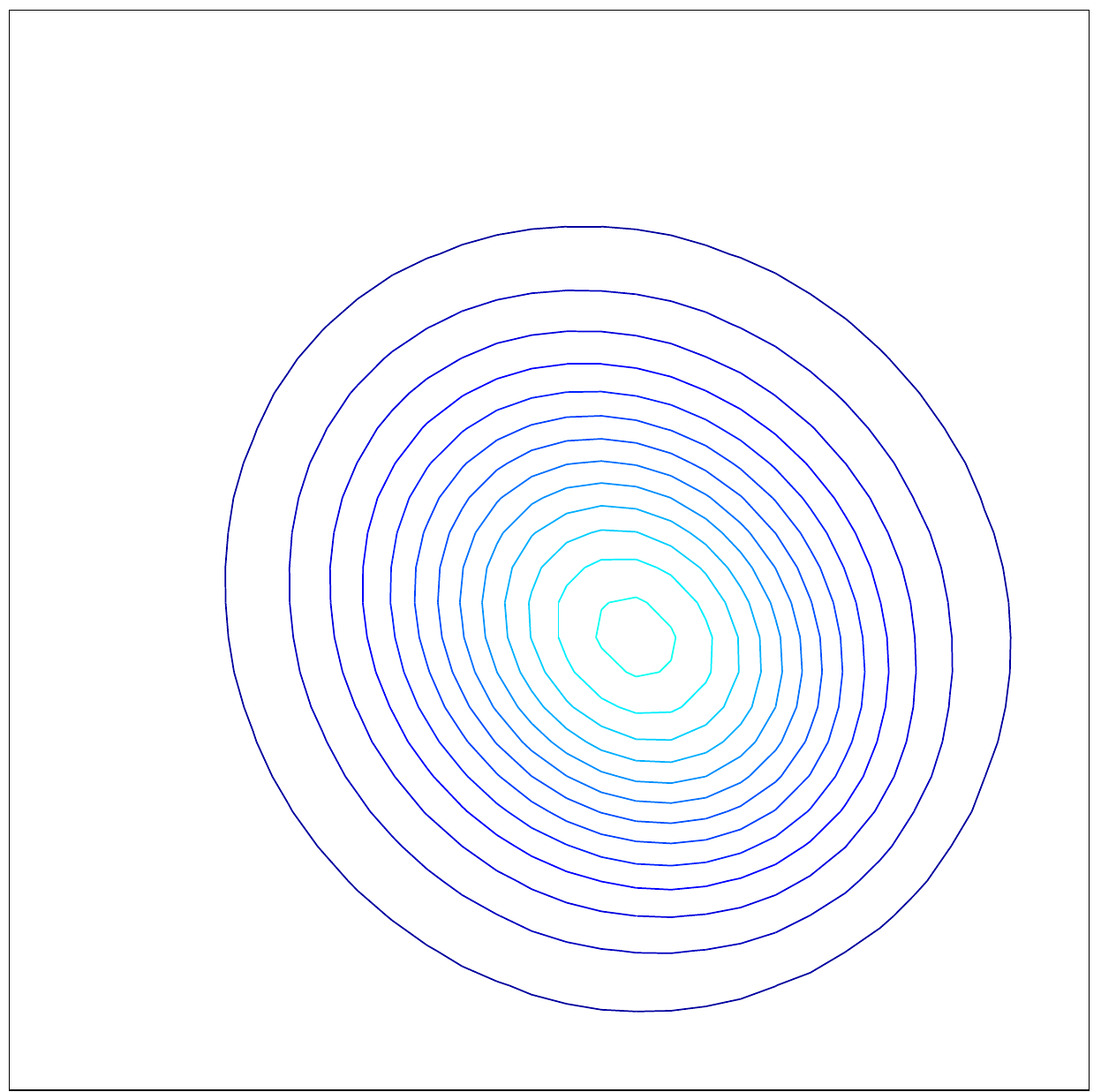}&
\myfigBeta{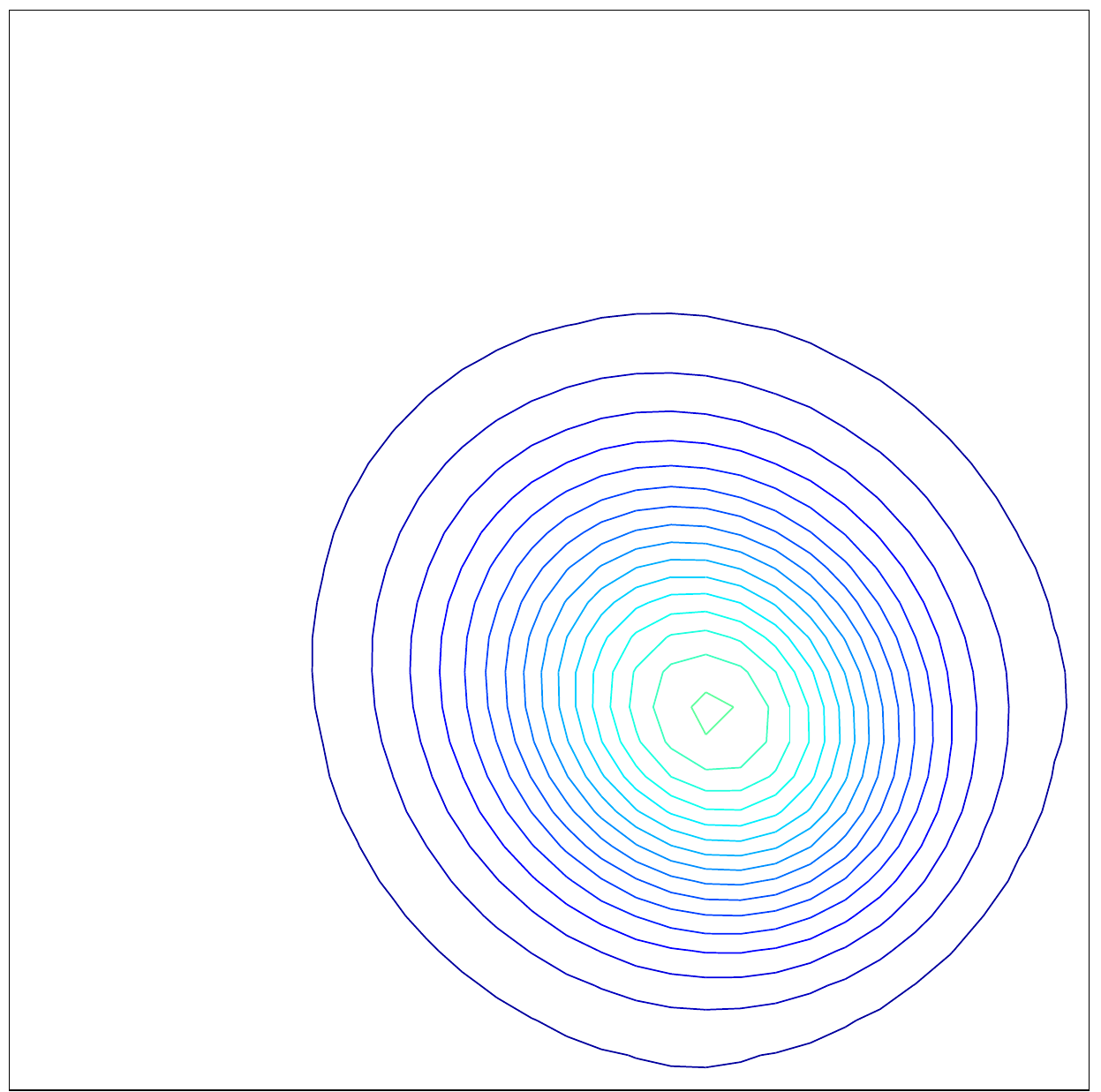}&
\myfigBeta{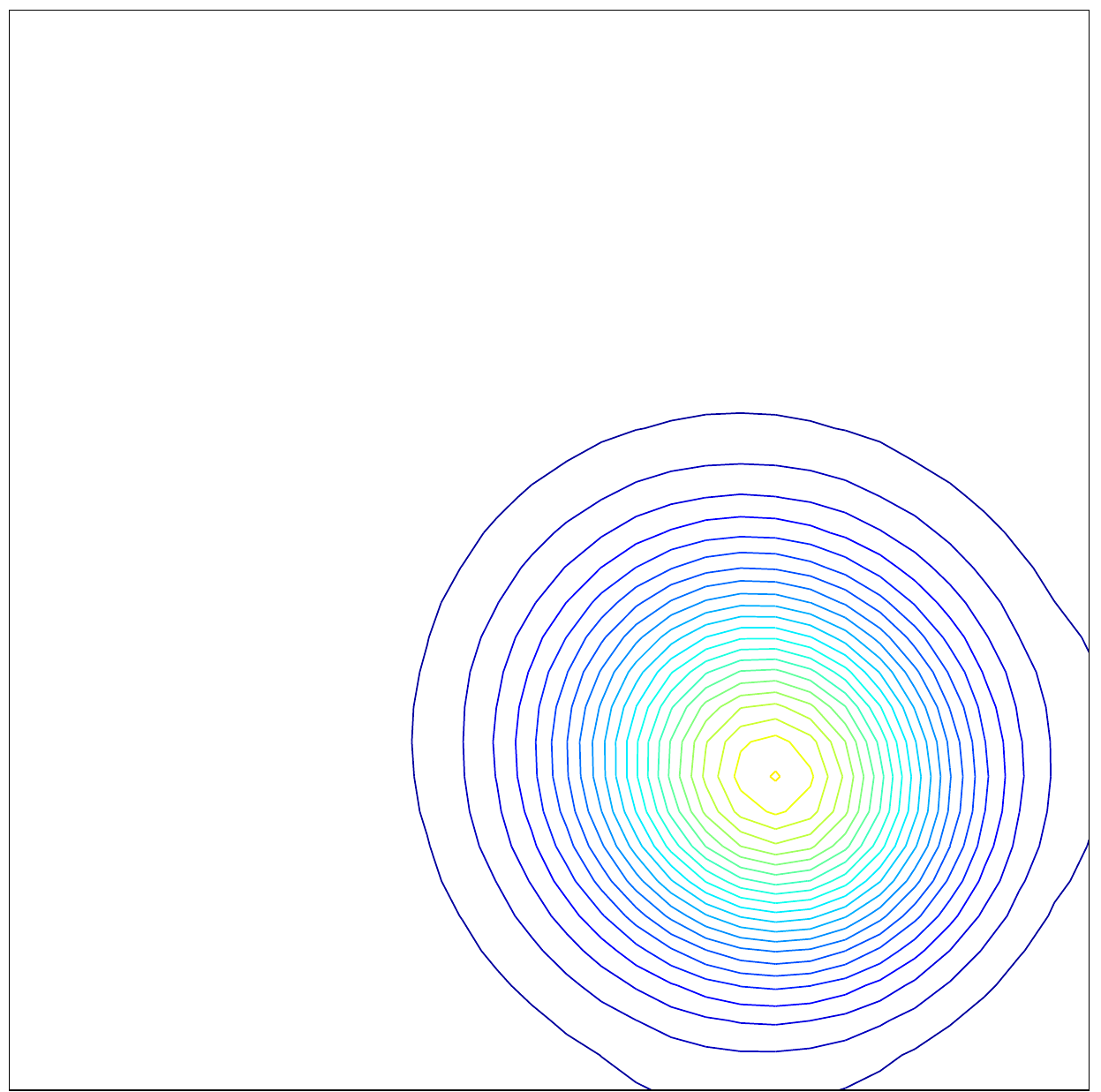}&
\myfigBeta{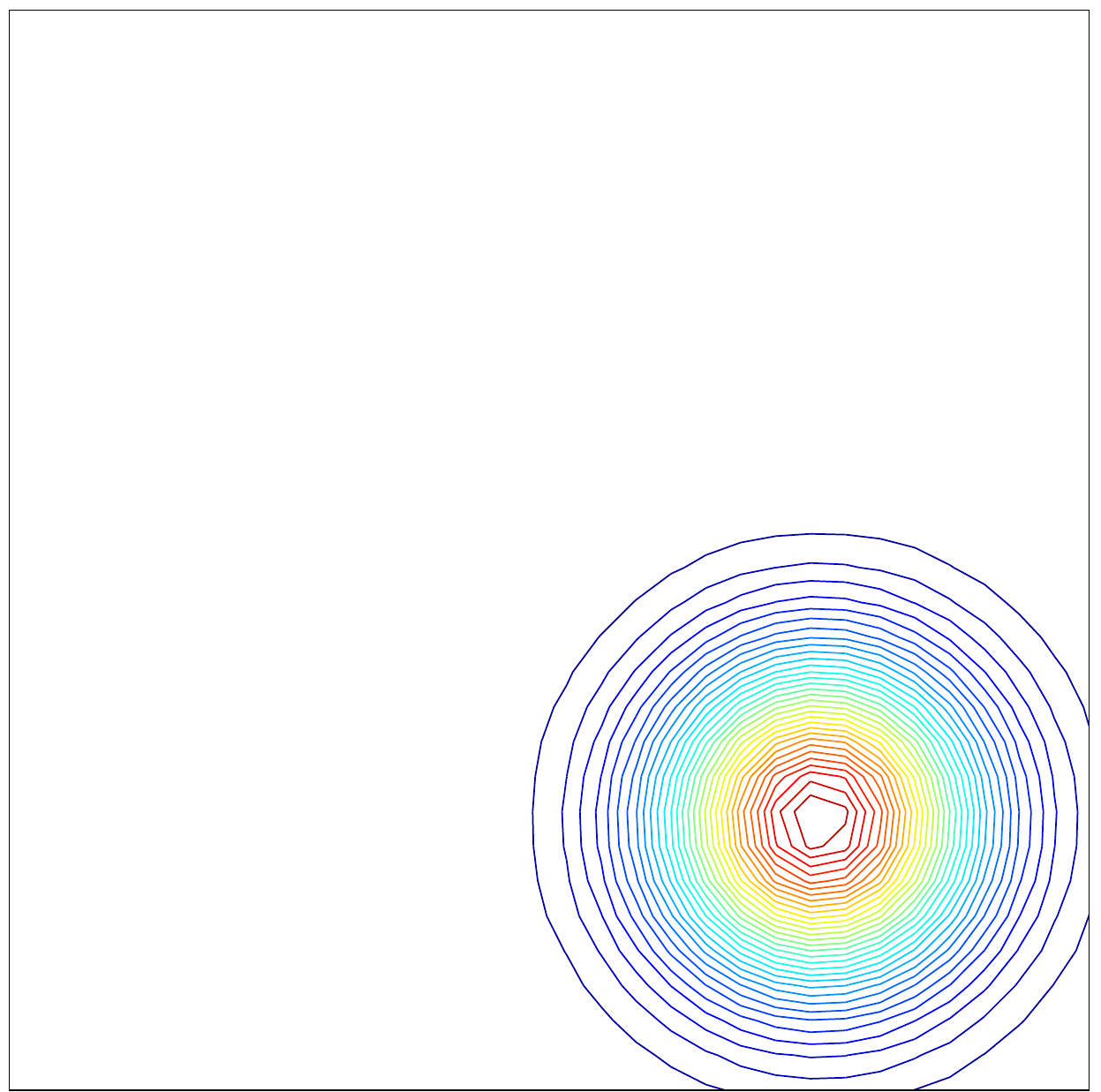}\\
\sidecap{$\beta=1$ } &
\myfigBeta{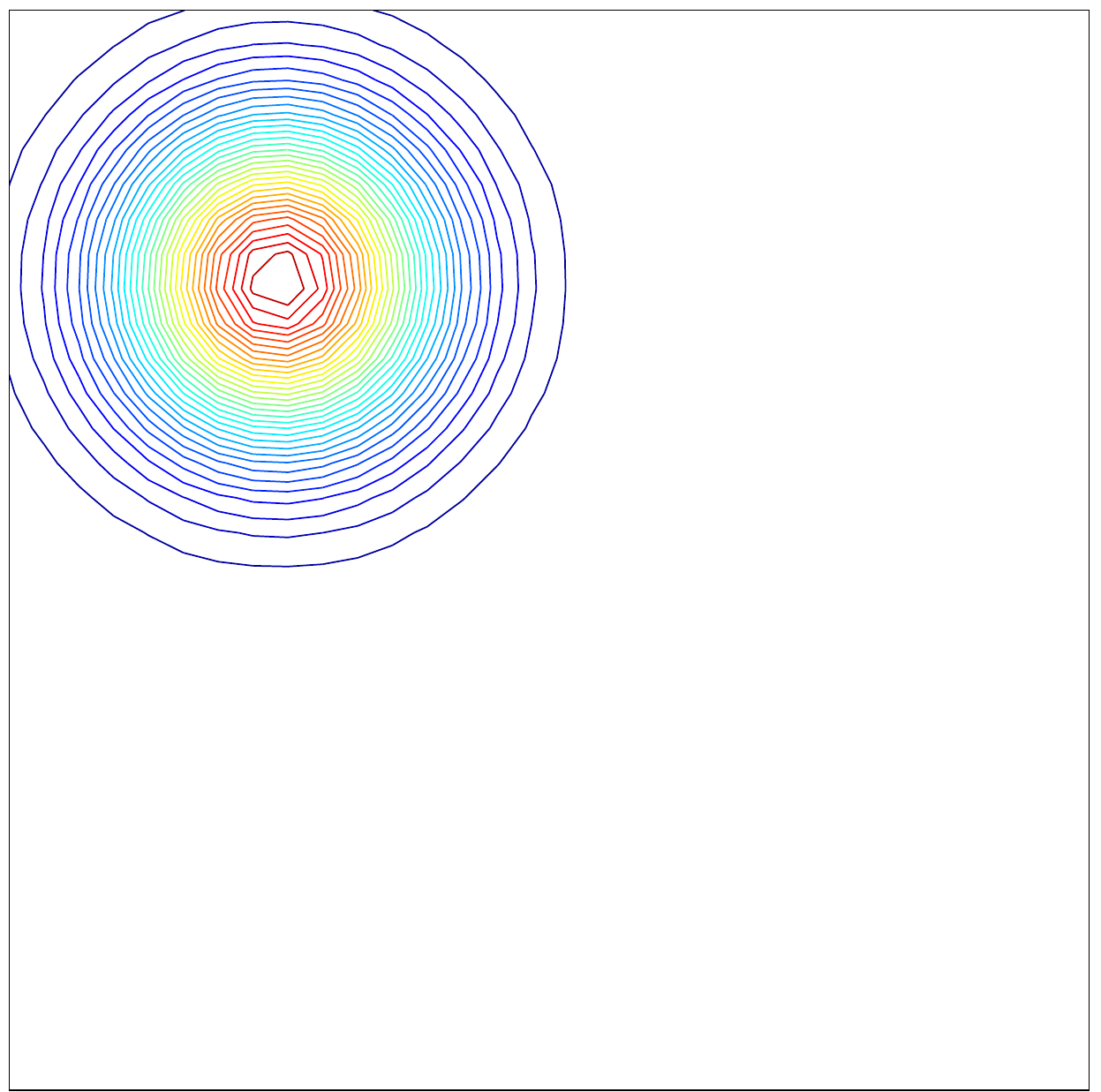}&
\myfigBeta{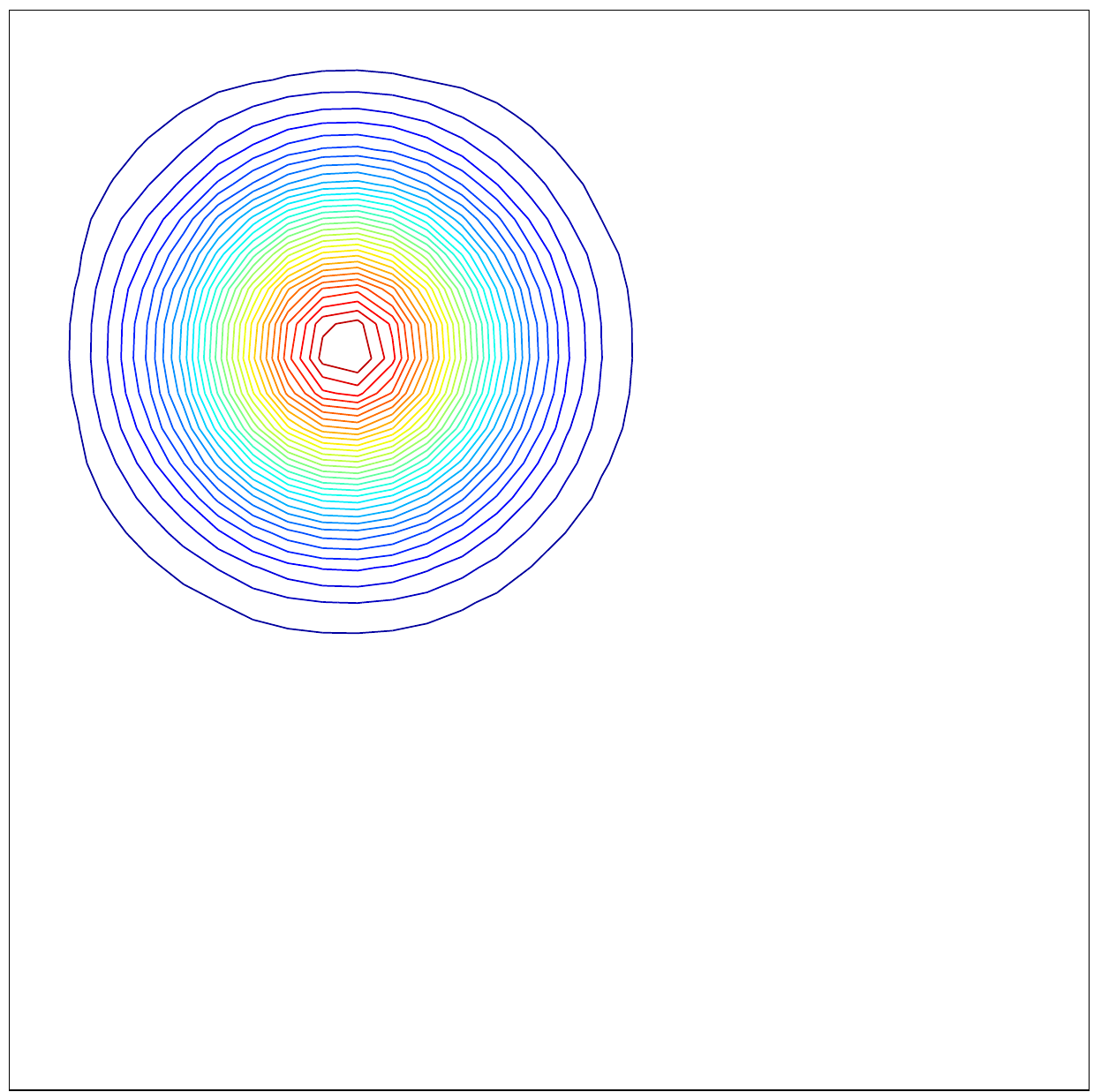}&
\myfigBeta{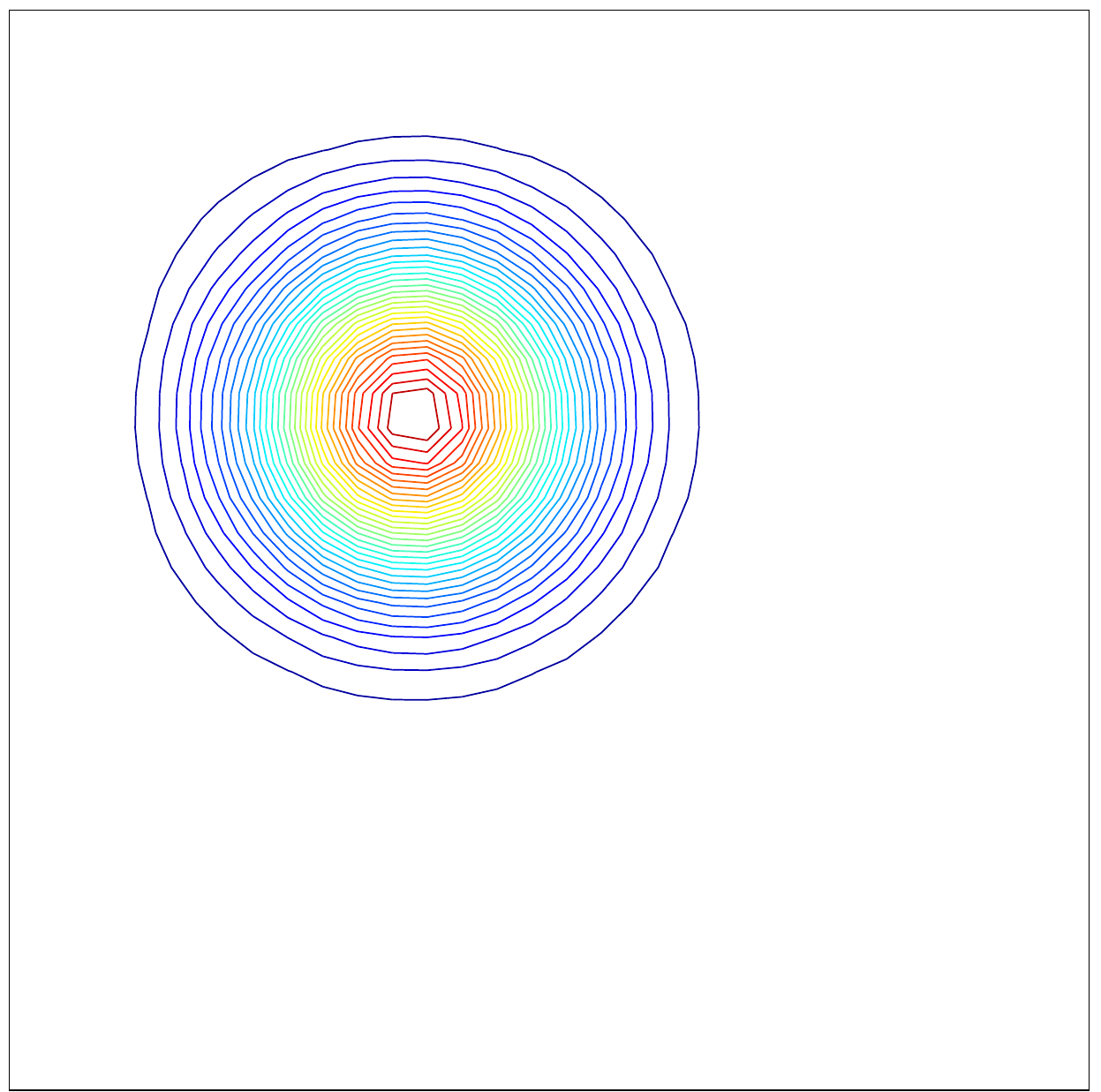}&
\myfigBeta{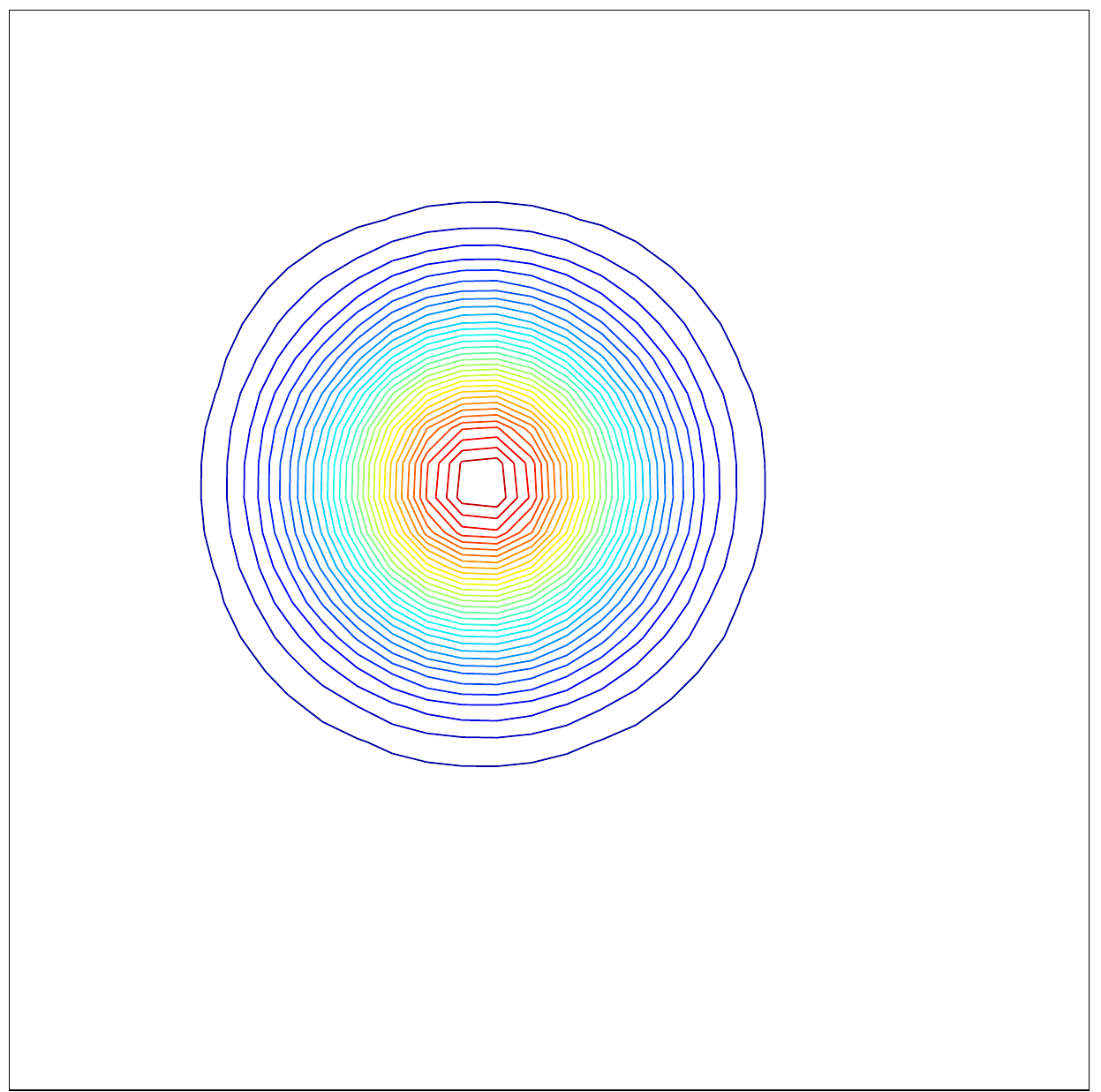}&
\myfigBeta{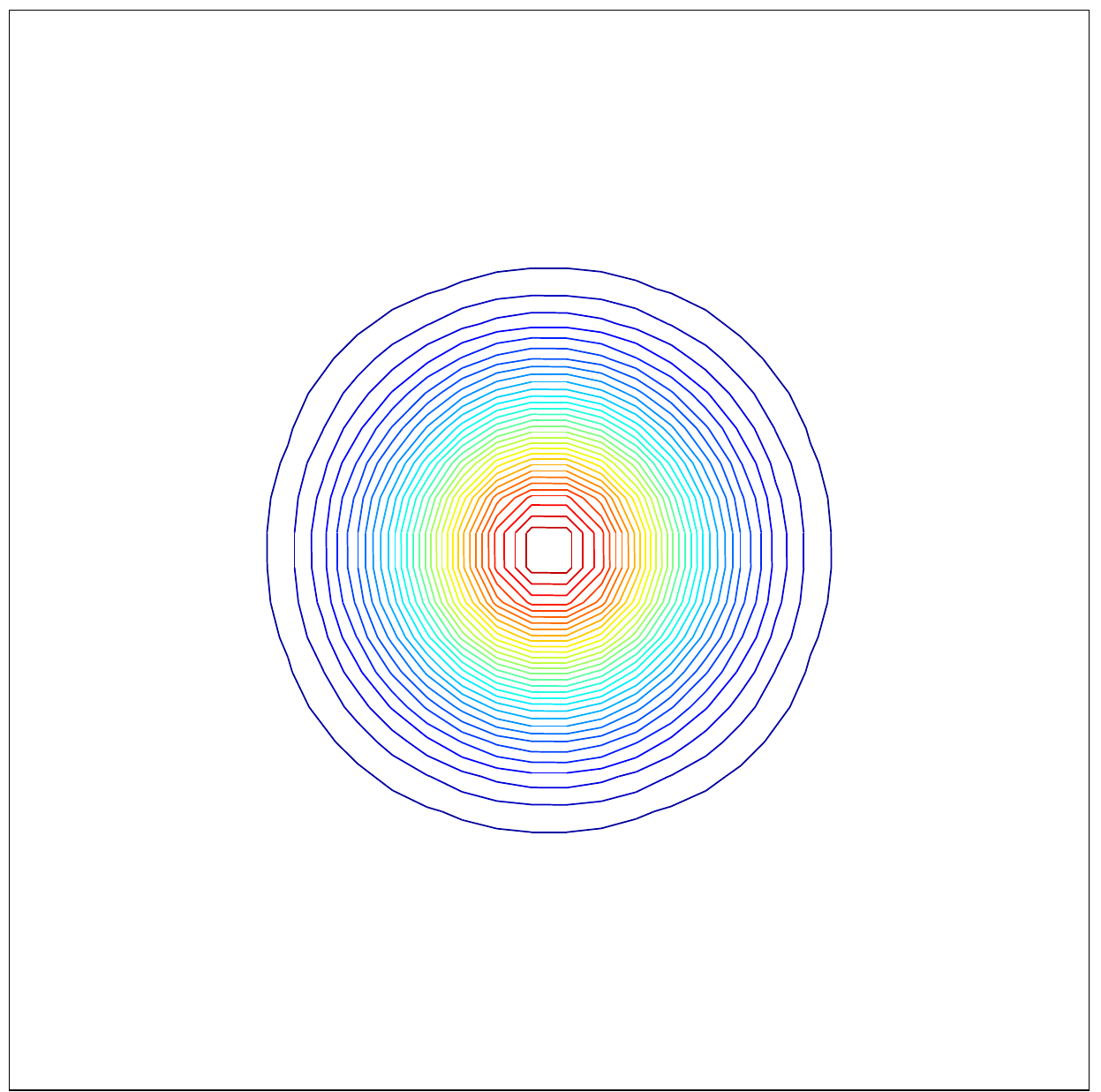}&
\myfigBeta{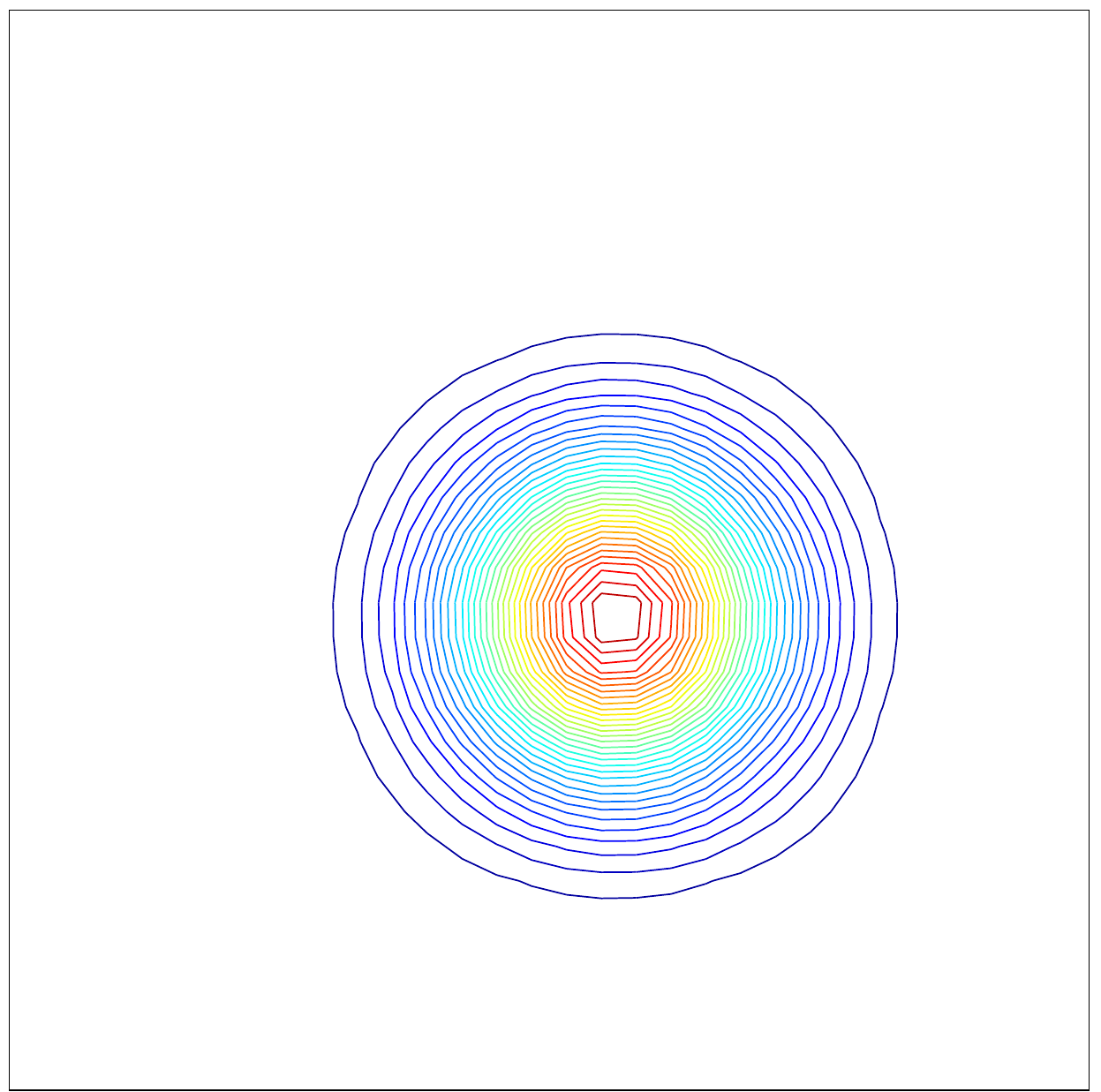}&
\myfigBeta{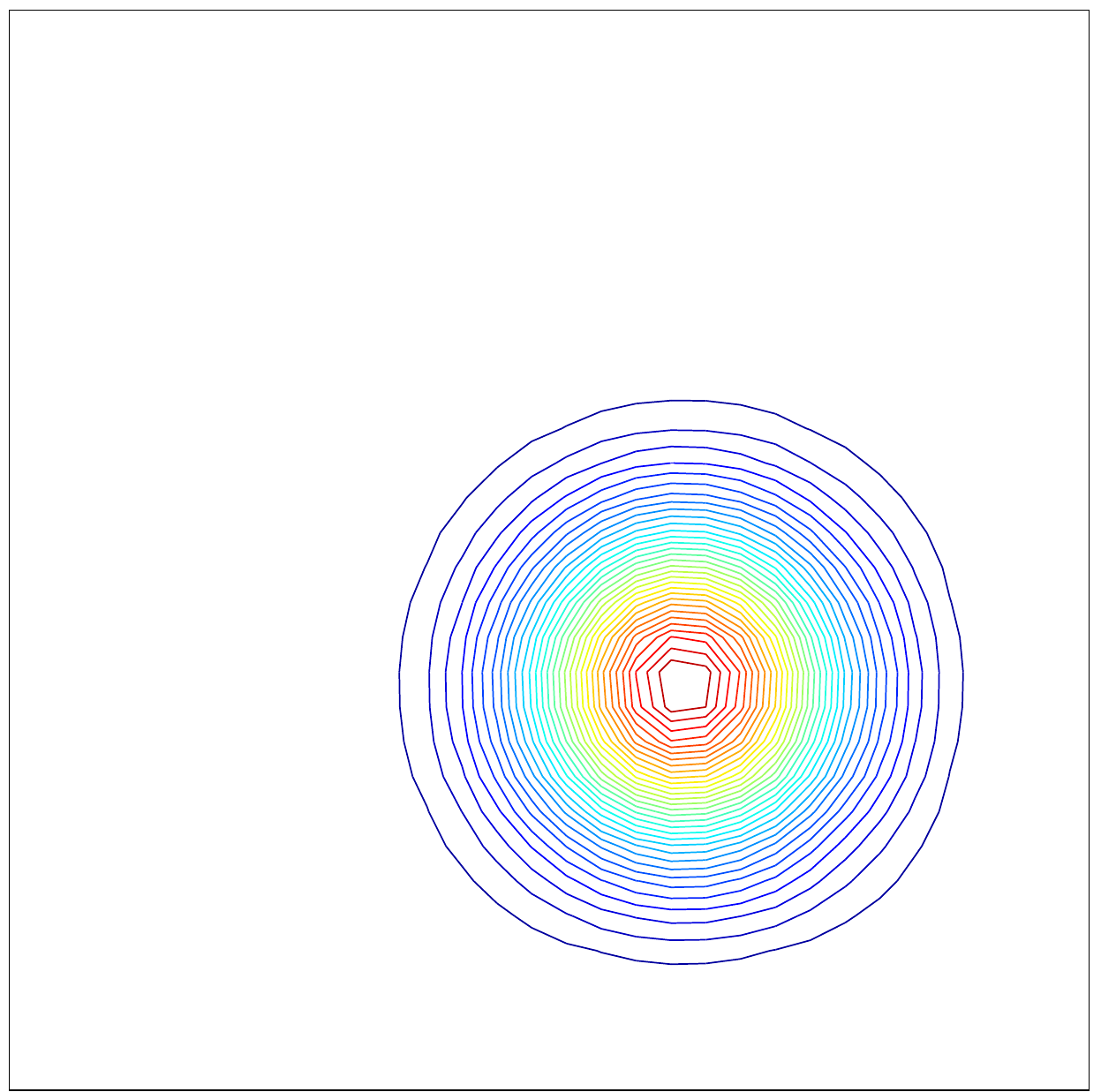}&
\myfigBeta{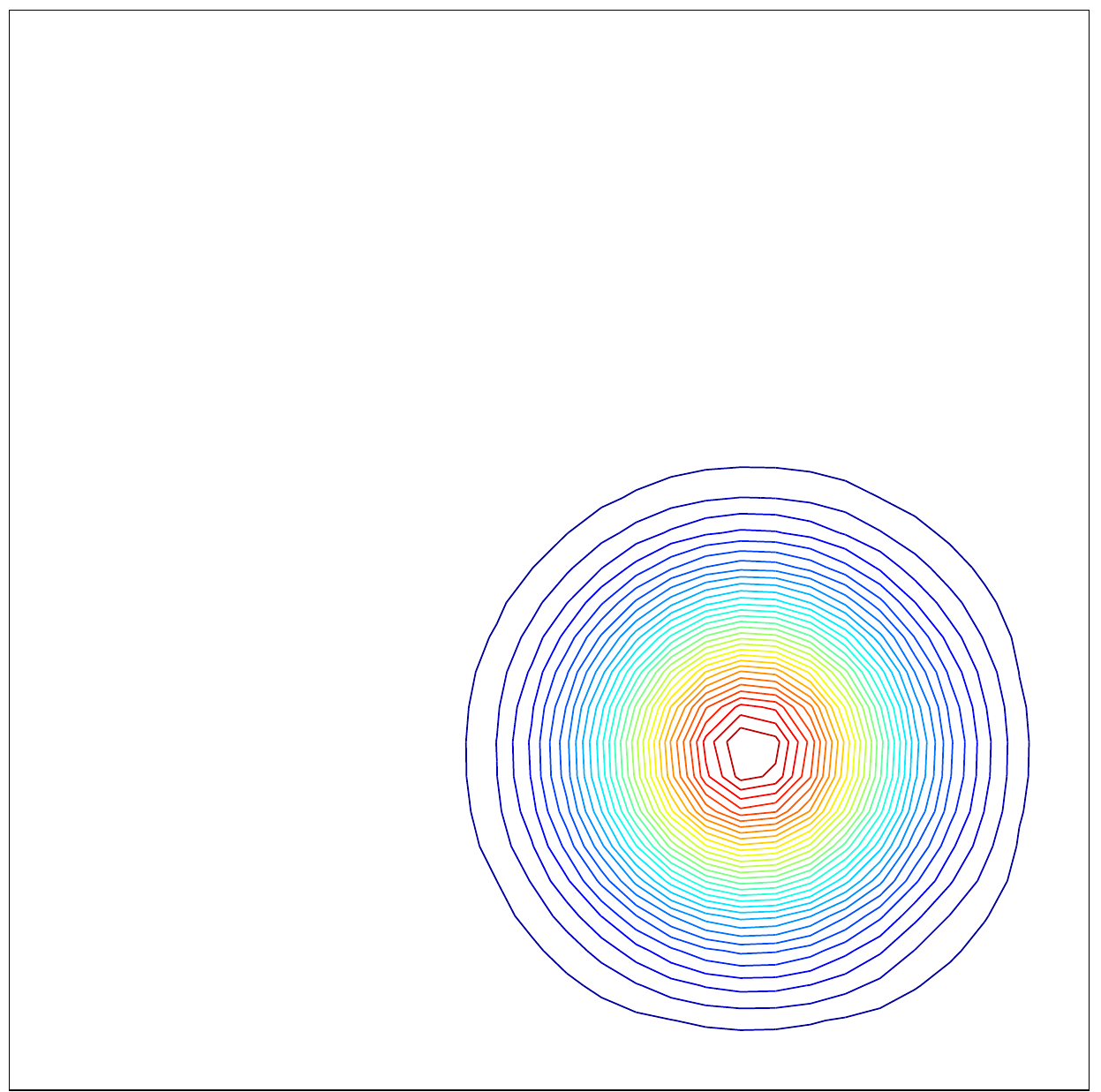}&
\myfigBeta{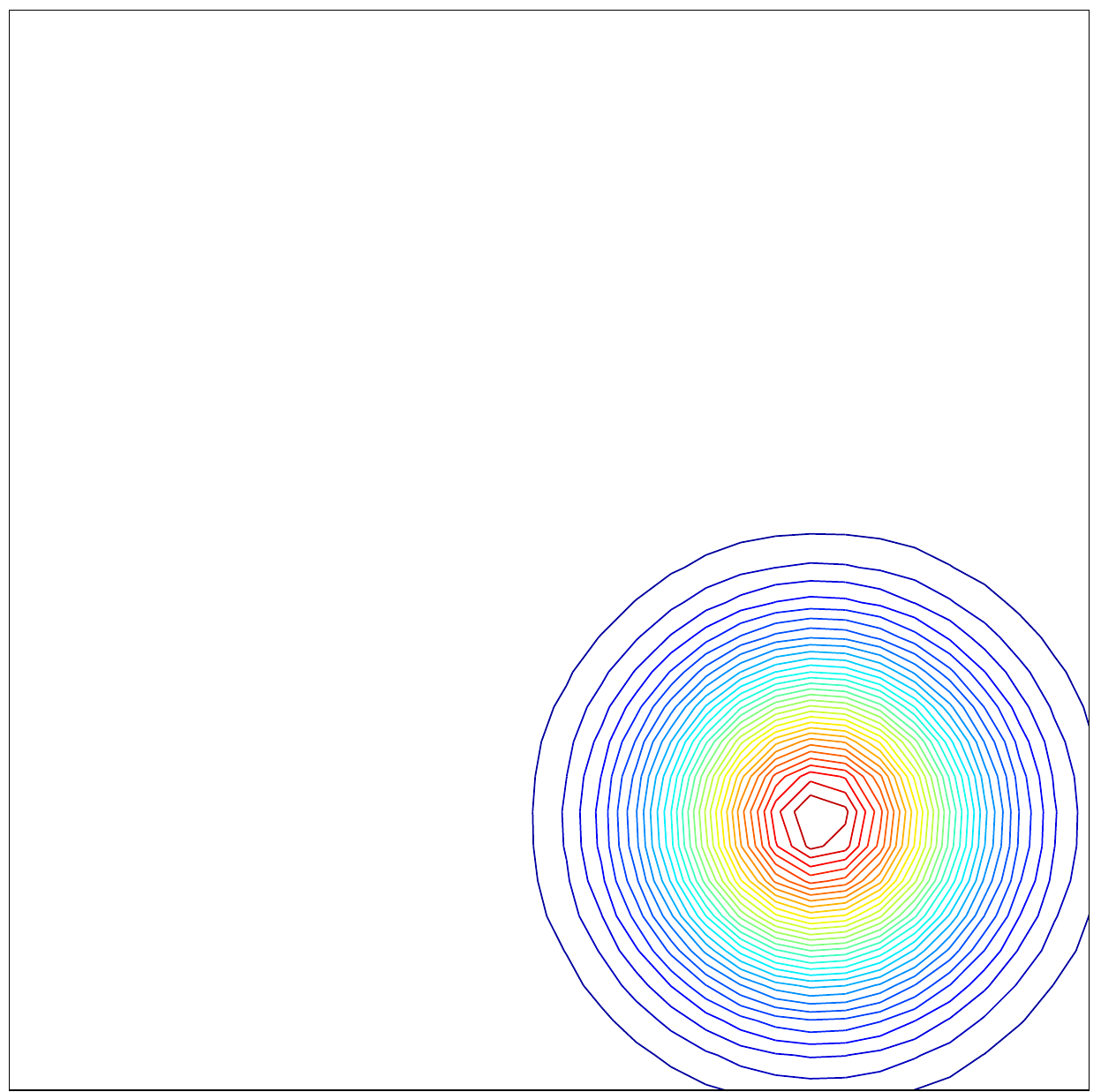}\\
&$t=0$&$t=1/8$&
$t=1/4$&$t=3/8$&
$t=1/2$&$t=5/8$&
$t=3/4$&$t=7/8$&$t=1$\vspace{-0.2cm}
\end{tabular}
\caption{\label{fig:generalized_bump} 
Display of the level sets of $\iter{f}(\cdot,t)$ for several value of $t$ and $\be$ (note that for $t=0$ and $t=1$, this corresponds to $f^0$ and $f^1$).}
\end{center}
\end{figure}

Figure~\ref{fig:evol_bump_beta} shows for $\beta = 1/2$ and $\be = 3/4$ the evolution of the cost function with the iterations index $\ell$, together with the convergence of the estimate to the reference solution $(m^\star,f^\star)$ (obtained after $10^5$ iterations of the PD algorithm). We can observe that the behavior of the process with $\beta \in ]0,1[$ is different than the one observed for $\beta=1$ in Figure \ref{fig:comp_bump}. Indeed, we observe a faster convergence of the functional value, which is consistent with the fact that $\Jfunc_\bet$ becomes more and more strongly convex as $\be$ approaches $1/2$ (see \eqref{hessian}). The oscillations come from the Newton's descent that only approximate the computation of $\prox_{\ga J_\bet}$.

\renewcommand{\sidecap}[1]{ {\begin{sideways}\parbox{4cm}{\centering #1}\end{sideways}} }

\begin{figure}[!ht]
\begin{center}
\begin{tabular}{@{}l@{}c@{}c@{}c@{}}
\sidecap{$\beta=0.5$ } &
\includegraphics[trim=30 10 40 20,clip,width=0.31\textwidth]{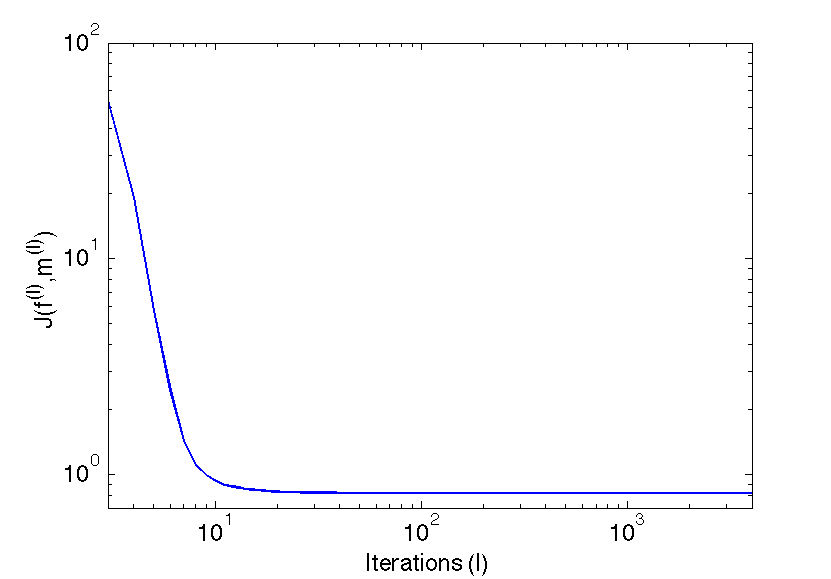}&
\includegraphics[trim=30 10 40 20,clip,width=0.31\textwidth]{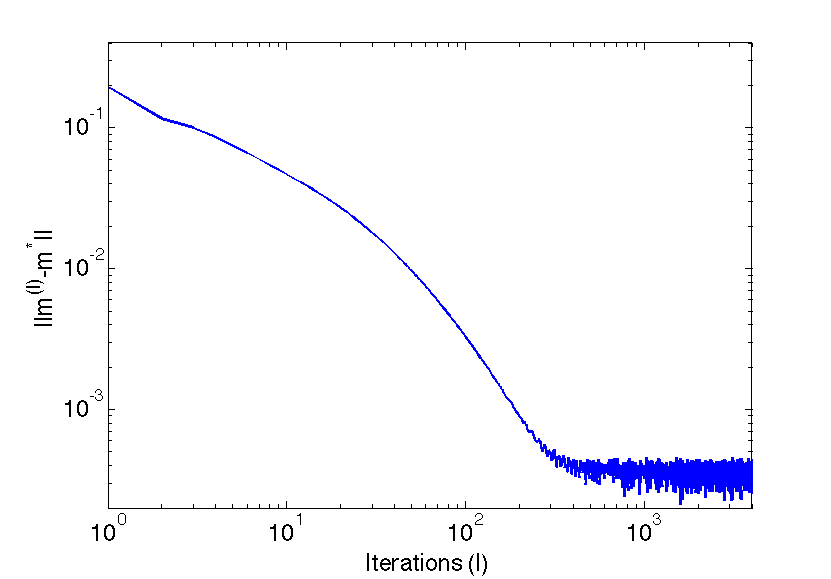}&
\includegraphics[trim=30 10 40 20,clip,width=0.31\textwidth]{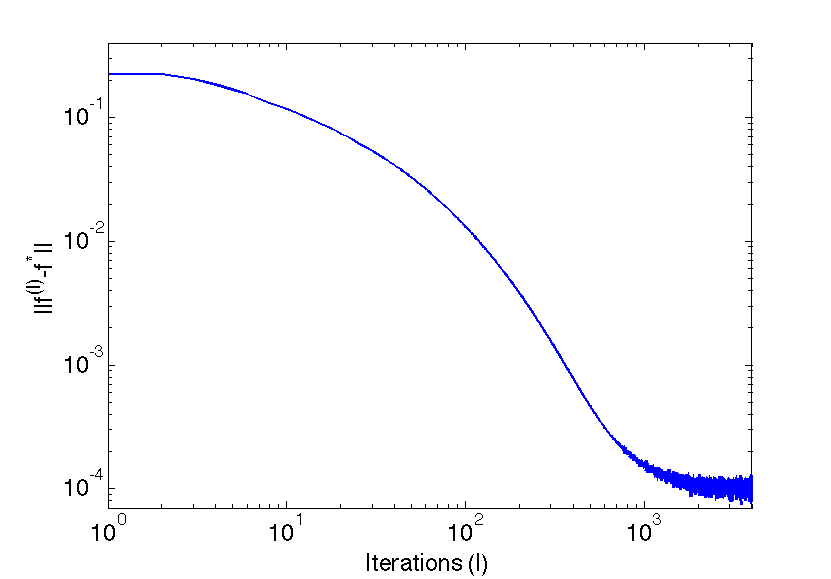}\\
\sidecap{$\beta=3/4$ } &
\includegraphics[trim=30 10 40 20,clip,width=0.31\textwidth]{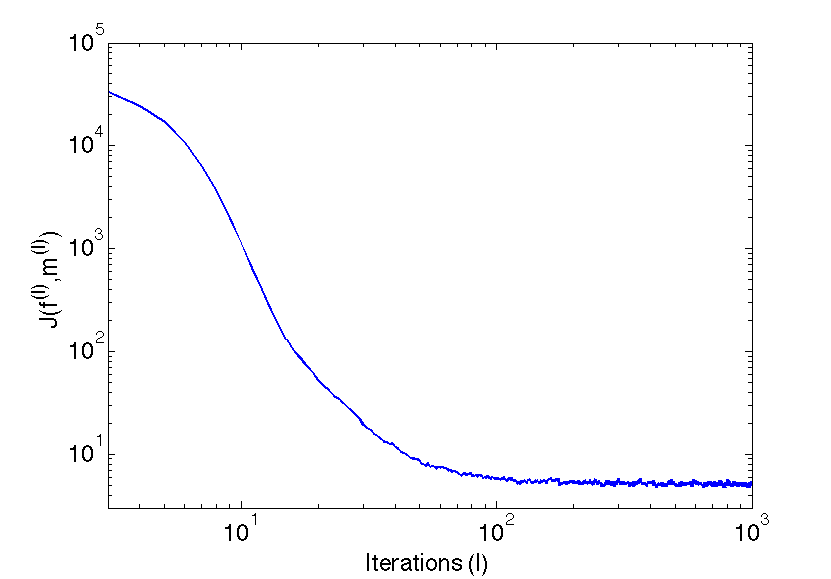}&
\includegraphics[trim=30 10 40 20,clip,width=0.31\textwidth]{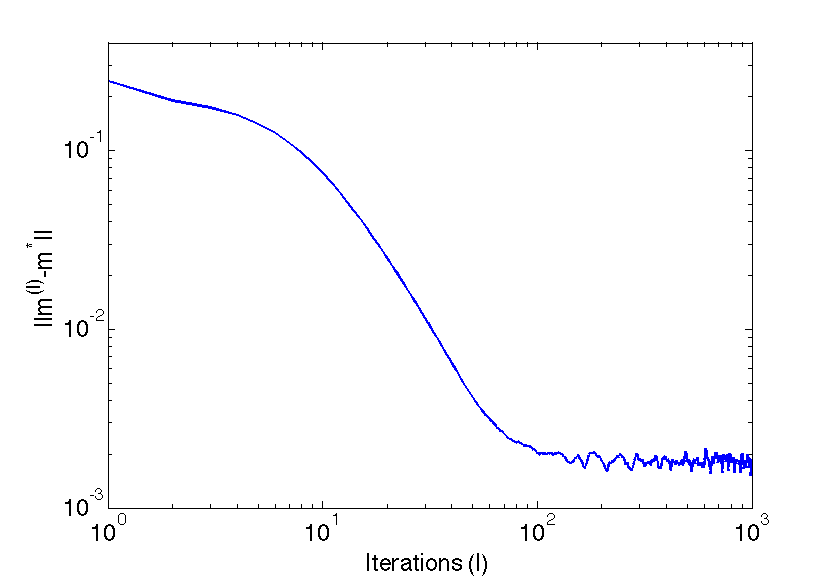}&
\includegraphics[trim=30 10 40 20,clip,width=0.31\textwidth]{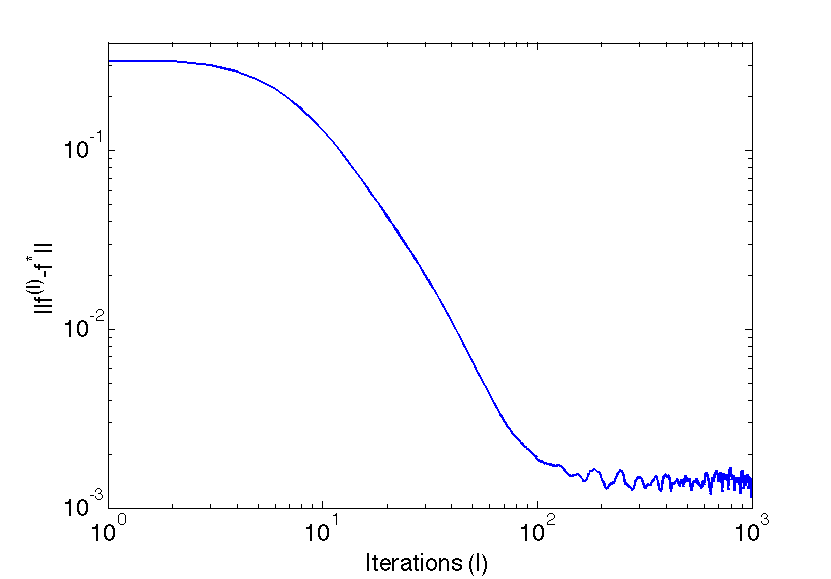}\vspace{0.1cm}\\
& $\Jfunc_\bet( \iter{m},\iter{f} )$ &  $\norm{m^\star-\iter{m}}$ & 
$\norm{f^\star-\iter{f}}$ 
\end{tabular}
\end{center}
\caption{\label{fig:evol_bump_beta} At each iteration $\ell$, we plot the value of the cost function $\Jfunc_\bet(\iter{m},\iter{f})$ and the distance between the reference solution $(m^\star,f^\star)$ and the estimation $(\iter{m},\iter{f})$. The first (resp. second) row presents the result with $\beta=1/2$ (resp. $\bet=3/4$).}
\end{figure}

As a second example, we show in Figure~\ref{fig:generalized_MK} the different morphings obtained between pictures of Gaspard Monge and Leonid Kantorovich. The grayscale representation scales linearly between black (value of 0) and white (value of 1)  and the dimensions are $N+1=75$, $M+1=100$, $P+1=60$, $M$ being the number of discrete points in the second spatial dimension.

\newcommand{\myfigMonge}[1]{\includegraphics[width=.13\linewidth]{images/MK/MK_beta#1}}

\renewcommand{\sidecap}[1]{ {\begin{sideways}\parbox{1.9cm}{\centering #1}\end{sideways}} }

\begin{figure}[!ht]
\begin{center}
\begin{tabular}{@{}c@{\hspace{1mm}}c@{\hspace{1mm}}c@{\hspace{1mm}}c@{\hspace{1mm}}c@{\hspace{1mm}}c@{\hspace{1mm}}c@{\hspace{1mm}}c@{}}
\sidecap{$\beta=0$ } &
\myfigMonge{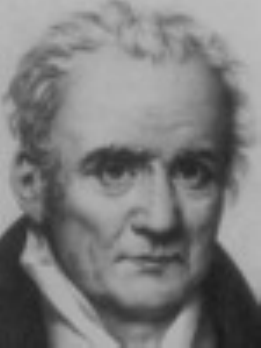}&
\myfigMonge{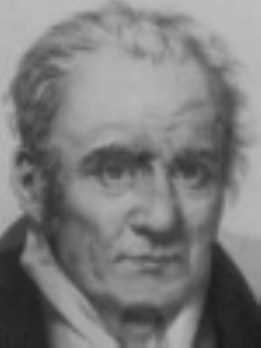}&
\myfigMonge{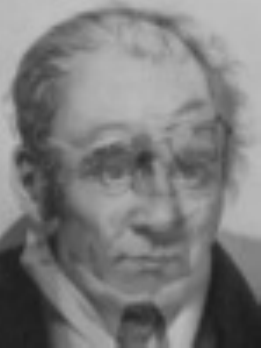}&
\myfigMonge{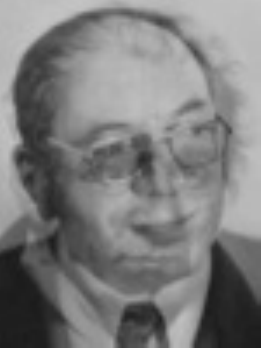}&
\myfigMonge{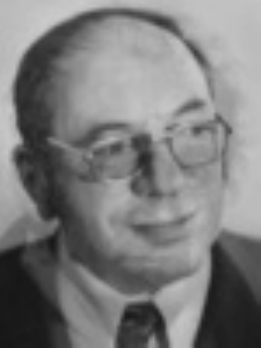}&
\myfigMonge{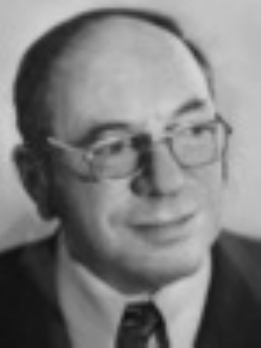}&
\myfigMonge{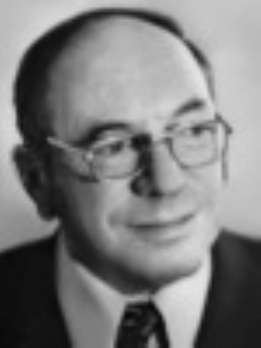}\\
\sidecap{$\beta=0.25$ } &
\myfigMonge{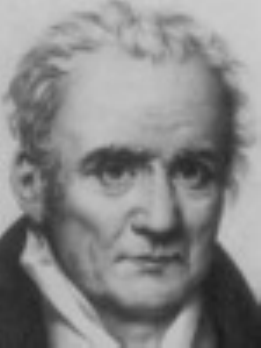}&
\myfigMonge{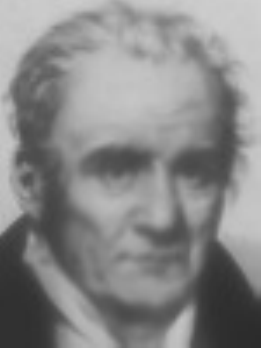}&
\myfigMonge{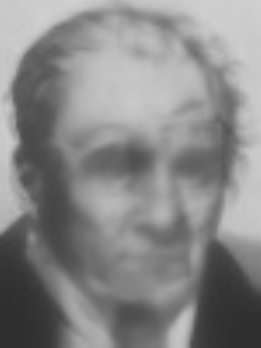}&
\myfigMonge{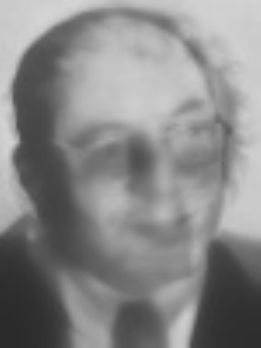}&
\myfigMonge{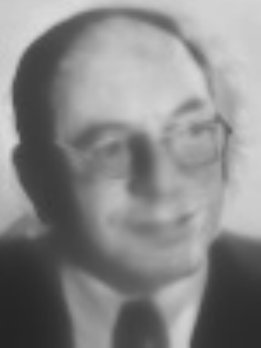}&
\myfigMonge{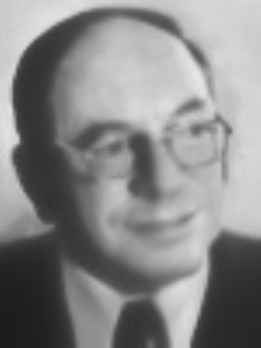}&
\myfigMonge{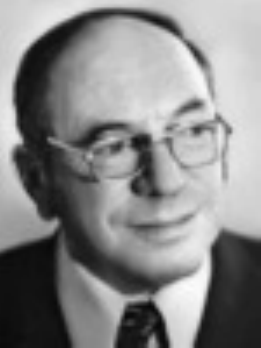}\\
\sidecap{$\beta=0.5$ } &
\myfigMonge{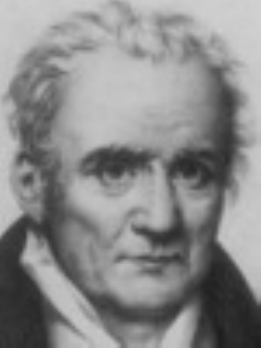}&
\myfigMonge{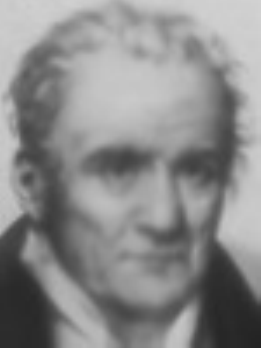}&
\myfigMonge{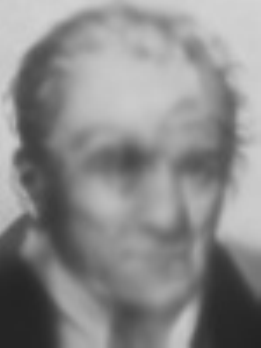}&
\myfigMonge{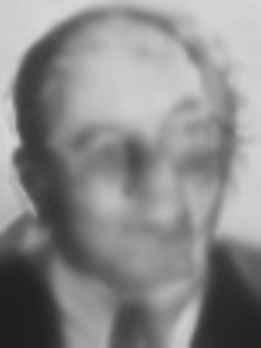}&
\myfigMonge{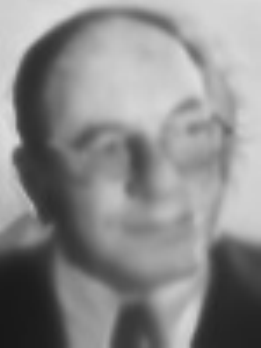}&
\myfigMonge{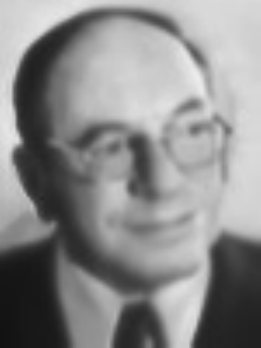}&
\myfigMonge{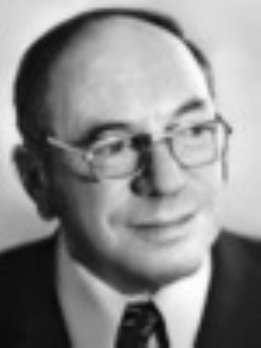}\\
\sidecap{$\beta=3/4$ } &
\myfigMonge{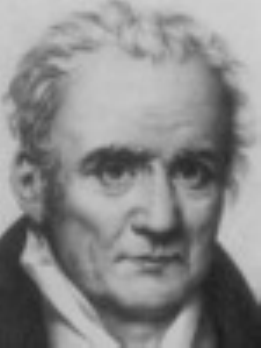}&
\myfigMonge{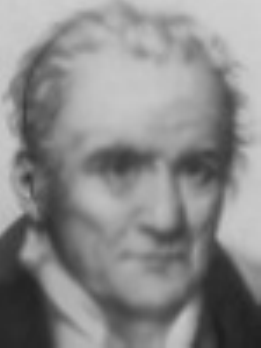}&
\myfigMonge{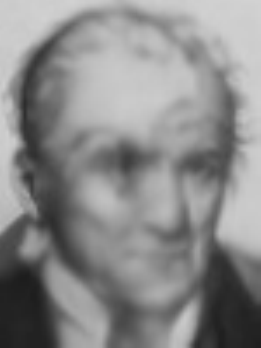}&
\myfigMonge{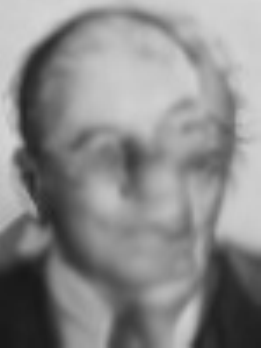}&
\myfigMonge{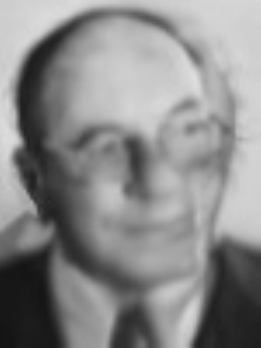}&
\myfigMonge{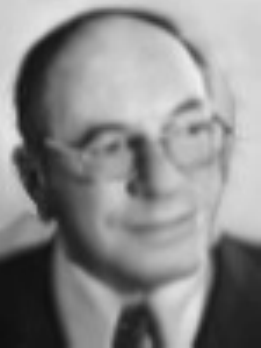}&
\myfigMonge{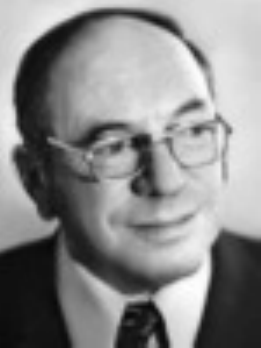}\\
\sidecap{$\beta=1$ } &
\myfigMonge{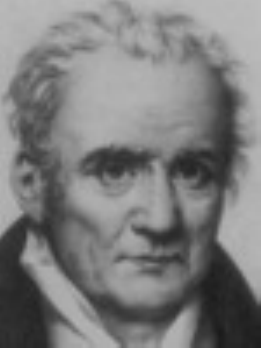}&
\myfigMonge{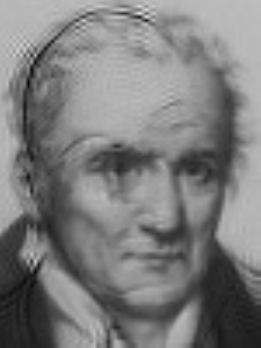}&
\myfigMonge{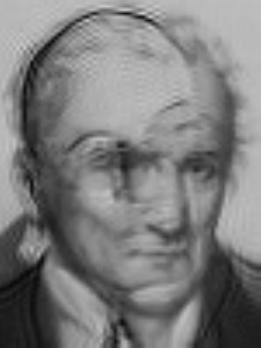}&
\myfigMonge{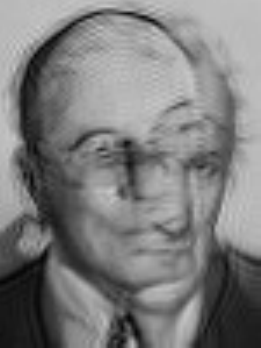}&
\myfigMonge{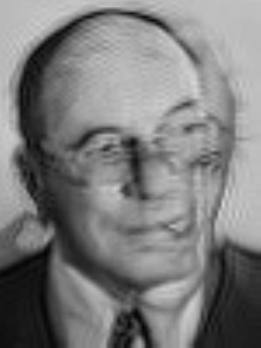}&
\myfigMonge{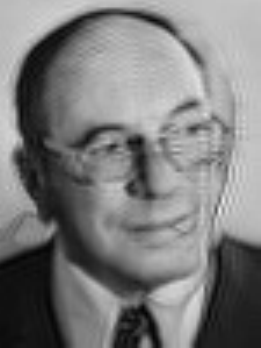}&
\myfigMonge{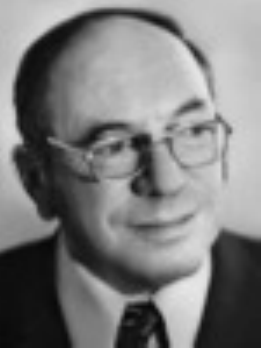}\\
&$t=0$ & $t=1/6$ &
$t=1/3$ & $t=1/2$ &
$t=2/3$ & $t=5/6$ & $t=1$
\end{tabular}
\end{center}
\caption{\label{fig:generalized_MK} 
Evolution of $f^\star(\cdot,t)$ for several value of $t$ and $\be$.
 The first and last columns represent the data $f^0$ and $f^1$. The intermediate ones present the reference solution $f^\star(t)$ for successive times $t=i/6$, $i=1\cdots 5$. 
 Each line illustrates $f^\star$ for different values $\beta=j/4$, $j=0\cdots 4$ of the generalized cost function. }
\end{figure}

\subsection{Riemannian Transportation}
\label{subsec-riemanian-examples}

We investigate in this section the approximation of a displacement interpolation for a ground cost being the squared geodesic distance on a Riemannian manifold. This is achieved by solving~\eqref{eq-optim-bb-gen} with $\be=1$ but a non-constant weight map $w$.

We exemplify this setting by considering optimal transport with obstacles, which corresponds to choosing weights $w$ that are infinity on the obstacle $\Oo \subset \RR^d \times \RR$, i.e. 
\eq{
	\foralls k \in \Gc, \quad w_k = 1 + \iota_{\Oo}(x_k,t_k) \in \{1,+\infty\}.
}
Note that the obstacles can be dynamic, i.e. the weight $w$ needs not to be constant in time. 

Figure~\ref{labyrinthe} shows a first example where $\Oo$ is a 2-D ($d=2$) static labyrinth map (the walls of the labyrinth being the obstacles and are displayed in black). We use a $50 \times 50 \times 100$ discretization grid of the space-time domain $[0,1]^3$ and the input measures $(f^0,f^1)$ are Gaussians  with standard deviations equal to $0.04$. For Gaussians with such a small variance, this example shows that the displacement interpolation is located  closely to the geodesic path between the centers of the gaussians.

\newcommand{\myfigLab}[1]{\includegraphics[width=.195\linewidth]{images/labyrinthe/bump_obstacle#1}}

\begin{figure}[!ht]
\begin{center}
\begin{tabular}{@{}c@{}c@{}c@{}c@{}c@{}}
\myfigLab{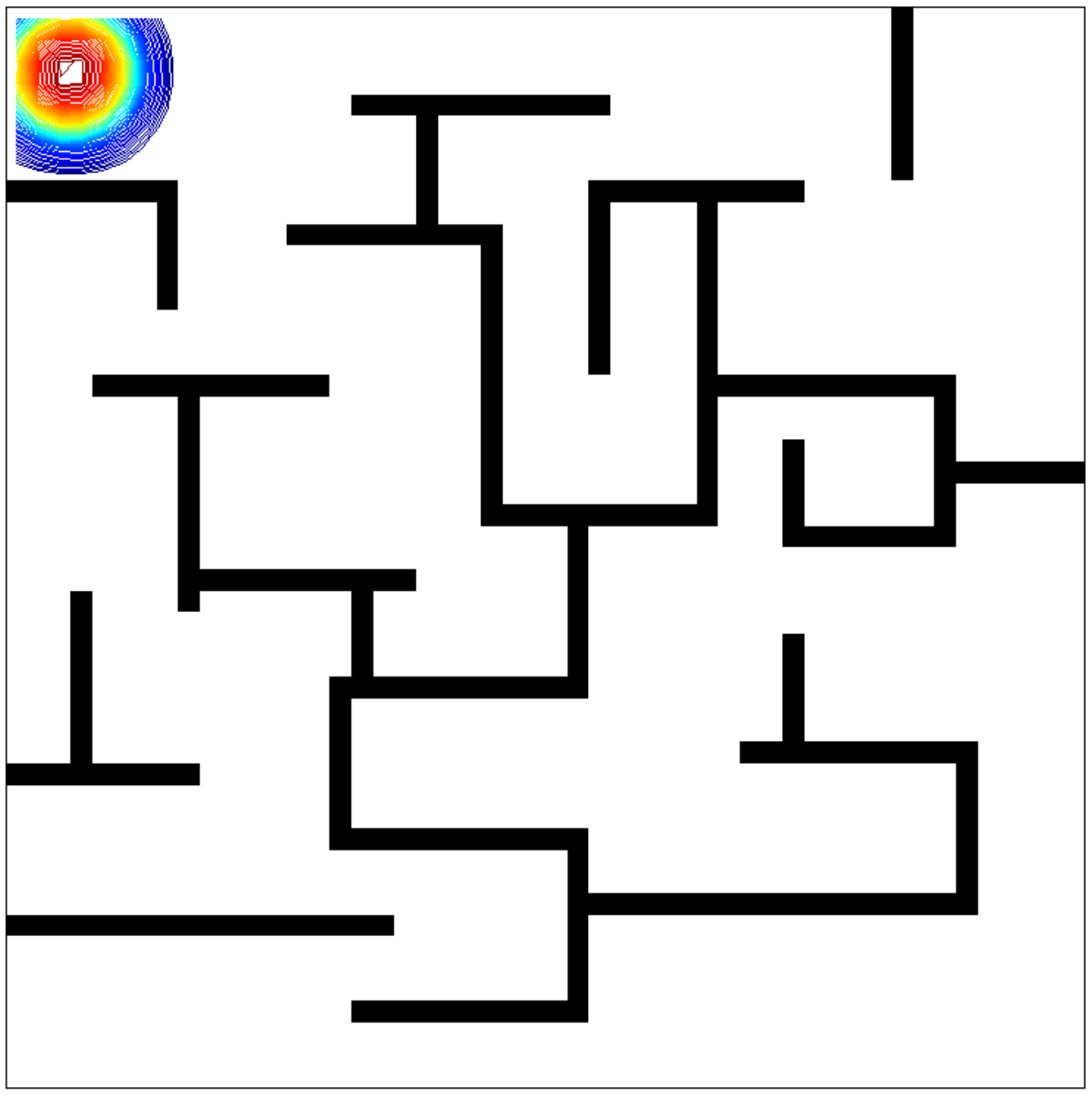}&
\myfigLab{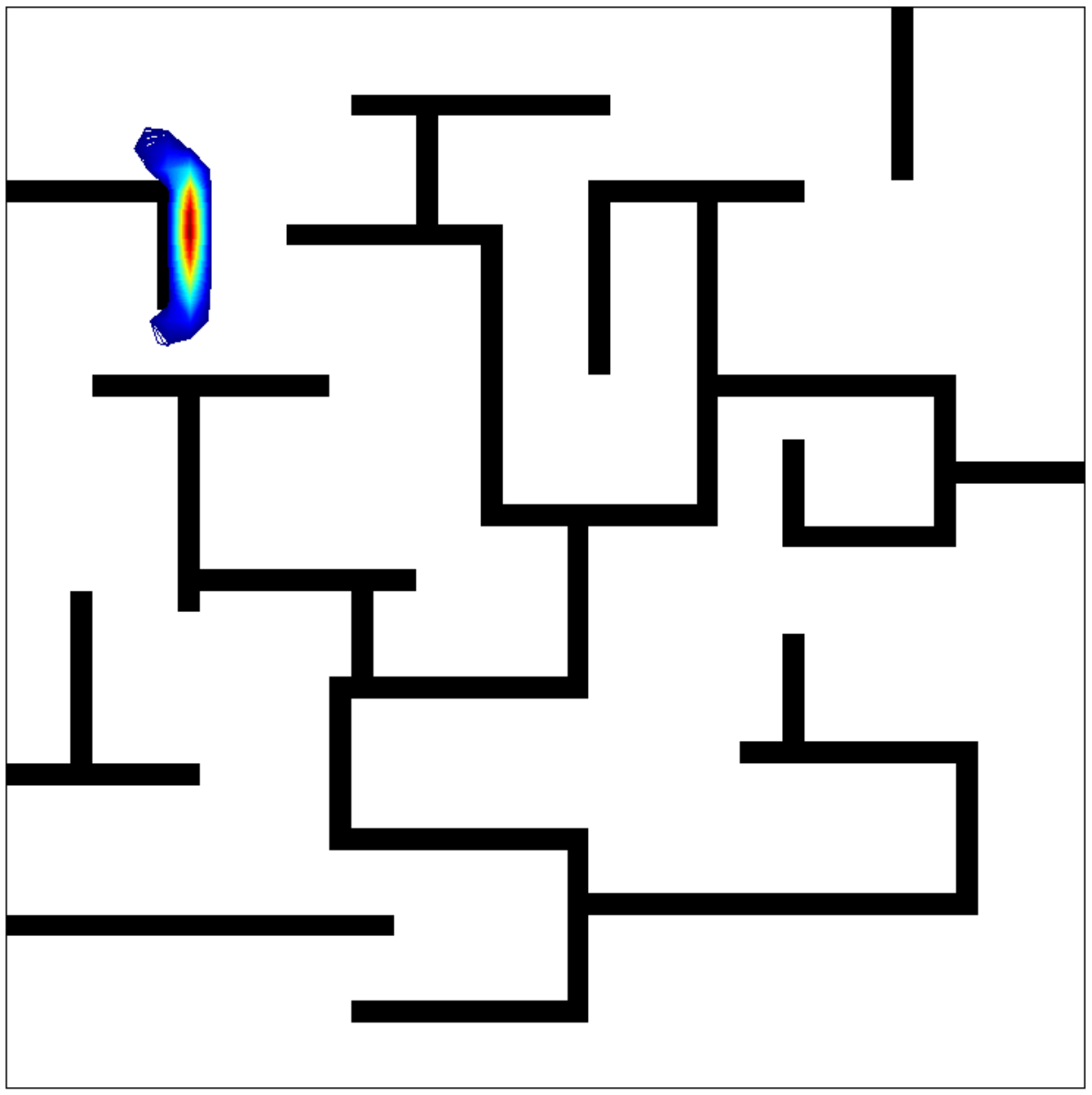}&
\myfigLab{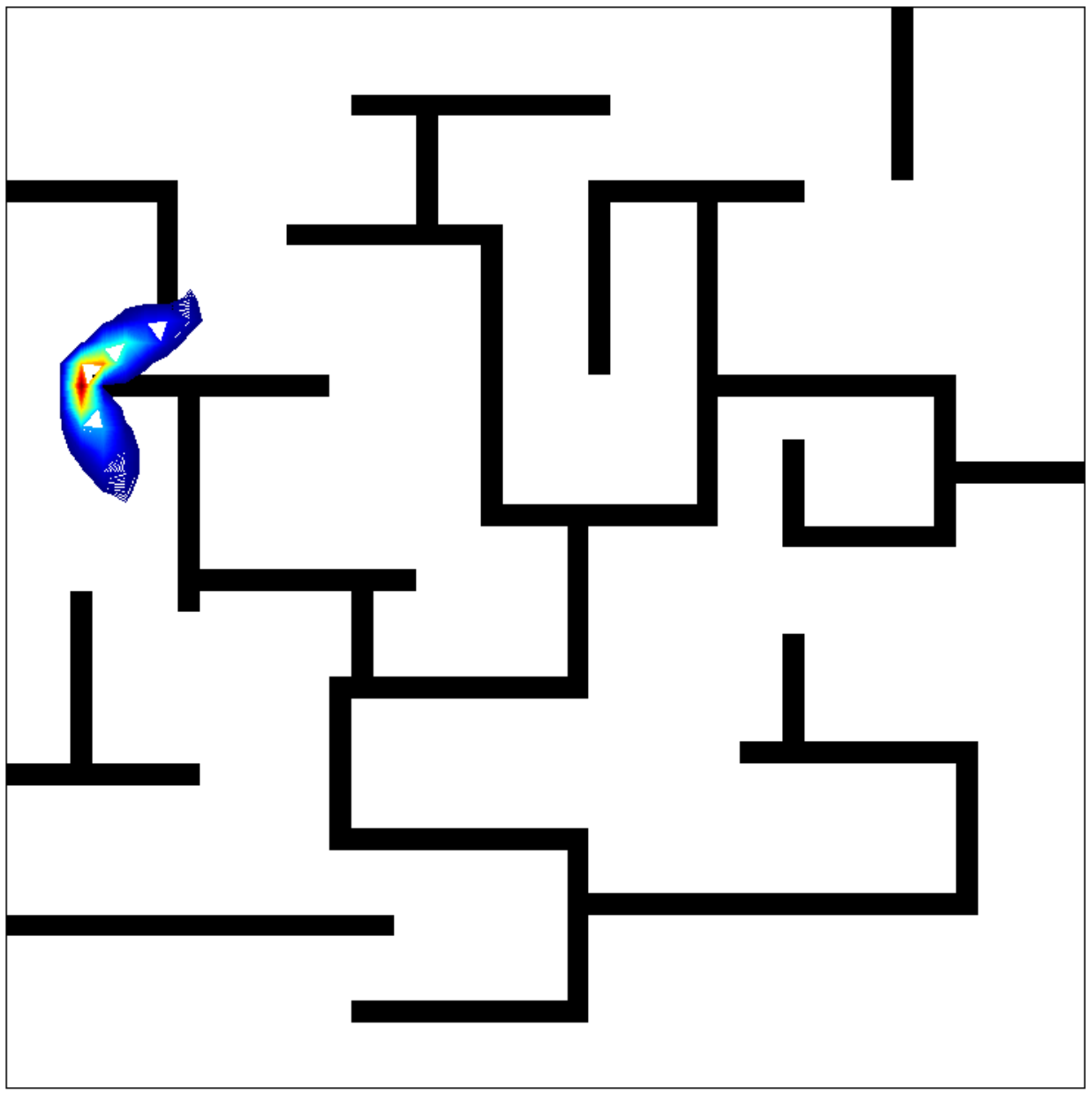}&
\myfigLab{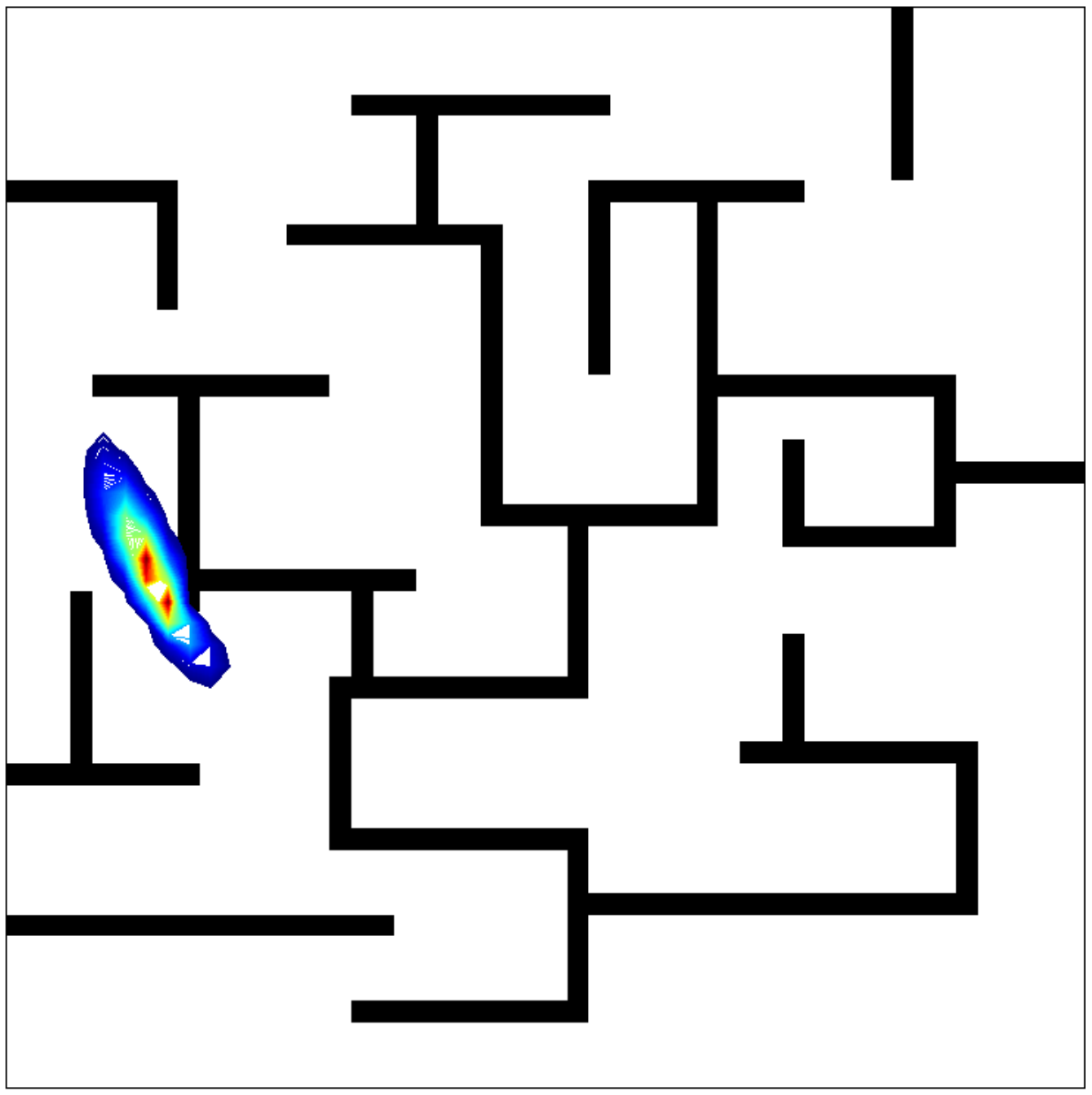}&
\myfigLab{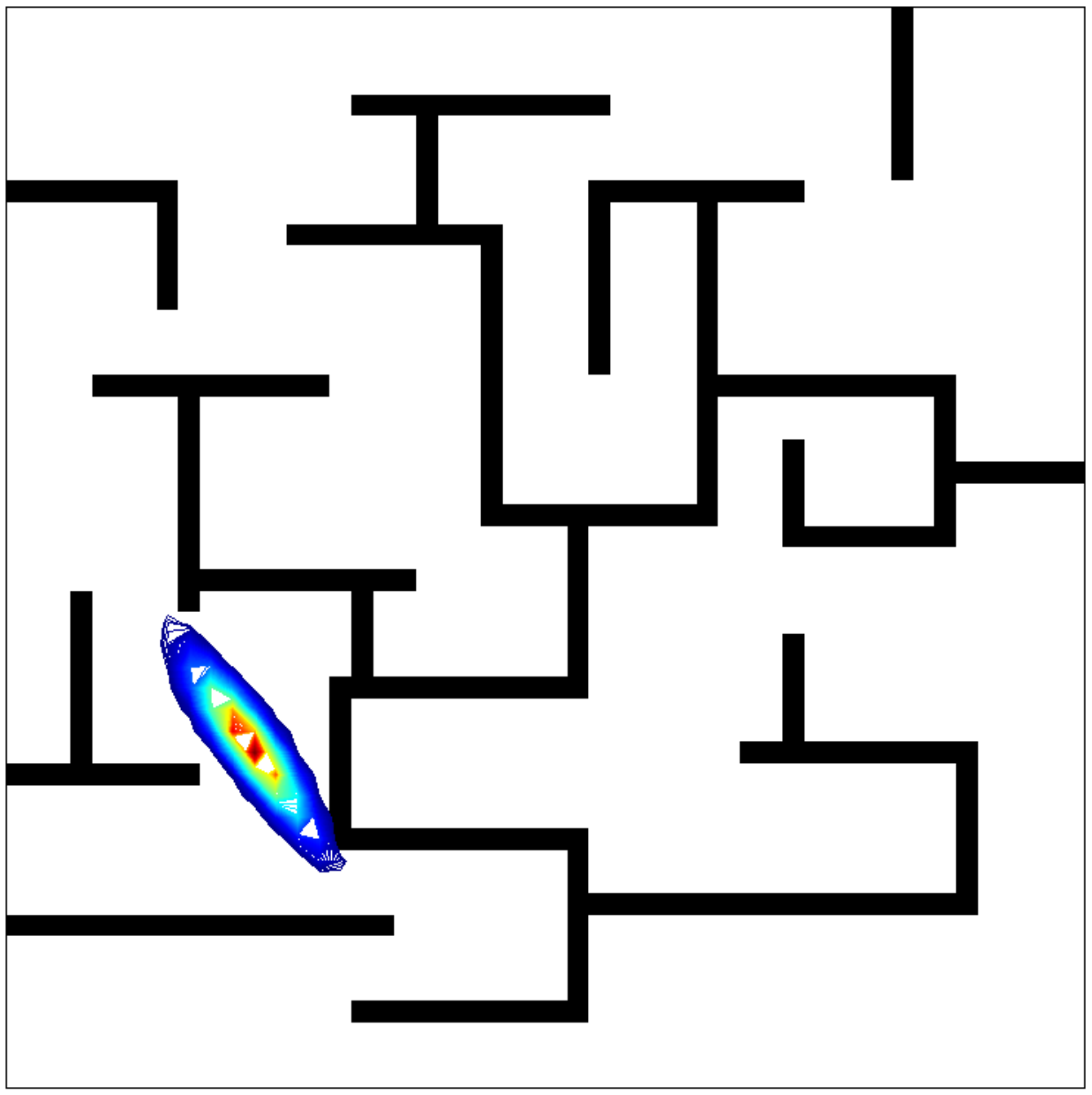}\\
$t=0$&
$t=1/9$&
$t=2/9$&
$t=1/3$&
$t=4/9$\\
\myfigLab{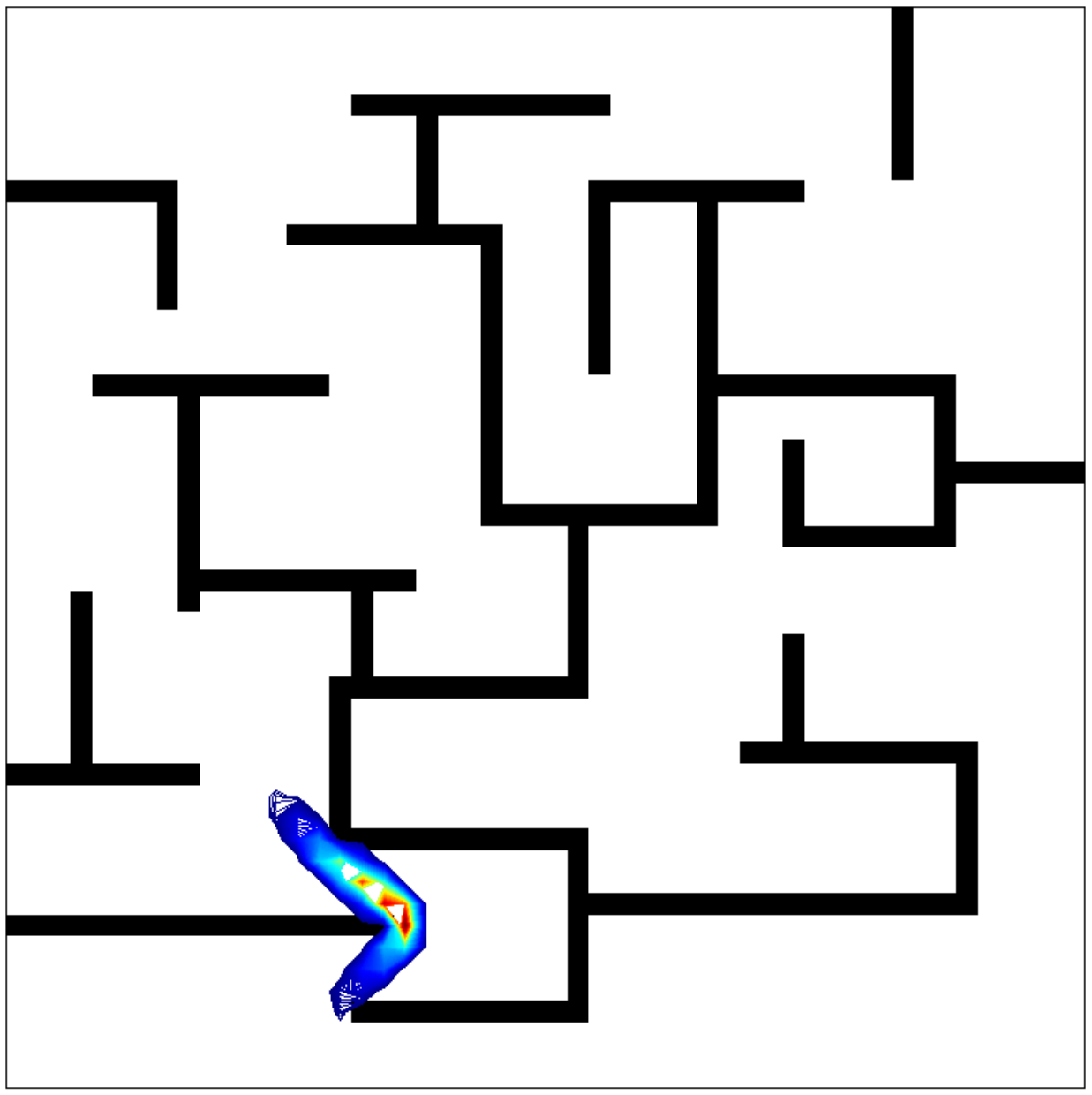}&
\myfigLab{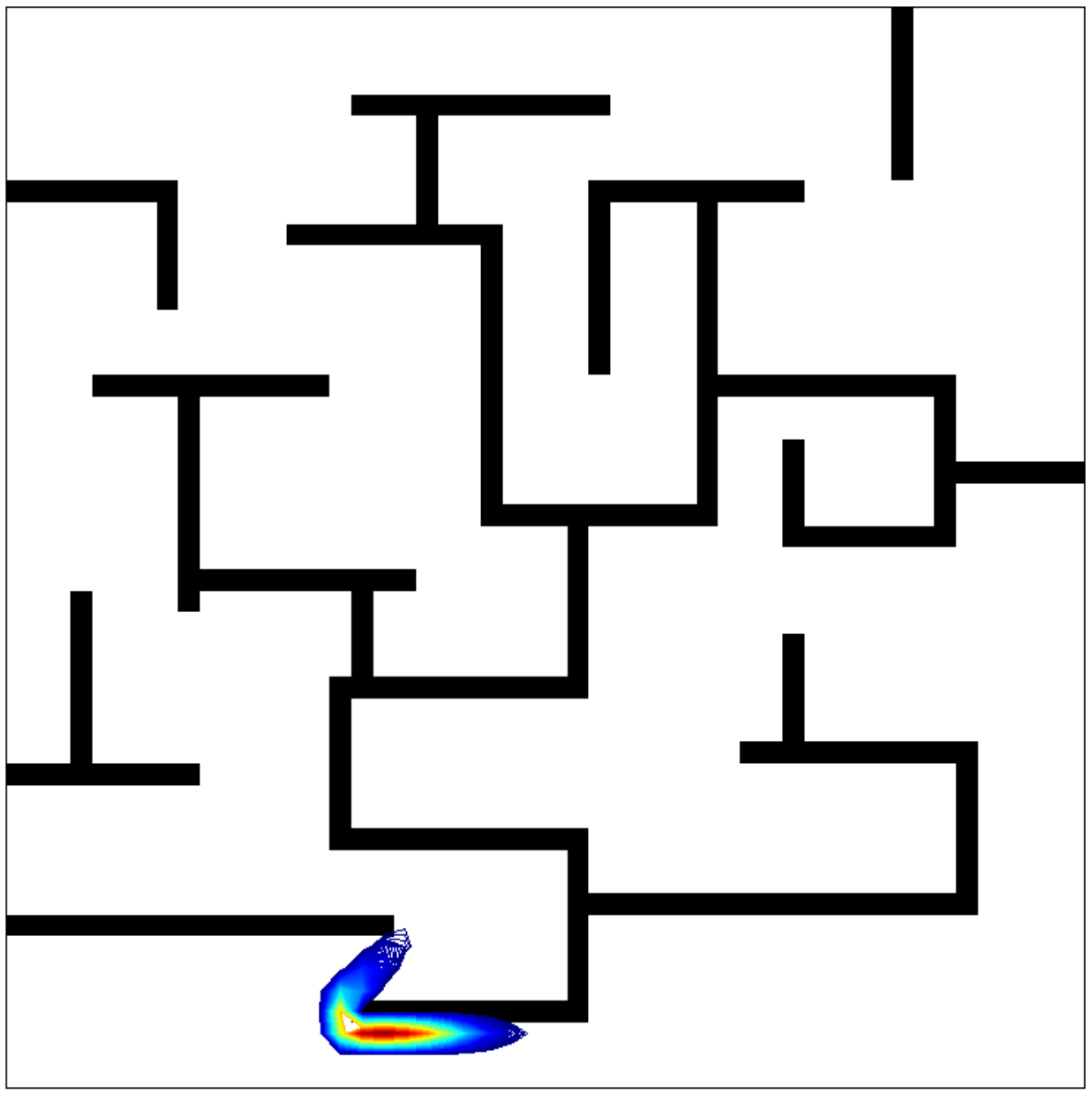}&
\myfigLab{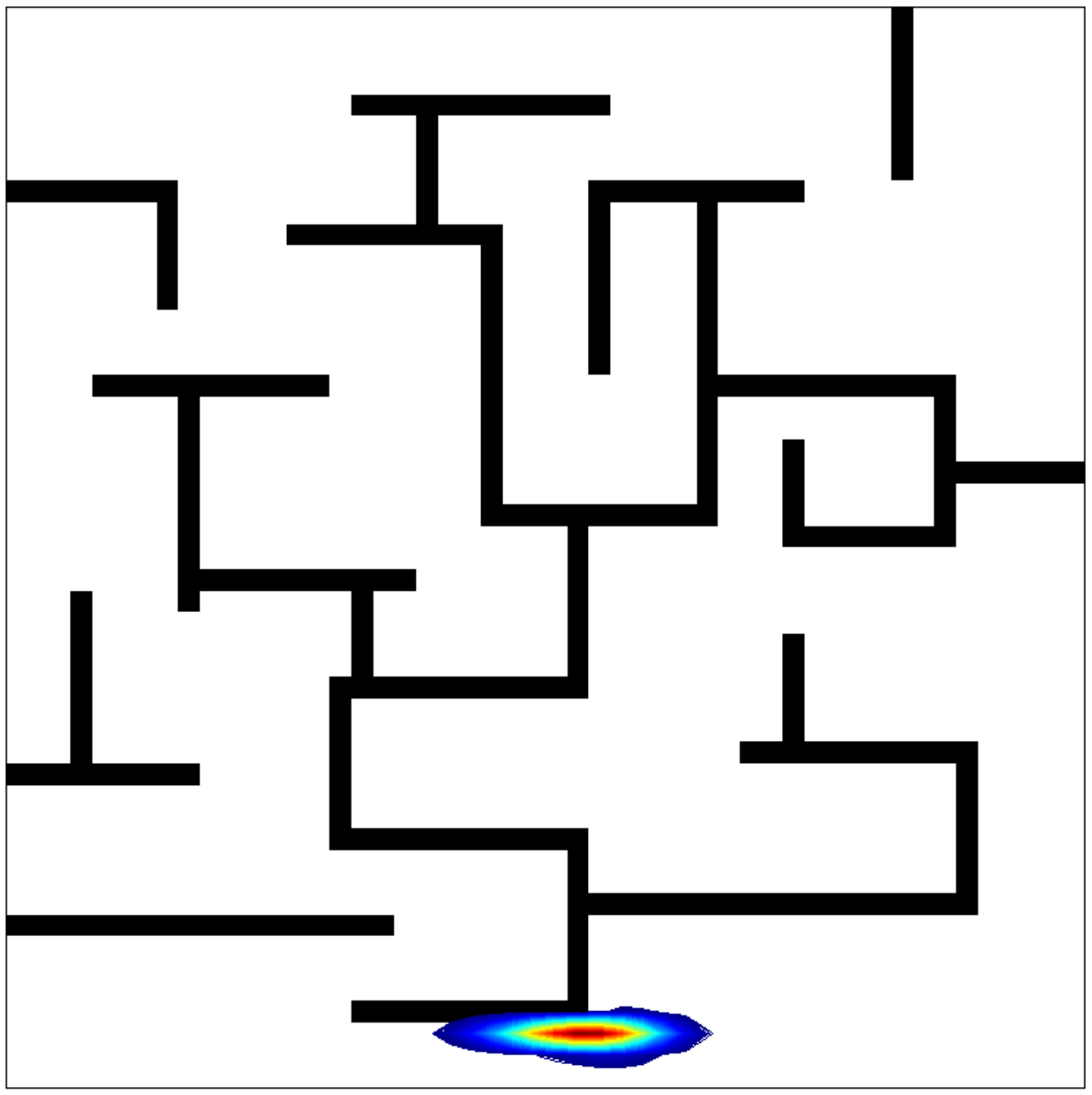}&
\myfigLab{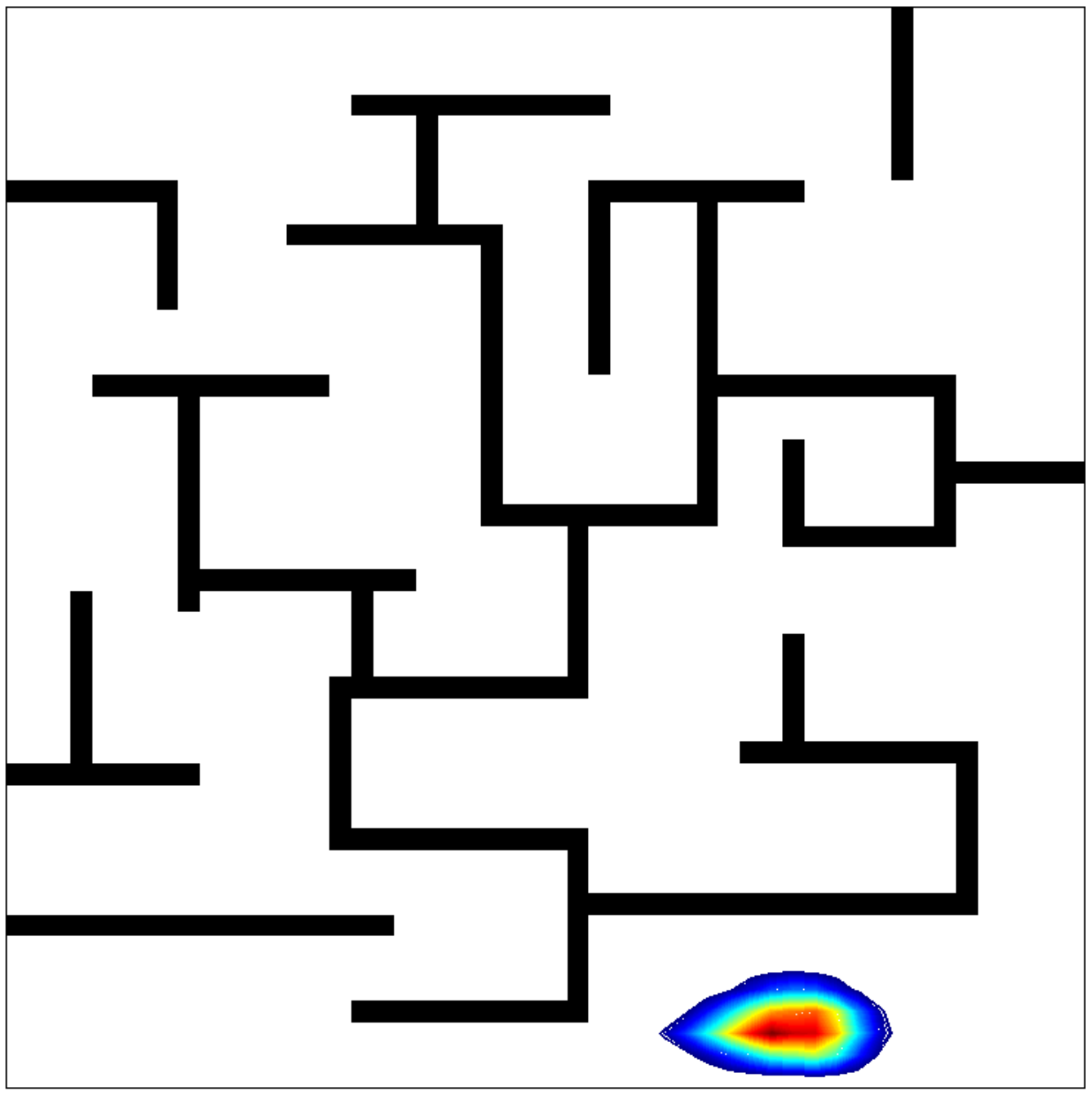}&
\myfigLab{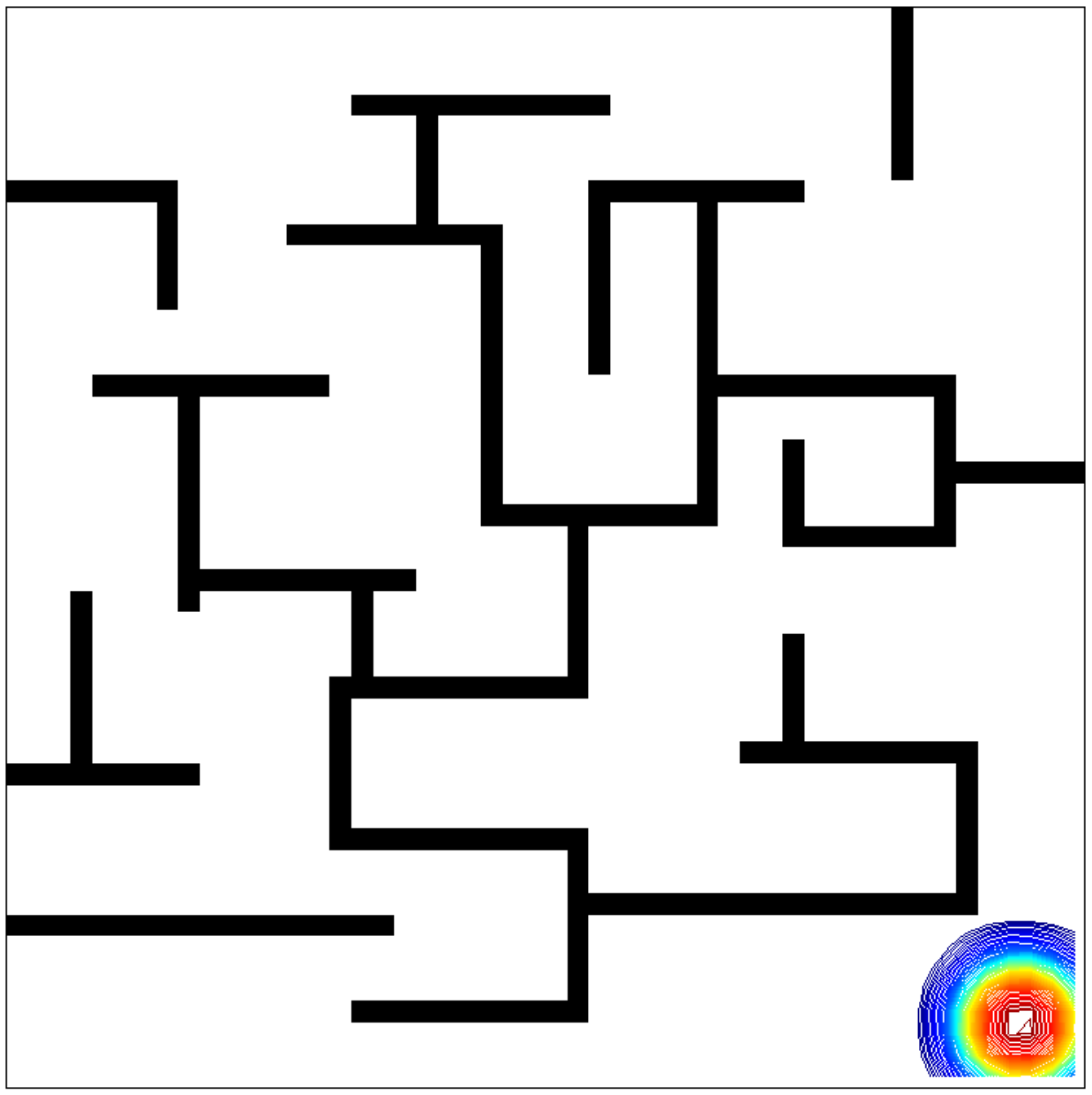}\\
$t=5/9$&
$t=2/3$&
$t=7/9$&
$t=8/9$&
$t=1$
\end{tabular}
\caption{\label{labyrinthe} 
Evolution of $f^\star(\cdot,t)$ for several values of $t$, using a Riemannian manifold with weights $w_k$ (constant in time) restricting the densities to lie within a 2-D static labyrinth map.  }
\end{center}
\vspace{3mm}
\end{figure}

Figure~\ref{labyrinthe2} shows a more complicated setting that includes a labyrinth with moving walls: a wall appears at time $t=1/4$ and another one disappears at time $1/2$. The difference with respect to the previous example is the fact that $w$ is now time dependent. This simple modification has a strong impact on the displacement interpolation. Indeed, the speed of propagation of the mean of the density is not constant anymore since the density measure is confined in a small area surrounded by walls for $t \in [1/4,1/2]$.

\begin{figure}[!ht]
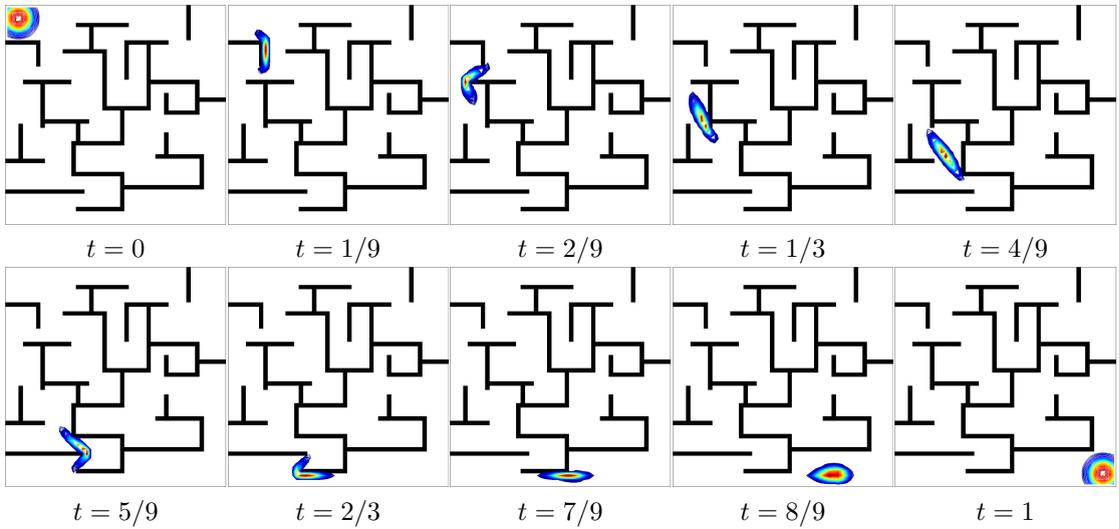

\begin{center}
\begin{tabular}{@{}c@{}c@{}c@{}c@{}c@{}}
\myfigLab{5_iso_001}&
\myfigLab{5_iso_012}&
\myfigLab{5_iso_023}&
\myfigLab{5_iso_034}&
\myfigLab{5_iso_045}\\
$t=0$&
$t=1/9$&
$t=2/9$&
$t=1/3$&
$t=4/9$\\
\myfigLab{5_iso_056}&
\myfigLab{5_iso_067}&
\myfigLab{5_iso_078}&
\myfigLab{5_iso_089}&
\myfigLab{5_iso_101}\\
$t=5/9$&
$t=2/3$&
$t=7/9$&
$t=8/9$&
$t=1$
\end{tabular}
\caption{\label{labyrinthe2} Evolution of $f^\star(\cdot,t)$ for several values of $t$, using a Riemannian  manifold with weights $w_k$ (evolving in time) restricting the densities to lie within a 2-D dynamic labyrinth map (i.e. with moving walls). }
\end{center}
\vspace{3mm}
\end{figure}

As a last example, we present in Figure~\ref{bretagne} an interpolation result in the context of  oceanography in the presence of coast. We here consider Gaussian mixture data in order to simulate the Sea Surface Temperature that can be observed from satellite. In order to model the influence of the sea ground height, we here considered weights $w$ varying w.r.t the distance to the coast. Denoting as $\Oo$ the area representing the complementary of the sea, we define
\eq{
	\foralls k \in \Gc, \quad w_k = 1 + d(x_k,\partial \Oo) + \iota_{\Oo} \in \{1,+\infty\},
} 
where $d(x,\partial \Oo)$ is the Euclidean distance between a pixel location $x$ and the boundary of $\Oo$. The estimation of such interpolations are of main interest in geophysic forecasting applications where the variables of numerical models are calibrated using external image observations (such as the Sea Surface Temperature). Data assimilation methods used in geophysics look for the best compromise between a model and the observations (see for instance~\cite{Blum2009}) and making use of optimal transportation methods in this context is an open research problem.

\newcommand{\myfigBret}[1]{\includegraphics[width=.195\linewidth]{images/bretagne/bretagne_#1}}

\begin{figure}[!ht]
\begin{center}
\begin{tabular}{@{}c@{\hspace{1mm}}c@{\hspace{1mm}}c@{\hspace{1mm}}c@{\hspace{1mm}}c@{}}
\myfigBret{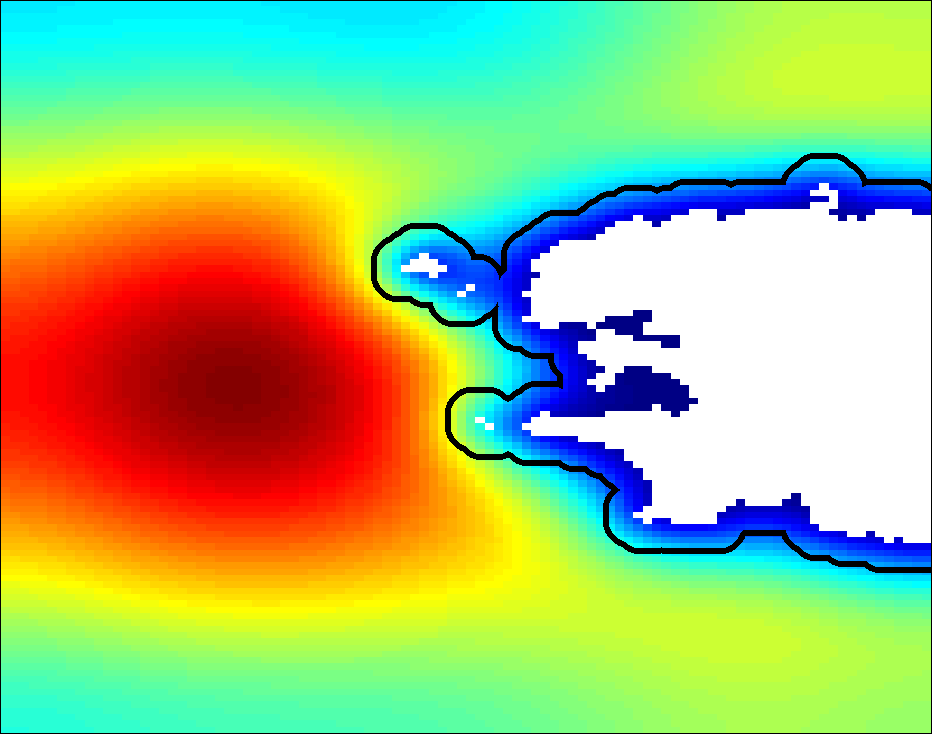}&
\myfigBret{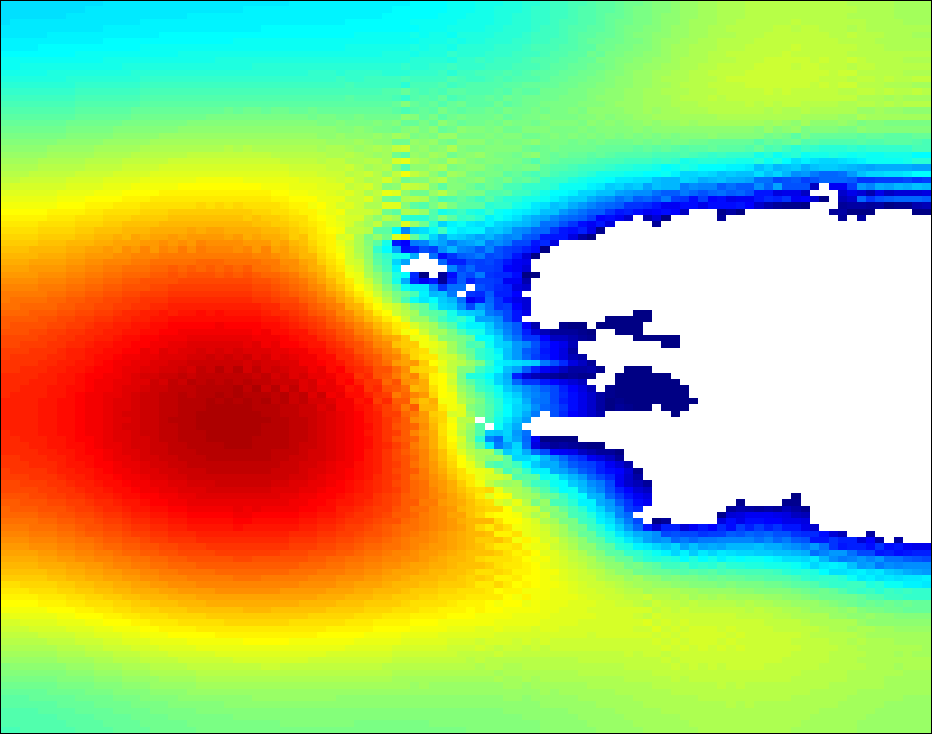}&
\myfigBret{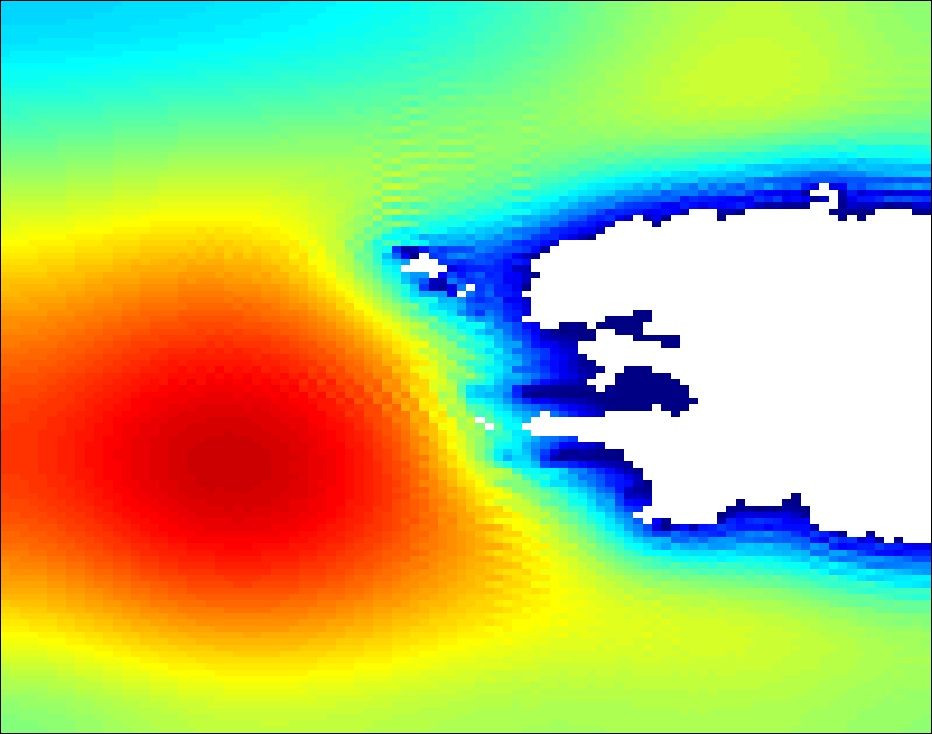}&
\myfigBret{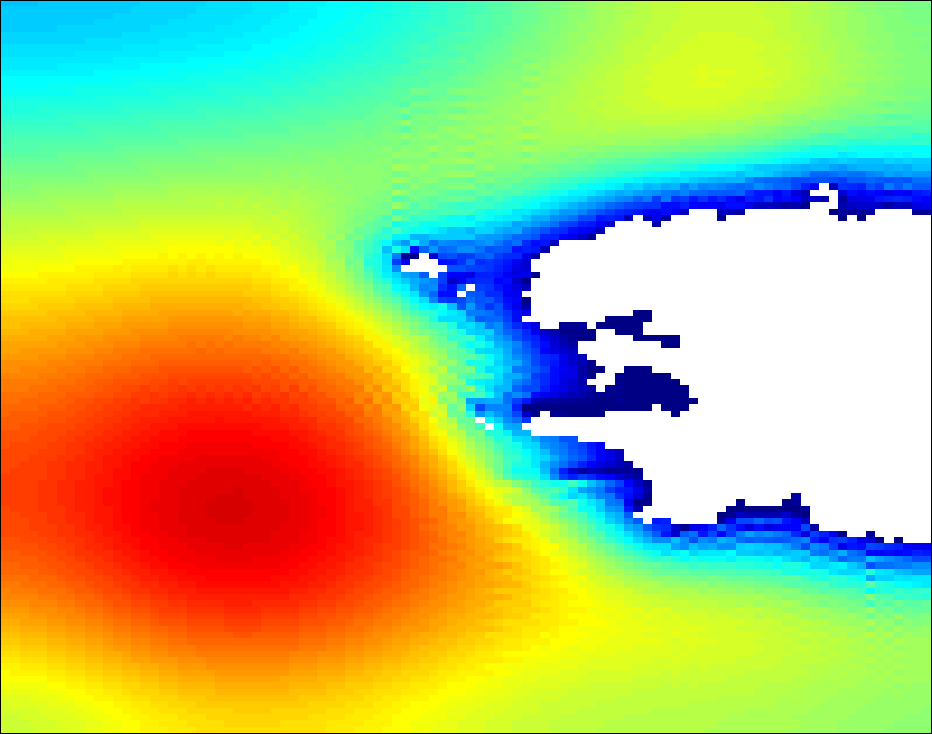}&
\myfigBret{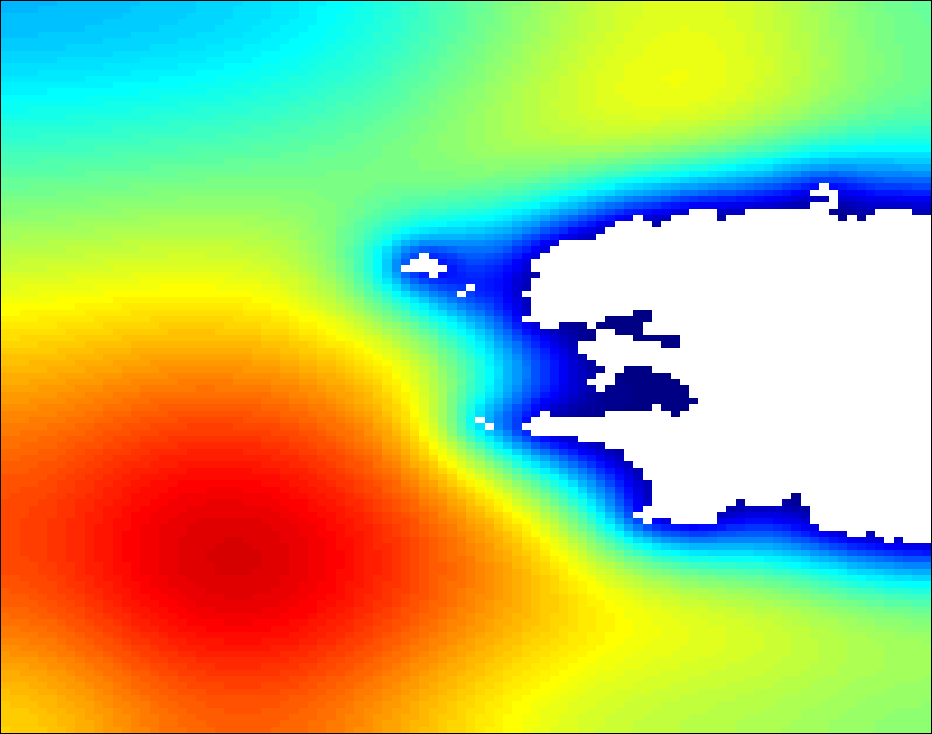}\\
$t=0$&
$t=1/4$&
$t=1/2$&
$t=3/4$&
$t=1$
\end{tabular}
\caption{\label{bretagne} Evolution of $f^\star(\cdot,t)$ for several values of $t$, using a Riemannian  manifold with weights $w_k$ defined from the distance to the boundary of the sea domain (its frontier is displayed in black in the first figure). }
\end{center}
\vspace{3mm}
\end{figure}

\section*{Conclusion}

In this article, we have shown how proximal splitting schemes offer an elegant and unifying framework to describe computational methods to solve the dynamical optimal transport with an Eulerian discretization. This allowed use to extend the original method of Benamou and Brenier in several directions, most notably the use of staggered grid discretization and the introduction of generalized, spatially variant, cost functions. 

\section*{Acknowledgment}

We would like to thank to Jalal Fadili for his detailed explanations of the connexion between the ADMM and DR algorithms. Nicolas Papadakis would like to acknowledge partial support from the LEFE program of INSU (CNRS). 
Gabriel Peyr\'e acknowledges support from the European Research Council (ERC project SIGMA-Vision). This work was also partially funded by the French Agence Nationale de la Recherche (ANR, Project TOMMI) under reference ANR-11-BS01-014-01.


\bibliographystyle{plain}  
\bibliography{refs}

\end{document}